\documentclass[12pt,a4paper]{article}
\usepackage[utf8]{inputenc}
\usepackage[english]{babel}
\usepackage{xcolor}

\usepackage{amssymb,mathtools,amsmath,amsthm,stmaryrd}
\usepackage[all,cmtip,2cell]{xy} 
\xyoption{rotate}

\usepackage[draft=false, hidelinks, linkcolor = blue]{hyperref}
\usepackage{cleveref}
\usepackage{tensor}
\usepackage{bbm}
\usepackage{url}
\usepackage{bookmark}
\usepackage{footmisc}

\usepackage{graphicx}
\usepackage{lmodern}
\usepackage{enumitem}
\usepackage{array}

\usepackage[left=2cm,right=2cm,top=2cm,bottom=2cm]{geometry}
\def\defthm#1#2#3#4{
  \newtheorem{#1}[theorem]{#3}
  \newtheorem*{#1*}{#3}
  \newtheorem{#2}[theorem]{#4}
  \newtheorem*{#2*}{#4}
  \crefname{#1}{#3}{#4}
  \crefname{#2}{#4}{#4}  
}

\newtheoremstyle{mythm}%
{10pt}
{}
{\itshape}
{}
{\bf}
{.}
{.5em}
{}%

\newtheoremstyle{mydef}%
{10pt}
{3pt}
{}
{}
{\bf}
{.}
{.5em}
{}%

\newtheoremstyle{myrmk}%
{10pt}
{3pt}
{}
{}
{\bf}
{.}
{.5em}
{}%

\theoremstyle{mythm}
\newtheorem{theorem}{Theorem}[section]
\newtheorem*{theorem*}{Theorem}

\defthm{corollary}{corollaries}{Corollary}{Corollaries}
\defthm{lemma}{lemmata}{Lemma}{Lemmata}
\defthm{proposition}{propositions}{Proposition}{Propositions}
\defthm{axiom}{axioms}{Axiom}{Axioms}
\defthm{propdef}{props/defs}{Proposition/Definition}{Prositions/Definitions}
\defthm{defcor}{defscors}{Corollary/Definition}{Corollaries/Definitions}
\defthm{conjecture}{conjectures}{Conjecture}{Conjectures}

\theoremstyle{mydef}
\defthm{definition}{definitions}{Definition}{Definitions}

\theoremstyle{myrmk}
\defthm{acknowledgment}{acknowledgments}{Acknowledgment}{Acknowledgments}
\defthm{remark}{remarks}{Remark}{Remarks}
\defthm{motivation}{motivations}{Motivation}{Motivations}
\defthm{example}{examples}{Example}{Examples}
\defthm{question}{questions}{Question}{Questions}
\defthm{observation}{observations}{Observation}{Observations}
\defthm{claim}{claims}{Claim}{Claims}
\defthm{notation}{notations}{Notation}{Notations}
\defthm{terminology}{terminologies}{Terminology}{Terminologies}
\defthm{construction}{constructions}{Construction}{Constructions}

\newcommand{\blank}{\mbox{\hspace{3pt}\underline{\ \ }\hspace{2pt}}}

\newcommand{\pbs}{\scalebox{1.5}{\rlap{$\cdot$}$\lrcorner$}}

\newcommand{\adj}{\rotatebox[origin=c]{180}{$\vdash$}\hspace{0.1pc}}

\newcommand*{\phv}{\makebox[1.5ex]{\textbf{$\cdot$}}}

\title{The $(\infty,2)$-category of internal $(\infty,1)$-categories}
\author{Raffael Stenzel}
\renewcommand\footnotemark{}

\begin{document}
\maketitle

\begin{abstract}
We define and study the $(\infty,2)$-category $\mathbf{Cat}_{\infty}(\mathcal{C})$ of 
$(\infty,1)$-categories internal to a general $(\infty,1)$-category $\mathcal{C}$ via an associated externalization construction.

In the first part, we show various formal closure properties of $\mathbf{Cat}_{\infty}(\mathcal{C})$ regarding limits, tensors, 
cotensors and internal mapping objects under the assumption of various suitable closure properties of $\mathcal{C}$. In particular, we 
show that $\mathbf{Cat}_{\infty}(\mathcal{C})$ defines a cartesian closed full sub-$\infty$-cosmos of the $\infty$-cosmos
$\mathbf{Fun}(\mathcal{C}^{op},\mathbf{Cat}_{\infty})$ of $\mathcal{C}$-indexed $(\infty,1)$-categories under suitable assumptions on
$\mathcal{C}$. We furthermore characterize the objects of $\mathbf{Cat}_{\infty}(\mathcal{C})$ 
by means of a Yoneda lemma that expresses indexed diagrams of internal shape over $\mathcal{C}$ in terms of an
$(\infty,1)$-categorical totalization.

In the second part, we relate the general theory developed to this point to results in the model categorical literature. We show that 
every model category $\mathbb{M}$ gives rise to a ``hands-on'' $\infty$-cosmos $\mathbf{Cat}_{\infty}(\mathbb{M})$
directly by restriction of the Reedy model structure on $\mathbb{M}^{\Delta^{op}}$. We then define a corresponding right 
derived model categorical externalization functor, and use it to show that the $(\infty,1)$-categorical and the model categorical 
constructions correspond to one another whenever
$\mathbb{M}$ is a suitable model category. 
\end{abstract}

\section{Introduction}\label{secintro}

\begin{terminology}
For the sake of readability and conformity with the common conventions, the term ``$\infty$-category`` shall mean
``$(\infty,1)$-category'' throughout this paper.
\end{terminology}

\subsection{A brief recollection of internal category theory}

The theory of ordinary categories is essentially the systematic study of the 2-category $\mathbf{Cat}$ of (small) categories, functors, 
and natural transformations. That is, the study of its properties and structures.
The non-trivial 2-categorical structure allows to define and study notions such as adjunctions, monads and their 
algebras, functor categories and categories of presheaves in particular, Kan lifts and extensions as well as many others. Via the 
practice of formal category theory as pioneered by the Australian School, these notions can be abstracted so to be defined and studied 
in any suitably rich 2-category. In this spirit, we recall that ordinary category theory is category theory internal to the category
$\mathrm{Set}$ of sets. And in fact, for every base category $\mathcal{C}$ the category $\mathrm{Cat}(\mathcal{C})$ of internal 
categories in $\mathcal{C}$ 
comes equipped with a canonical structure of a 2-category $\mathbf{Cat}(\mathcal{C})$ as well. Its objects are the internal 
categories, its morphisms are the internal functors, and its 2-cells are the internal natural transformations for which a formula can be 
written down by hand (see e.g.\ \cite[Definition 7.2.1]{jacobsttbook}). For $\mathcal{C}=\mathrm{Set}$, the 2-category
$\mathbf{Cat}(\mathrm{Set})$ recovers exactly the 2-category $\mathbf{Cat}$ of (ordinary) categories as a special case. As obvious as 
this is to the contemporary category theorist, the existence of this 2-categorical structure on $\mathrm{Cat}(\mathcal{C})$ is a
non-trivial observation in as much as the base category $\mathcal{C}$ has no non-trivial 2-categorical structure to begin with. In this 
sense, this structure is not inherited from $\mathcal{C}$ but rather is created ex-nihilo by means of $\mathcal{C}$.
Furthermore, various structural properties of the base $\mathcal{C}$ -- which in conjunction enable one to think of $\mathcal{C}$ as a 
suitably rich theory or even an ambient universe of abstract sets itself -- imply various structural properties of the 2-category
$\mathbf{Cat}(\mathcal{C})$, which in conjunction allow for an increasingly convenient study of the theory of categories internal 
to $\mathcal{C}$.

Given that every base category $\mathcal{C}$ is $\mathrm{Set}$-enriched by definition itself, it follows that the 2-categorical 
structure on $\mathrm{Cat}(\mathrm{Set})$ is universal among the 2-categorical structures on $\mathrm{Cat}(\mathcal{C})$ over 
arbitrary bases $\mathcal{C}$
in the following way. The left exact Yoneda embedding $y\colon\mathcal{C}\rightarrow\mathrm{Fun}(\mathcal{C}^{op},\mathrm{Set})$ induces 
a functorial push-forward
\[y_{\ast}\colon\mathrm{Cat}(\mathcal{C})\rightarrow\mathrm{Fun}(\mathcal{C}^{op},\mathrm{Cat}(\mathrm{Set}))\]
which is called the externalization functor in \cite{jacobsttbook}. It takes an internal category $X\in\mathrm{Cat}(\mathcal{C})$ to the 
$\mathcal{C}$-indexed category $y_{\ast}(X)$ which evaluates an object $C\in\mathcal{C}$ at the category whose objects are
$C$-indexed generalized elements of objects in $X$ -- i.e.\ elements in the set $\mathcal{C}(C,X_0)$ -- and whose morphisms are
$C$-indexed generalized elements of morphisms in $X$ - i.e.\ elements in the set $\mathcal{C}(C,X_1)$. The externalization functor
can be enhanced to a functor
\begin{align}\label{equ1extintro}
\mathrm{Ext}\colon\mathbf{Cat}(\mathcal{C})\rightarrow\mathbf{Fun}(\mathcal{C}^{op},\mathbf{Cat})
\end{align}
of 2-categories, where the codomain is canonically $\mathrm{Cat}$-enriched by virtue of the 2-categorical structure of $\mathbf{Cat}$.
This 2-functor is locally fully-faithful-and-essentially-surjective (and furthermore preserves lots of additional structure as well) as 
is shown in \cite[Section 7.3]{jacobsttbook}. The $\mathcal{C}$-indexed categories contained in the essential image of the 
externalization functor are commonly referred to as the ``small'' categories over $\mathcal{C}$. Thus, the 2-category of internal 
categories in $\mathcal{C}$ is equivalent to the 2-category of small indexed categories over $\mathcal{C}$.

\subsection{Internal \texorpdfstring{$(\infty,1)$-}{higher }category theory and outline of the paper}

Fundamentally, all we said about category theory in the prior section remains true for $\infty$-category theory -- the 
latter being the study of the $(\infty,2)$-category $\mathbf{Cat}_{\infty}$ of $\infty$-categories, homotopy-coherent functors,
homotopy-coherent natural transformations, and their higher homotopies. 
For every base $\infty$-category $\mathcal{C}$ there is an $\infty$-category $\mathrm{Cat}_{\infty}(\mathcal{C})$ of internal
$\infty$-categories, which in this generality will be defined in Section~\ref{secpre}. Once more, it is true that the $\infty$-category
$\mathrm{Cat}_{\infty}$ of $\infty$-categories is equivalent to the $\infty$-category $\mathrm{Cat}_{\infty}(\mathcal{S})$ of
$\infty$-categories internal to the $\infty$-category $\mathcal{S}$ of spaces. This equivalence has been enhanced to an
$(\infty,2)$-categorical equivalence in \cite[Proposition E.2.2]{riehlverityelements}. In Section~\ref{secformal} we 
will more generally define an $(\infty,2)$-categorical structure $\mathbf{Cat}_{\infty}(\mathcal{C})$ on $\mathrm{Cat}_{\infty}(\mathcal{C})$ 
for every $\infty$-categorical base $\mathcal{C}$, which recovers $\mathbf{Cat}_{\infty}$ up to equivalence in case
$\mathcal{C}=\mathcal{S}$. While it may be trickier to define such an $(\infty,2)$-categorical structure on
$\mathrm{Cat}_{\infty}(\mathcal{C})$ by hand 
directly than it is to define a 2-categorical structure on $\mathrm{Cat}(\mathcal{C})$ by hand in the 1-categorical case, we will use 
the 2-equivalence (\ref{equ1extintro}) as motivation to \emph{define} the $(\infty,2)$-category $\mathbf{Cat}_{\infty}(\mathcal{C})$ 
representably by means of an $\infty$-categorical externalization construction together with the canonically induced 
$(\infty,2)$-categorical structure on $\mathrm{Fun}(\mathcal{C}^{op},\mathrm{Cat}_{\infty})$ from that of $\mathbf{Cat}_{\infty}$. That 
means we identify internal $\infty$-categories with small $\mathcal{C}$-indexed $\infty$-categories to do so. This identification in 
particular induces a functorial choice of mapping $\infty$-categories on the collection of internal $\infty$-categories essentially by 
definition. Whenever $\mathcal{C}$ has finite limits, we will however also describe equivalent ``enhanced mapping $\infty$-categories'' on 
$\mathrm{Cat}_{\infty}(\mathcal{C})$ in the sense of \cite{ghnlaxlimits} that are expressed in terms of the mapping spaces of
$\mathrm{Cat}_{\infty}(\mathcal{C})$ directly.

The canonical $(\infty,2)$-categorical structure $\mathbf{Fun}(\mathcal{C}^{op},\mathbf{Cat}_{\infty})$ on
$\mathrm{Fun}(\mathcal{C}^{op},\mathrm{Cat}_{\infty})$ induced by that of $\mathbf{Cat}_{\infty}$ is very rich in 
structure. To make this precise, we will make use of Riehl and Verity's framework of $\infty$-cosmoses from \cite{riehlverityelements} 
which will be recalled in Section~\ref{secpre}. In essence, an $\infty$-cosmos $\mathbf{C}$ is a quasi-categorically enriched fibration 
category of cofibrant objects (or at least with cofibrant replacements) that has enough $\mathrm{Cat}_{\infty}$-enriched limits to set 
up an internal theory of a plethora of $\infty$-categorical structures. An
$\infty$-cosmos $\mathbf{C}$ provides a convenient framework to do formal $\infty$-category theory in the (finitely) complete
$(\infty,2)$-category associated to $\mathbf{C}$, very much in style of how model categories -- or Brown's fibration categories more 
generally -- provide a convenient framework to do (fragments of) formal homotopy theory in their associated (finitely) complete
$\infty$-category. The main benefit of knowing that a given $(\infty,2)$-category $\mathbf{C}$ has an $\infty$-cosmological presentation is 
that Riehl and Verity have generated a large amount of general results and constructions in \cite{riehlverityelements} which apply to all 
(or, respectively, large classes of) $\infty$-cosmoses. In particular, these results and constructions can be directly ``loaded'' into such
$\mathbf{C}$ so they do not have to be reconstructed time and time again. This paper will therefore not develop any particular formal
$\infty$-categorical notions inside $\mathbf{Cat}_{\infty}(\mathcal{C})$. Rather, among others, it will give criteria on the base
$\mathcal{C}$ that allow to apply the results of \cite{rvyoneda, riehlverityelements} to $\mathbf{Cat}_{\infty}(\mathcal{C})$.

Therefore, first and foremost, we use that the $(\infty,2)$-category $\mathbf{Cat}_{\infty}$ is (presented by) an $\infty$-cosmos, and that
$\mathbf{Fun}(\mathcal{C}^{op},\mathbf{Cat}_{\infty})$ is so as well for every $\infty$-category $\mathcal{C}$. The main results are the following.

\begin{theorem*}[Theorem~\ref{corcosmos1}]
The $(\infty,2)$-category $\mathbf{Cat}_{\infty}(\mathcal{C})$ is
\begin{enumerate}
\item a full finitary sub-$\infty$-cosmos of $\mathbf{Fun}(\mathcal{C}^{op},\mathbf{Cat}_{\infty})$ whenever $\mathcal{C}$ is left exact. 
It thus defines an $\infty$-cosmos in the weaker sense of \cite{rvyoneda}.
\item a full sub-$\infty$-cosmos of $\mathbf{Fun}(\mathcal{C}^{op},\mathbf{Cat}_{\infty})$ that is closed under all small limits (and 
exponentials) whenever $\mathcal{C}$ is complete (and cartesian closed). It thus defines a (cartesian closed) $\infty$-cosmos in the 
sense of \cite{riehlverityelements}.
\end{enumerate}
\end{theorem*}

In style of Street's results \cite[Section 4]{streetintcat} for the 2-category of categories internal to a 1-category, we obtain the 
following hierarchy of regularity properties of $\mathbf{Cat}_{\infty}(\mathcal{C})$.

\begin{theorem*}[Theorem~\ref{corcosmos2} and Proposition~\ref{prop1oc2topos}]\label{thm2}
Let $\mathcal{C}$ be an $\infty$-category and let
\[\mathrm{Ext}\colon\mathbf{Cat}_{\infty}(\mathcal{C})\hookrightarrow\mathbf{Fun}(\mathcal{C}^{op},\mathbf{Cat}_{\infty})\]
be the canonical embedding of $(\infty,2)$-categories.
\begin{enumerate}
\item Suppose $\mathcal{C}$ is (finitely) complete. Then $\mathbf{Cat}_{\infty}(\mathcal{C})$ is (finitely) $(\infty,2)$-complete, and
$\mathrm{Ext}$ preserves these limits.
\item Suppose $\mathcal{C}$ is countably complete and cartesian closed. Then the underlying $\infty$-category
$\mathrm{Cat}_{\infty}(\mathcal{C})$ is cartesian closed and $\mathrm{Ext}$ preserves exponentials.
\item Suppose $\mathcal{C}$ is finitary lextensive. Then $\mathbf{Cat}_{\infty}(\mathcal{C})$ is strongly corepresentable (in a sense dual to 
that of Gray \cite{graymidwest}, see Definition~\ref{defcorep}).
\item Suppose $\mathcal{C}$ is an $\infty$-topos. Then $\mathrm{Cat}_{\infty}(\mathcal{C})$ is an $(\infty,1)$-localic $(\infty,2)$-topos.
\end{enumerate}
\end{theorem*}

\begin{remark*}
Part 4 of Theorem~\ref{thm2} is first and foremost listed as a reference statement that the other parts are to be compared 
against. Depending on the definition of ``$(\infty,1)$-localic $(\infty,2)$-topos'' the statement is in fact trivial
\cite[Definition 9.2.1]{am_2topos}. We will however give a definition of an $(\infty,1)$-localic $(\infty,2)$-topos in 
style of Rezk's model toposes \cite{rezkhtytps} that is more meaningful in context of Theorem~\ref{corcosmos2}.
\end{remark*}

\begin{corollary*}
Every complete $\infty$-category $\mathcal{C}$ has 
\begin{itemize}
\item internal comma $\infty$-categories, and hence in particular internal $\infty$-categories of morphisms, internal $\infty$-categories 
of (co)cones, and internal (co)slice $\infty$-categories.
\item an internal theory of (fibered) diagram $\infty$-categories, together with a theory of (fibered) internal adjunctions, of
(fibered) internal equivalences, and of limits and colimits of diagrams. This comprises many characterizations thereof, as well as 
theorems about preservation, reflection, creation and (co)finality of such.
\item an internal theory of (co)cartesian fibrations, discrete fibrations, two-sided fibrations and their modules. This comprises various 
versions of the internal Yoneda lemma.
\item an internal theory of Kan extensions and Kan lifts, both global and pointwise.
\item an internal theory of pointed and stable $\infty$-categories together with a theory of loop and suspension functors for such.
\item internal $\infty$-categories of algebras for monads on internal $\infty$-categories, together with a corresponding monadicity 
theorem.
\end{itemize}
\end{corollary*}
\begin{proof}
This is \cite{riehlverityelements} by way of Theorem~\ref{corcosmos1}.
\end{proof}

Various parts of the Corollary already hold when $\mathcal{C}$ is finitely complete, see \cite{rvyoneda}.
The Corollary in particular applies to every presentable $\infty$-category $\mathcal{C}$; for instance it applies to any $\infty$-topos, 
but for another instance it also applies to the $\infty$-category $\mathrm{Cat}_{(\infty,n)}$ of $(\infty,n)$-categories for any $n\geq 0$ 
\cite{bsp_inf_n}. 
It however applies as well to Clausen and Scholze's $\infty$-category of condensed spaces \cite{cs_condensed}, which is not presentable. It 
further applies to the $\infty$-category 
of spectra, the $\infty$-category of symmetric monoidal $\infty$-categories, the $\infty$-category of $\infty$-operads, and many others. 
This has the potential to be useful in various ways as the notion of internal $\infty$-category in particular captures both that of
$E_1$-monoid (in cartesian monoidal structures) as well as that of internal equivalence relation for instance.

The $(\infty,2)$-categorical limits in $\mathbf{Cat}_{\infty}(\mathcal{C})$ in Theorem~\ref{thm2}.1 are inherited directly from the
$(\infty,2)$-category $\mathbf{Fun}(\mathcal{C}^{op},\mathbf{Cat}_{\infty})$ whenever they exist. The $(\infty,2)$-categorical colimits -- 
such as the tensors in Theorem~\ref{corcosmos2}.3 -- however are generally not preserved by the embedding of
$\mathbf{Cat}_{\infty}(\mathcal{C})$ in
$\mathbf{Fun}(\mathcal{C}^{op},\mathbf{Cat}_{\infty})$. This is a variation of the fact that coproducts -- i.e.\ Set-enriched tensors -- are 
generally not preserved by the Yoneda embedding either.
In context of this comparison, we additionally state a corresponding Yoneda lemma (Proposition~\ref{lemmaextyoneda}) which expresses
the $\infty$-category of natural transformations out of a small indexed $\infty$-category in terms of a corresponding $\infty$-categorical 
totalization. This is useful for instance to show that the $\infty$-category of internal presheaves over any internal $\infty$-category 
in an $\infty$-topos is again an $\infty$-topos (Corollary~\ref{corintpretopoi}).

In Section~\ref{secmodcat} we point out that constructions parallel to those of Section~\ref{secformal} can be found in the literature 
regarding combinatorial and left proper model categories. This is the case implicitly in Hovey's textbook \cite{hovey} and subsequently
in Dugger's work on internal $\infty$-groupoids \cite{duggersimp}. This is more explicitly the case in Riehl and Verity's 
construction of the $\infty$-cosmos of ``Rezk-objects'' in such a model category, see \cite[Proposition 2.2.9]{rvyoneda} and
\cite[Proposition E.3.7]{riehlverityelements}. 
The aim of Section~\ref{secmodcat} is to show that the underlying $(\infty,2)$-categorical structures derived from 
\cite{duggersimp} and \cite[Proposition 2.2.9]{rvyoneda} reduce exactly to the ones studied in 
Section~\ref{secformal}. Therefore, we first observe that most of the relevant constructions in the context of such model categories can 
in fact be carried out for
all model categories. We thus define the $\infty$-cosmos $\mathbf{Cat}_{\infty}(\mathbb{M})$ of internal $\infty$-categories in any 
model category $\mathbb{M}$ (in the weaker sense of \cite{rvyoneda} in this generality only in as much as not all objects are 
necessarily cofibrant). This recovers the $\infty$-cosmoses obtained from \cite{duggersimp} and \cite[Proposition 2.2.9]{rvyoneda} 
whenever $\mathbb{M}$ is left proper and combinatorial. We then construct a right derived externalization functor which is suitably 
continuous, enriched and exact, and which furthermore recovers the $\infty$-categorical externalization functor from 
Section~\ref{secformal} whenever $\mathbb{M}$ is a suitable model category. This in particular shows that the $(\infty,2)$-category
$\mathbf{Cat}_{\infty}(\mathbb{M})$ recovers that of Section~\ref{secformal}, which allows for a straight-forward proof of 
Proposition~\ref{prop1oc2topos}.

\paragraph{Related literature}

A comprehensive development of internal (1-)category theory is provided among others by \cite{streetintcat}, \cite[Section 7]{jacobsttbook},
\cite[Section B2]{elephant}, and \cite{bourkethesis}.
Many fundamental $\infty$-categorical constructions internal to a presheaf $\infty$-topos have been developed in
\cite{shahparahctI,shahparahctII,barwicketalparahct}. Many fundamental and further $\infty$-categorical constructions internal to a general 
$\infty$-topos have been developed in a series of papers by Martini \cite{martini_yoneda}, and Martini and Wolf 
\cite{martiniwolf_lim,martiniwolf_straight,martiniwolf_pres,martiniwolf_ihtt}.

\paragraph{Future work}
Given a general 2-category $\mathbf{C}$, Bourke \cite{bourkethesis} gave a list of internal criteria that characterizes precisely when
$\mathbf{C}$ is equivalent to the 2-category $\mathbf{Cat}(\mathcal{C})$ of internal categories for some category $\mathcal{C}$. In case 
these criteria are satisfied, the category $\mathcal{C}$ is given by the category of discrete objects in
$\mathbf{C}$. An analogous list of criteria is expected to characterize $(\infty,2)$-categories $\mathbf{C}$ as the $(\infty,2)$-category
$\mathbf{Cat}_{\infty}(\mathcal{C})$ of internal categories for an $\infty$-category $\mathcal{C}$.
Such an internal characterization will be left open however.

\section{Preliminaries}\label{secpre}

\begin{notation*}
In all of the following, $\mathbf{S}$ will denote the simplicially enriched category of simplicial sets. The simplicially enriched 
category $\mathbf{S}$ equipped with the standard model structure for Kan complexes will be denoted by $(\mathbf{S},\mathrm{Kan})$, and 
with the Joyal model structure for quasi-categories by $(\mathbf{S},\mathrm{QCat})$. The category of simplicially enriched categories 
will be denoted by $\mathbf{S}\text{-Cat}$.
Variables ranging over $\infty$-categories (and over ordinary categories in particular) will be italized, variables and constructions 
ranging over (mostly simplicially) enriched categories will be bold faced, variables ranging over categories equipped with additional 
homotopical structure will be denoted by blackboard letters. 
We denote the quasi-category of small $\infty$-categories (for any of the equivalent models of $\infty$-category theory) by
$\mathrm{Cat}_{\infty}$. We denote the quasi-category of small spaces by $\mathcal{S}$. The exponential of two quasi-categories
$\mathcal{C}$, $\mathcal{D}$ in $\mathbf{S}$ will (most of the times) be denoted by $\mathrm{Fun}(\mathcal{C},\mathcal{D})$. The
quasi-category $\hat{\mathcal{C}}$ denotes the quasi-category $\mathrm{Fun}(\mathcal{C}^{op},\mathcal{S})$ of presheaves over
$\mathcal{C}$.
\end{notation*}

\subsection{\texorpdfstring{$(\infty,2)$}{Higher 2}-categories and \texorpdfstring{$\infty$}{infinity}-cosmoses}

As far as this paper is concerned, an $(\infty,2)$-category is a simplicially enriched category $\mathbf{C}$ whose hom-objects are 
quasi-categories. These are called ``strict $(\infty,2)$-categories'' in \cite[Section 5.5.8]{kerodon}. Accordingly, a functor of
$(\infty,2)$-categories is simply a simplicially enriched functor between $(\infty,2)$-categories. The associated homotopy theory 
is equivalent to other models of $(\infty,2)$-category theory via \cite[Theorem 0.0.3 and Remark 0.0.4]{lgood}.

\begin{example}[$\infty$-categories]\label{expleinftycatsareinfty2}
The model category $(\mathbf{S},\mathrm{QCat})$ is Quillen equivalent to the category $\mathbf{S}\text{-Cat}$ of 
(small) simplicially enriched categories equipped with the Bergner model structure via the homotopy coherent nerve
$N_{\Delta}\colon\mathbf{S}\text{-Cat}\rightarrow\mathbf{S}$ and its Quillen left adjoint $\mathfrak{C}$ \cite{cordiernerve, luriehtt}. 
In particular, the $\infty$-category of $\infty$-categories (however presented) is equivalent to the homotopy $\infty$-category 
of $(\mathbf{S}\text{-Cat},\text{Bergner})$. Furthermore, the fibrant objects in
$(\mathbf{S}\text{-Cat},\text{Bergner})$ are exactly the simplicially enriched categories whose hom-objects are Kan complexes. 
Thus, the fibrant models of $\infty$-categories in $(\mathbf{S}\text{-Cat},\text{Bergner})$ are exactly those $(\infty,2)$-categories whose 
hom-quasi-categories are Kan complexes.
\end{example}		

To recall Riehl and Verity's definition of $\infty$-cosmoses, as well as to treat varying themes thereof which arise in 
Sections~\ref{secformal} and \ref{secmodcat} simultaneously, we only ask the reader to recall Brown's notion of a category of fibrant 
objects \cite{brown}. We will refer to such as \emph{fibration categories} following \cite{szumilococomplqcats}, and further recall the 
notion of exact functors between fibration categories \cite[Definition 1.6]{szumilococomplqcats}. The dual theory is that of
cofibration categories and exact functors between such. We denote the underlying 
fibration (cofibration) category of a model category $\mathbb{M}$ by $\mathbb{M}^f$ ($\mathbb{M}^c$). An object in a fibration category 
is cofibrant if it has the left lifting property against all trivial fibrations. One defines fibrant objects in cofibration categories 
dually. 

\begin{definition}\label{def_fibcatcof}
Say that a fibration category $\mathbb{F}=(\mathcal{F},W,F)$ is a \emph{fibration category with cofibrant replacements} if for all 
objects $A\in\mathbb{F}$ there is a trivial fibration $\mathbb{L}A\twoheadrightarrow A$ whose domain is cofibrant. A 
fibration category $\mathbb{F}$ is a \emph{fibration category of cofibrant objects} if all its objects are 
cofibrant. 
\end{definition}

\begin{definition}\label{defmoncofcat}
Let $\mathbb{C}$ be a cofibration category such that the underlying category of $\mathbb{C}$ is equipped with a monoidal structure
$\otimes$. Say that $(\mathbb{C},\otimes)$ is a monoidal cofibration category if for every pair $j\colon A\rightarrow B$,
$k\colon C\rightarrow D$ of cofibrations in $\mathbb{C}$ the co-gap map
$j\hat{\otimes}k\colon A\otimes D\sqcup_{A\otimes C} B\otimes C \rightarrow B\otimes D$ is again a cofibration in $\mathcal{C}$ which is 
trivial whenever either $j$ or $k$ is trivial.
\end{definition}

One may want to further add cocontinuity conditions on the functor $\otimes\colon\mathbb{C}\times\mathbb{C}\rightarrow\mathbb{C}$ in 
Definition~\ref{defmoncofcat}, however we won't need such for the few definitions to follow in this generality.

\begin{example}\label{explecofcatkappasmallss}
The cofibration category $(\mathbf{S},\mathrm{QCat})^c$ is monoidal via its cartesian product. Let
$\mathbf{S}^{\text{fin}}\subset\mathbf{S}$ be the full subcategory spanned by those simplicial sets which are weakly equivalent in
$(\mathbf{S},\mathrm{QCat})$ to a simplicial set with only finitely many non-degenerate simplices. We will refer to such simplicial sets 
simply as \emph{finite}. Then $\mathbf{S}^{\text{fin}}$ inherits a canonical structure of a monoidal cofibration category
$(\mathbf{S}^{\text{fin}},\mathrm{QCat})^c$ by reflection of that from $(\mathbf{S},\mathrm{QCat})^c$.
\end{example}

\begin{definition}\label{defenrfibcat}
Let $(\mathbb{C},\otimes)$ be a monoidal cofibration category,
and let $\mathbf{F}$ be a cotensored $(\mathbb{C},\otimes)$-enriched category. 
Furthermore, suppose $\mathbb{F}_0=(\mathbf{F}_0,W,F)$ is a fibration category structure on the underlying category $\mathbf{F}_0$ of
$\mathbf{F}$. 
Say that $\mathbb{F}=(\mathbf{F},W,F)$ is a $\mathbb{C}$-enriched fibration category if for every fibration
$p\colon X\twoheadrightarrow Y$ in $\mathbb{F}_0$ and every cofibration $j\colon A\rightarrow B$ in $\mathbb{C}$, 
the gap map $p^j\colon X^B\rightarrow X^A\times_{Y^A} Y^B$ is a fibration in $\mathbb{F}_0$, which is trivial whenever 
either $p$ or $j$ is trivial.
\end{definition}

\begin{example}\label{expleenrmodtofibcat}
If $\mathbb{M}$ is a monoidal model category and $\mathbb{N}=(\mathcal{N},C,W,F)$ is a model category whose underlying category has a bitensored $\mathbb{M}$-enrichment
$\mathbf{N}$, then $(\mathbf{N},C,W,F)$ is an $\mathbb{M}$-enriched model category \cite[Section A.3.1]{luriehtt} if and only if
$(\mathbf{N}^f,W,F)$ is an $\mathbb{M}^c$-enriched fibration category.
\end{example}

%
%

\begin{definition}\label{defCexact}
A $\mathbb{C}$-enriched functor $G\colon \mathbb{F}_1\rightarrow\mathbb{F}_2$ between $\mathbb{C}$-enriched fibration categories
is \emph{$\mathbb{C}$-exact} if its underlying functor is an exact functor of underlying fibration categories, and $G$ preserves all
$\mathbb{C}$-cotensors. If $\mathbb{F}_1$ has all small products as well as countably sequential limits of fibrations, say 
$G$ is \emph{transfinitely $\mathbb{C}$-exact} if $G$ furthermore preserves those limits.
\end{definition}

We note that Definition~\ref{defCexact} only depends on the underlying monoidal category of $\mathbb{C}$ (without its cofibration 
structure) when domain and codomain are known to be $\mathbb{C}$-enriched. Therefore, we will suppress the cofibration structure in this 
context whenever notationally convenient. Clearly, the composition of two $\mathbb{C}$-exact functors is again $\mathbb{C}$-exact.

\begin{definition}
Let $(\mathbb{C},\otimes)$ be a monoidal cofibration category, and let $\mathbb{F}$ be a $\mathbb{C}$-enriched fibration 
category. We say that $\mathbb{F}$ is \emph{cartesian closed} if for all objects $A\in\mathbb{F}$ there is a $\mathbb{C}$-enriched 
adjunction $A\times (\cdot)\adj(\cdot)^A$ such that each right adjoint $(\cdot)^A\colon\mathbf{F}\rightarrow\mathbf{F}$ preserves both 
fibrations and trivial fibrations.
\end{definition}

\begin{remark}\label{remfibcatcc}
A $\mathbb{C}$-enriched fibration category $\mathbb{F}$ is cartesian closed if and only if the underlying category $\mathbf{F}_0$ is 
cartesian closed, and for each $A\in\mathbf{F}$ the right adjoint $(\cdot)^A\colon\mathbb{F}\rightarrow\mathbb{F}$ is (transfinitely)
$\mathbb{C}$-exact.
\end{remark}

\begin{example}[$\infty$-cosmoses]\label{explecosmosalt}
A $(\mathbf{S},\mathrm{QCat})^c$-enriched fibration category $\mathbb{F}$ of cofibrant objects which 
has countable sequential limits of fibrations and small products is exactly an $\infty$-cosmos (``of cofibrant objects'') in the sense 
of \cite{riehlverityelements}. One direction is immediate, the other direction is \cite[Example C.1.3]{riehlverityelements}
(and \cite[Lemma C.1.9]{riehlverityelements}). Any such $(\mathbf{S},\mathrm{QCat})^c$-enriched fibration category $\mathbb{F}$ is 
cartesian closed if and only if it is so as an $\infty$-cosmos in the sense of \cite[Definition 1.2.23]{riehlverityelements}.
A transfinitely $(\mathbf{S},\mathrm{QCat})^c$-exact functor between $\infty$-cosmoses is exactly a cosmological functor in the sense 
of \cite[Definition 1.3.1]{riehlverityelements}.
\end{example}

\begin{example}\label{expleqcatmcs}
In particular, every $(\mathbf{S},\mathrm{QCat})$-enriched model category $\mathbb{M}$ in which all fibrant objects are cofibrant has an 
underlying $\infty$-cosmos $\mathbb{M}^f$ of fibrant objects as defined in \cite{riehlverityelements}. Whenever such a model category is 
furthermore cartesian closed (as a simplicially enriched model category),
then so is its underlying $\infty$-cosmos $\mathbb{M}^f$. In particular, both model categories $(\mathbf{S},\mathrm{QCat})$ and
$(\mathbf{S},\mathrm{Kan})$ have underlying cartesian closed $\infty$-cosmoses (of cofibrant objects), which we short-handedly denote by 
$\mathbf{QCat}$ and $\mathbf{Kan}$ respectively. The former is an $\infty$-cosmos of $(\infty,1)$-categories in the sense of 
\cite{riehlverityelements}. It in fact is the reference structure for the definition of such $\infty$-cosmoses, and up to equivalence 
the only one. Accordingly, the latter is the $\infty$-cosmos of discrete $(\infty,1)$-categories
\cite[Propositions 1.2.12 and 6.1.6]{riehlverityelements}. In particular, its own hom-quasi-categories are Kan complexes, and so
$\mathbf{Kan}$ is an $\infty$-category in the sense of Example~\ref{expleinftycatsareinfty2}.
\end{example}

\begin{remark}\label{remfibcatsinfty2}
If $\mathbb{F}_1$ is a $(\mathbf{S},\mathrm{QCat})^c$-enriched fibration category and $A\in\mathbb{F}_1$ is cofibrant, one computes that 
the hom-object $\mathbb{F}_1(A,B)\in\mathbb{C}$ is a quasi-category for all $B\in\mathbb{F}_1$. In particular, the simplicially 
enriched full subcategory $\mathbb{F}_1^c\subseteq\mathbb{F}_1$ spanned by the cofibrant objects is an $(\infty,2)$-category. Suppose
$\mathbb{F}_2$ is some other $(\mathbf{S},\mathrm{QCat})^c$-enriched fibration category which exhibits a simplicially enriched cofibrant 
replacement functor $\lambda\colon\mathbb{F}_2\rightarrow\mathbb{F}_2^c$. Then every simplicially enriched functor
$G\colon\mathbb{F}_1\rightarrow\mathbb{F}_2$ induces a functor
\[\mathbb{F}_1^c\hookrightarrow\mathbb{F}_1\xrightarrow{G}\mathbb{F}_2\xrightarrow{\lambda}\mathbb{F}_2^c\]
of $(\infty,2)$-categories. In particular, every $\infty$-cosmos is an $(\infty,2)$-category, and every cosmological functor between
$\infty$-cosmoses is a functor of $(\infty,2)$-categories.
\end{remark}

Subsequently, to set things up, in this section we use the $\infty$-cosmos $\mathbf{QCat}$ of quasi-categories as a reference structure 
for $\infty$-category theory. We make a note about model independence concerning the rest of the paper in the end of this section.

The model category $(\mathbf{S}\text{-Cat},\text{Bergner})$ itself is not cartesian closed as a model category (see e.g.\
\cite[Remark 4.5.8]{bergnerbook}), which is one of the reasons that it is rarely worked with as a model for $\infty$-category theory in 
practice. Yet, the category $\mathbf{S}\text{-Cat}$ itself is cartesian closed indeed, and there are various special cases in which the 
simplicially enriched exponential $\mathbf{Fun}(\mathbf{C},\mathbf{D})$ of two simplicially enriched categories $\mathbf{C}$,
$\mathbf{D}$ does help to compute the corresponding $\infty$-categorical exponential after all. 
Such a case is given whenever the codomain $\mathbf{D}$ comes equipped with a $(\mathbf{S},\mathrm{QCat})$-enriched model structure, as 
is the case in the following example.

\begin{example}\label{expleqcatiscosmos}
For any small simplicially enriched category $\mathbf{C}$, the exponential $\mathbf{Fun}(\mathbf{C}^{op},\mathbf{S})$ inherits an 
injective and a projective model structure from both $(\mathbf{S},\mathrm{QCat})$ and $(\mathbf{S},\mathrm{Kan})$ each. The simplicially
enriched category $\mathbf{Fun}(\mathbf{C}^{op},\mathbf{S})$ is both tensored and cotensored over $\mathbf{S}$, all four model 
structures are $(\mathbf{S},\mathrm{QCat})$-enriched, and in the case of the injective model structures also all objects are cofibrant.
We thus obtain underlying $\infty$-cosmoses
\[\mathbf{Fun}(\mathbf{C}^{op},\mathbf{QCat}):=\mathbf{Fun}(\mathbf{C}^{op},(\mathbf{S},\mathrm{QCat}))_{\mathrm{inj}}^f\]
and
\[\mathbf{Fun}(\mathbf{C}^{op},\mathbf{Kan}):=\mathbf{Fun}(\mathbf{C}^{op},(\mathbf{S},\mathrm{Kan}))_{\mathrm{inj}}^f.\]
If we denote by $\mathcal{C}$ the associated quasi-category of $\mathbf{C}$ (i.e.\ $\mathcal{C}$ is equivalent to the simplicial nerve 
of a fibrant replacement of $\mathbf{C}$ in the Bergner model structure), the former is the $\infty$-cosmos of $\mathcal{C}$-indexed
quasi-categories. The latter is the $\infty$-cosmos of presheaves over $\mathcal{C}$ (which in fact is an $\infty$-category in the sense of Example~\ref{expleinftycatsareinfty2}).
By construction, the $(\infty,2)$-categorical structure $\mathbf{Fun}(\mathbf{C}^{op},\mathbf{QCat})$ on the collection of
$\mathcal{C}$-indexed quasi-categories is induced by the $(\infty,2)$-categorical structure of $\mathbf{QCat}$ itself. In particular, 
its 1-cells are homotopy-coherent natural transformations between functors, and its 2-cells are homotopy-coherent modifications between 
1-cells.

The underlying quasi-category (i.e.\ the simplicial nerve) of $\mathbf{Fun}(\mathbf{C}^{op},\mathbf{QCat})$ is the quasi-category
$\mathrm{Fun}(\mathcal{C}^{op},N_{\Delta}(\mathbf{QCat}))$. The underlying quasi-category of
$\mathbf{Fun}(\mathbf{C}^{op},\mathbf{Kan})$ is the quasi-category $\hat{\mathcal{C}}$ of presheaves over $\mathcal{C}$.
\end{example}

\begin{lemma}\label{lemmacqcatiscc}
For all small simplicially enriched categories $\mathbf{C}$, both $(\mathbf{S},\mathrm{QCat})$-enriched model categories
$\mathbf{Fun}(\mathbf{C}^{op},(\mathbf{S},\mathrm{QCat}))_{\mathrm{inj}}$ and
$\mathbf{Fun}(\mathbf{C}^{op},(\mathbf{S},\mathrm{Kan}))_{\mathrm{inj}}$ are cartesian closed. In particular, the $\infty$-cosmoses
$\mathbf{Fun}(\mathbf{C}^{op},\mathbf{QCat})$ and $\mathbf{Fun}(\mathbf{C}^{op},\mathbf{Kan})$ are cartesian closed.
\end{lemma}
\begin{proof}
We do the case for $\mathbf{Fun}(\mathbf{C}^{op},(\mathbf{S},\mathrm{QCat}))_{\mathrm{inj}}$; the other is completely analogous.
We recall that the underlying category of $\mathbf{Fun}(\mathbf{C}^{op},\mathbf{S})$ is cartesian closed. Its products are computed 
pointwise, and the exponential $G^F$ for a pair of functors $F,G\colon\mathbf{C}^{op}\rightarrow\mathbf{S}$ evaluates an object 
$C\in\mathbf{C}$ at the simplicial set $\mathrm{Nat}(yC\times F,G)$ of simplicially enriched natural transformations. Each adjunction 
$ F\times(\cdot)\adj (\cdot)^F$ is a simplicially enriched adjunction by \cite[Proposition 3.7.10]{riehlcthty}, because the left 
adjoints $F\times (\cdot)$ commute with simplicial tensors (as the tensor $G\otimes J$ with a simplicial set $J$ can be computed as the 
product $G\times c_J$ where $c_J$ is the constant functor with value $J$). Furthermore, the model category
$\mathbf{Fun}(\mathbf{C}^{op},(\mathbf{S},\mathrm{QCat}))_{\mathrm{inj}}$ is cartesian closed as well. Indeed, the product functor
\[(\cdot)\times(\cdot)\colon\mathbf{Fun}(\mathbf{C}^{op},(\mathbf{S},\mathrm{QCat}))_{\mathrm{inj}}\times\mathbf{Fun}(\mathbf{C}^{op},(\mathbf{S},\mathrm{QCat}))_{\mathrm{inj}}\rightarrow\mathbf{Fun}(\mathbf{C}^{op},(\mathbf{S},\mathrm{QCat}))_{\mathrm{inj}}\]
is a left Quillen bifunctor as is verified pointwise using that $(\mathbf{S},\mathrm{QCat})$ is cartesian closed. That means that
$\mathbf{Fun}(\mathbf{C}^{op},(\mathbf{S},\mathrm{QCat}))_{\mathrm{inj}}$ is cartesian closed as stated. It follows that 
the $\infty$-cosmos $\mathbf{Fun}(\mathbf{C}^{op},\mathbf{QCat})$ is cartesian closed by Example~\ref{expleqcatmcs}.
\end{proof}

\begin{definition}
Let $\mathbf{C}$ be an $(\infty,2)$-category. The full sub-$(\infty,2)$-category generated by a collection 
$D\subseteq\mathbf{C}$ of objects is the full simplicial subcategory $\mathbf{D}$ of $\mathbf{C}$ spanned by $D$. That is, the
$(\infty,2)$-category $\mathbf{D}$ given by the objects in $D$ and the hom-quasi-categories $\mathbf{D}(C,D):=\mathbf{C}(C,D)$. 
\end{definition}

\begin{definition}\label{defsubcosmoses}
Suppose $\mathbb{F}$ is a $(\mathbf{S},\mathrm{QCat})^c$-enriched fibration category. If $D$ is a collection of objects in 
$\mathbb{F}$ such that the inclusion $\mathbf{D}\subseteq\mathbb{F}$ is replete with respect to the class of weak equivalences in
$\mathbb{F}$, and such that $D$ contains the terminal object, is closed under pullbacks of fibrations as well as under all simplicial 
cotensors with finite simplicial sets, we say that $\mathbf{D}$ equipped with the canonical fibration category structure 
inherited from $\mathbb{F}$ is a \emph{full finitary sub-$\infty$-cosmos} of $\mathbb{F}$. 
If $\mathbb{F}$ has all countable sequential limits of fibrations and small products, and $D$ is furthermore closed under countable 
sequential limits of fibrations, under small products as well as under all simplicial cotensors, we say that $\mathbf{D}$ equipped with 
the canonical fibration category structure inherited from $\mathbb{F}$ is a \emph{full sub-$\infty$-cosmos} of $\mathbb{F}$.
\end{definition}

\begin{example}
Every full finitary sub-$\infty$-cosmos of a $(\mathbf{S},\mathrm{QCat})^c$-enriched fibration category is an $\infty$-cosmos in the weaker 
sense of \cite{rvyoneda}. Every full sub-$\infty$-cosmos of an $\infty$-cosmos in the sense of Example~\ref{explecosmosalt} is again an 
$\infty$-cosmos in the sense of Example~\ref{explecosmosalt}.
\end{example}

\begin{example}
The collection of finite quasi-categories defines a full finitary sub-$\infty$-cosmos $\mathbf{QCat}^{\mathrm{fin}}$ in $\mathbf{QCat}$.
\end{example}

We have seen in Example~\ref{expleqcatiscosmos} that for every small simplicial category $\mathbf{C}$ the $(\infty,2)$-category
$\mathbf{Fun}(\mathbf{C}^{op},\mathbf{QCat})$ is an $\infty$-cosmos. In the course of the following sections we will be interested in a 
uniform treatment of large simplicial categories $\mathbf{C}$ as well. 

\begin{notation}\label{notation_univ}
For the following, we fix three Grothendieck universes so to be able to talk about 
a large (and locally small) simplicial category $\mathbf{QCat}$ of small quasi-categories, as well as a superlarge
simplicial category $\mathbf{QCAT}$ of large quasi-categories. We denote the superlarge category of large simplicial sets by
$\mathbf{S}^+$.
The plain term ``$\infty$-category'' will henceforth refer to ``large $\infty$-category''.
\end{notation}


\begin{example}\label{exple_size}
The superlarge simplicial category $\mathbf{QCAT}$ is a $(\mathbf{S}^+,\mathrm{QCat})^c$-enriched fibration category of cofibrant objects 
which has countable sequential limits of fibrations and large products (i.e.\ it is an $\infty$-cosmos in the universe of large
quasi-categories). The embedding of simplicial categories $\mathbf{QCat}\hookrightarrow\mathbf{QCAT}$ exhibits $\mathbf{QCat}$ as an
$\infty$-cosmos that inherits all its $\infty$-cosmological structure from the latter (i.e.\ it is a full small sub-$\infty$-cosmos). More 
generally, for every (large) simplicial category $\mathbf{C}$, 
the superlarge simplicial category $\mathbf{Fun}(\mathbf{C}^{op},\mathbf{QCAT})$ is a $(\mathbf{S}^+,\mathrm{QCat})^c$-enriched fibration 
category of cofibrant objects with countable sequential limits of fibrations and large products as well. The embedding
\[\mathbf{Fun}(\mathbf{C}^{op},\mathbf{QCat})\hookrightarrow\mathbf{Fun}(\mathbf{C}^{op},\mathbf{QCAT})\]
hence equips the large simplicial category $\mathbf{Fun}(\mathbf{C}^{op},\mathbf{QCat})$ with the structure of a
$(\mathbf{S}^+,\mathrm{QCat})^c$-enriched fibration category of cofibrant objects with countable sequential limits of fibrations and small 
products even when $\mathbf{C}$ is large. This is precisely the $\infty$-cosmological structure of Example~\ref{expleqcatiscosmos} whenever $\mathbf{C}$ is small.
\end{example}

\begin{remark}\label{remscbicat}
Just as every $\infty$-category $\mathbf{C}$ in the sense of Example~\ref{expleinftycatsareinfty2} has an underlying 
quasi-category given by its associated simplicial nerve $N_{\Delta}(\mathbf{C})$, every $(\infty,2)$-category $\mathbf{C}$ has 
an underlying $\infty$-bicategory given by its associated scaled simplicial nerve $N_{\Delta}^{\mathrm{sc}}(\mathbf{C})$
\cite[Section 4]{lgood}\footnote{To be more precise, the scaled simplicial nerve is really evaluated at the simplicial category
$\mathbf{C}^{\natural}$ whose hom-quasi-categories are canonically marked by their equivalences.}.
It will be used in Theorem~\ref{corcosmos2} only once implicitly to deduce closure of the $(\infty,2)$-category
$\mathbf{Cat}_{\infty}(\mathcal{C})$ under general $(\infty,2)$-limits under the given assumptions of Theorem~\ref{corcosmos1}.
\end{remark}

\subsection{The underlying \texorpdfstring{$\infty$-}{higher }category of \texorpdfstring{an $(\infty,2)$}{a higher 2}-category}\label{secsubformalunderqcat}

We recall that the canonical inclusion $\iota\colon\mathcal{S}\hookrightarrow\mathrm{Cat}_{\infty}$ of quasi-categories has a right 
adjoint $(\cdot)^{\simeq}\colon\mathrm{Cat}_{\infty}\rightarrow\mathcal{S}$. It assigns to an $\infty$-category $\mathcal{C}$ its core
$\mathcal{C}^{\simeq}\hookrightarrow\mathcal{C}$, which is essentially defined by the universal property it satisfies by virtue of the 
adjunction $\iota\adj(\cdot)^{\simeq}$. In $\mathrm{Cat}_{\infty}$ presented as the homotopy quasi-category of
$(\mathbf{S},\mathrm{QCat})$, the core $\mathcal{C}^{\simeq}\subseteq\mathcal{C}$ of 
a (small) $\infty$-category $\mathcal{C}\in\mathrm{Cat}_{\infty}$ has an analytical construction: it is given by the largest 
Kan complex contained in $\mathcal{C}$. This assignment is functorial, and in fact is induced by a right Quillen functor
$k^!\colon(\mathbf{S},\mathrm{QCat})\rightarrow(\mathbf{S},\mathrm{Kan})$ which comes together with a natural trivial fibration
$k^!(\mathcal{C})\overset{\sim}{\twoheadrightarrow}\mathcal{C}^{\simeq}$ for all quasi-categories $\mathcal{C}$
\cite[Section 1]{jtqcatvsss}. If $I\Delta^n$ denotes the nerve of the free groupoid generated by $[n]$, the functor $k^!$ is given by 
the formula $k^!(S)_n=\mathbf{S}(I\Delta^n,S)$. 

\begin{definition}\label{defpith}
The \emph{pith} (or the \emph{$(\infty,1)$-core}) of an $(\infty,2)$-category $\mathbf{C}$ is the simplicially enriched category
$\underline{\mathbf{C}}$ given by $\mathrm{Ob}(\underline{\mathbf{C}})=\mathrm{Ob}(\mathbf{C})$ and
$\underline{\mathbf{C}}(X,Y)=\mathbf{C}(X,Y)^{\simeq}$ for all $X,Y\in\mathcal{C}$. A quasi-category $\mathcal{C}$ is \emph{the 
underlying quasi-category} of $\mathbf{C}$ if
it is equivalent to the underlying quasi-category $N_{\Delta}(\underline{\mathbf{C}})$. 
\end{definition}

The terminology ``pith'' follows \cite{kerodon}, the terminology ``$(\infty,1)$-core'' follows \cite{riehlverityelements}.
The underlying quasi-category of an $(\infty,2)$-category $\mathbf{C}$ is defined to be an equivalence-invariant notion given that, 
first, in models of $\infty$-category theory other than $\mathbf{QCat}$ it generally may only be defined up to equivalence in the first 
place, and second, it is characterized by the following universal property, of which the pith itself is an analytical instantiation.

\begin{proposition}
For any $(\infty,2)$-category $\mathbf{C}$ and any quasi-category $\mathcal{D}$, the canonical inclusion
$\iota\colon\underline{\mathbf{C}}\rightarrow\mathbf{C}$ induces an equivalence
\[N_{\Delta}(\iota)_{\ast}\colon\mathrm{Fun}(\mathcal{D},N_{\Delta}(\underline{\mathbf{C}}))\rightarrow\mathrm{Fun}(\mathcal{D},N_{\Delta}(\mathbf{C}))\]
of quasi-categories.
\end{proposition}
\begin{proof}
By \cite[Corollary 5.5.8.8]{kerodon} the inclusion
$N_{\Delta}(\iota)\colon N_{\Delta}(\underline{\mathbf{C}})\rightarrow N_{\Delta}(\mathbf{C})$ is isomorphic to the inclusion
$\mathrm{Pith}(N_{\Delta}(\mathbf{C}))\hookrightarrow N_{\Delta}(\mathbf{C})$ as defined in \cite[Section 5.5.5]{kerodon}.
It follows that $N_{\Delta}(\iota)_{\ast}$ is an isomorphism by \cite[Remark 5.5.7.6]{kerodon}.
\end{proof}

\begin{example}\label{expleunderqcatCindqcat}
The underlying quasi-category of $\mathbf{QCat}$ is the quasi-category $\mathrm{Cat}_{\infty}$ of (small) quasi-categories. For any
quasi-category $\mathcal{C}$, the underlying quasi-category of $\mathbf{Fun}(\mathfrak{C}(\mathcal{C})^{op},\mathbf{QCat})$ is the 
quasi-category $\mathrm{Fun}(\mathcal{C}^{op},\mathrm{Cat}_{\infty})$. Analogously, as
$\underline{\mathbf{Fun}(\mathfrak{C}(\mathcal{C})^{op},\mathbf{Kan}})=\mathbf{Fun}(\mathfrak{C}(\mathcal{C})^{op},\mathbf{Kan})$, the 
underlying quasi-category of $\mathbf{Fun}(\mathfrak{C}(\mathcal{C})^{op},\mathbf{Kan})$ is the quasi-category
$\hat{\mathcal{C}}$ of presheaves over $\mathcal{C}$.
\end{example}

\begin{lemma}\label{lemmacartclosed}
The underlying quasi-category $\mathcal{C}:=N_{\Delta}(\underline{\mathbb{M}^{cf}})$ of a cartesian closed
$(\mathbf{S},\mathrm{QCat})$-enriched model category $\mathbb{M}$ is cartesian closed. The exponential of two bifibrant objects in
$\mathbb{M}$ represents the corresponding exponential in $\mathcal{C}$. In particular, the $\infty$-category
$\mathrm{Fun}(\mathcal{C}^{op},\mathrm{Cat}_{\infty})$ is cartesian closed for every small $\infty$-category $\mathcal{C}$.
\end{lemma}
\begin{proof}
Given an object $A\in\mathcal{C}$, we want to show that the functor $A\times(\cdot)\colon\mathcal{C}\rightarrow\mathcal{C}$ has a right 
adjoint. We may assume that $A\in\mathbb{M}$ is bifibrant. As the ordinary product functor
$A\times(\cdot)\colon\mathbb{M}\rightarrow\mathbb{M}$ is a (simplicial) left Quillen functor with right adjoint $(\cdot)^A$, it induces 
a left adjoint $\mathrm{Ho}_{\infty}(A\times(\cdot))\colon\mathcal{C}\rightarrow\mathcal{C}$ (see e.g.\
\cite[Theorem 2.1]{mazelgeequadj}; the more basic proof of \cite[Proposition 5.2.4.6]{luriehtt} applies to this context however as 
well). It thus suffices to show that the endofunctor
$\mathrm{Ho}_{\infty}(A\times(\cdot))\colon\mathcal{C}\rightarrow\mathcal{C}$ computes the $\infty$-categorical product
$A\times(\cdot)$. This however follows from the fact that the (simplicially enriched) product functor
$\times\colon\mathbb{M}\times\mathbb{M}\rightarrow\mathbb{M}$ is the right Quillen adjoint to the diagonal $\Delta\colon\mathbb{M}\rightarrow\mathbb{M}\times\mathbb{M}$.
The second statement follows directly from Lemma~\ref{lemmacqcatiscc}.
\end{proof}

Lastly, for any given quasi-category $\mathcal{C}$ there may be many $(\infty,2)$-categories $\mathbf{C}$ such that
$\mathcal{C}$ is the underlying quasi-category of $\mathbf{C}$. Each such induces an enhanced mapping
$\infty$-category functor on $\mathcal{C}$ in the sense of \cite{ghnlaxlimits}, whose definition we recall alongside its associated notions 
of tensors and cotensors.

\begin{definition}[{\cite[Definition 6.1]{ghnlaxlimits}}]\label{defenhmapqcats}
A mapping $\infty$-category functor for a quasi-category $\mathcal{C}$ is a functor
\[\mathrm{Map}_{\mathcal{C}}\colon\mathcal{C}^{op}\times\mathcal{C}\rightarrow\mathrm{Cat}_{\infty}\]
together with an equivalence from the composite
$\mathcal{C}^{op}\times\mathcal{C}\xrightarrow{\mathrm{Map}_{\mathcal{C}}}\mathrm{Cat}_{\infty}\xrightarrow{(\cdot)^{\simeq}}\mathcal{S}$ to the mapping space 
functor $\mathcal{C}(\phv,\phv)\colon\mathcal{C}^{op}\times\mathcal{C}\rightarrow\mathcal{S}$ of $\mathcal{C}$.
\end{definition}

Furthermore, we briefly introduce the concept of relative right (and left) adjoints in this context, and show that corresponding left 
(right) adjoints can be picked functorially whenever they exist. 
Therefore, let $F\colon\mathcal{D}\rightarrow\mathcal{E}$ be a functor between $\infty$-categories. Then, for any $\infty$-category 
$\mathcal{C}$ we obtain an associated functor defined as the following composition.
\[\mathcal{E}(F(\cdot),\blank(\cdot))\colon\mathrm{Fun}(\mathcal{C},\mathcal{E})\xrightarrow{y_{\ast}}\mathrm{Fun}(\mathcal{C},\hat{\mathcal{E}})\xrightarrow{(F^{\ast})_{\ast}}\mathrm{Fun}(\mathcal{C},\hat{\mathcal{D}})\xrightarrow{\simeq}\mathrm{Fun}(\mathcal{C}\times\mathcal{D}^{op},\mathcal{S})\xrightarrow{\simeq}\mathrm{Fun}(\mathcal{D}^{op},\mathrm{Fun}(\mathcal{C},\mathcal{S}))
\]
\begin{definition}
Let $F\colon\mathcal{D}\rightarrow\mathcal{E}$ be a functor between $\infty$-categories. A functor
$R\colon\mathcal{C}\rightarrow\mathcal{E}$ is an \emph{$F$-relative right adjoint} if for every object $D\in\mathcal{D}$, the copresheaf 
$\mathcal{E}(F(D),R(\cdot))\colon\mathcal{C}\rightarrow\mathcal{S}$ is corepresentable.
A functor $L\colon\mathcal{C}\rightarrow\mathcal{E}$ is an \emph{$F$-relative left adjoint} if
$L^{op}\colon\mathcal{C}^{op}\rightarrow\mathcal{E}^{op}$ is a $F^{op}$-relative right adjoint.
\end{definition}

As the Yoneda embedding $\mathcal{C}^{op}\rightarrow\mathrm{Fun}(\mathcal{C},\mathcal{S})$ is fully faithful, so is the push-forward
\[y_{\ast}\colon\mathrm{Fun}(\mathcal{D}^{op},\mathcal{C}^{op})\rightarrow\mathrm{Fun}(\mathcal{D}^{op},\mathrm{Fun}(\mathcal{C},\mathcal{S})).\]
It hence defines an equivalence between the $\infty$-category $\mathrm{Fun}(\mathcal{D}^{op},\mathcal{C}^{op})$
and the full sub-$\infty$-category $\mathrm{Fun}(\mathcal{D}^{op},y[\mathcal{C}^{op}])\subseteq\mathrm{Fun}(\mathcal{D}^{op},\mathrm{Fun}(\mathcal{C},\mathcal{S}))$ of functors
$K\colon\mathcal{D}^{op}\rightarrow\mathrm{Fun}(\mathcal{C},\mathcal{S})$ such that $K(D)\colon\mathcal{C}\rightarrow\mathcal{S}$ is 
corepresentable for all objects $D\in\mathcal{D}$. This defines the full sub-$\infty$-category
\[\xymatrix{
\mathrm{FRAdj}(\mathcal{C},\mathcal{E})\ar[rr]\ar@{}[d]|{\rotatebox[origin=c]{-90}{$\subseteq$}}\ar@{}[drr]|(.3){\pbs} & & \mathrm{Fun}(\mathcal{D}^{op},y[\mathcal{C}^{op}])\ar@{}[d]|{\rotatebox[origin=c]{-90}{$\subseteq$}} \\
\mathrm{Fun}(\mathcal{C},\mathcal{E})\ar[rr]_(.4){\mathcal{E}(F(\cdot),\blank(\cdot))} & & \mathrm{Fun}(\mathcal{D}^{op},\mathrm{Fun}(\mathcal{C},\mathcal{S}))
}\]
of $\mathcal{C}$-indexed $F$-relative right adjoints. The composition
\begin{align}\label{equdefpla}
\mathrm{FRAdj}(\mathcal{C},\mathcal{E})\rightarrow\mathrm{Fun}(\mathcal{D}^{op},y[\mathcal{C}^{op}])\xrightarrow{\simeq}\mathrm{Fun}(\mathcal{D},\mathcal{C})^{op}
\end{align}
associates to any given $F$-relative right adjoint functor $R\colon\mathcal{C}\rightarrow\mathcal{E}$ an associated $F$-left adjoint
$L\colon\mathcal{D}\rightarrow\mathcal{C}$. The case for $F$-relative left adjoints is entirely dual.

The following two definitions are straight-forward relativizations of \cite[Definition 6.5]{ghnlaxlimits} and \cite[Definition 8.2]{ghnlaxlimits}.

\begin{definition}\label{defenhmaptensors}
Suppose $\mathcal{C}$ is a quasi-category with a mapping $\infty$-category functor $\mathrm{Map}_{\mathcal{C}}$. Let 
$\iota\colon K\hookrightarrow\mathrm{Cat}_{\infty}$ be the inclusion of a full sub-quasi-category. We say that
$(\mathcal{C},\mathrm{Map}_{\mathcal{C}})$ is $K$-tensored
if for every object $C\in\mathcal{C}$ the functor $\mathrm{Map}_{\mathcal{C}}(C,\phv)\colon\mathcal{C}\rightarrow\mathrm{Cat}_{\infty}$ 
has a $\iota$-relative left adjoint $(\cdot)\otimes C\colon K\rightarrow\mathcal{C}$; in this case these adjoints determine an 
essentially unique functor $\otimes\colon K\times\mathcal{C}\rightarrow\mathcal{C}$ via the composition (\ref{equdefpla}).
\end{definition}

\begin{definition}\label{defenhmapcotensors}
Suppose $\mathcal{C}$ is a quasi-category with a mapping $\infty$-category functor $\mathrm{Map}_{\mathcal{C}}$. Let 
$\iota\colon K\hookrightarrow\mathrm{Cat}_{\infty}$ be the inclusion of a full sub-quasi-category. We say that
$(\mathcal{C},\mathrm{Map}_{\mathcal{C}})$ is $K$-cotensored
if for every object $C\in\mathcal{C}$ the functor
$\mathrm{Map}_{\mathcal{C}}(\phv,C)\colon\mathcal{C}\rightarrow\mathrm{Cat}_{\infty}^{op}$ has a $\iota$-relative right adjoint
$C^{(\cdot)}\colon K^{op}\rightarrow\mathcal{C}$; in this case these adjoints determine an essentially unique functor
$(\cdot)^{(\cdot)}\colon K^{op}\times\mathcal{C}\rightarrow\mathcal{C}$ via the composition (\ref{equdefpla}).
\end{definition}

\begin{example}\label{exple_cosm_cotensors}
As stated in \cite[Lemma 6.2]{ghnlaxlimits}, every $(\infty,2)$-category $\mathbf{C}$ canonically equips its underlying quasi-category
$\mathcal{C}$ with a mapping $\infty$-category functor that sends a pair $(C,D)$ to the quasi-category of maps $\mathbf{C}(C,D)$.
If $\mathbf{C}$ is an $\infty$-cosmos, then its simplicially enriched cotensors represent cotensors for this enhanced mapping functor in 
the sense of Definition~\ref{defenhmapcotensors} for $K=\mathrm{Cat}_{\infty}$.
If $\mathbf{D}$ is a finitary sub-$\infty$-cosmos of $\mathbf{C}$, then its simplicially enriched cotensors represent cotensors for this 
enhanced mapping functor in the sense of Definition~\ref{defenhmapcotensors} for $K=\mathrm{Cat}_{\infty}^{\mathrm{fin}}$ the full
sub-quasi-category spanned by the finite quasi-categories.
\end{example}

\begin{example}\label{exple_sub_cotensors}
Suppose $\iota\colon K\hookrightarrow\mathrm{Cat}_{\infty}$ is the inclusion of a full sub-quasi-category.
Suppose $\mathcal{C}$ is a $K$-(co)tensored $(\infty,2)$-category, and $\mathcal{D}$ is a full sub-$(\infty,2)$-category of $\mathcal{C}$. 
If for all $D\in\mathcal{D}$ and all $J\in K$ the object $J\otimes D$ ($D^J$) is again contained in $\mathcal{D}$, then $\mathcal{D}$ is 
also $K$-(co)tensored.
\end{example}

\subsection{Notes on model independence}\label{secsubformalmodelindep}

For the sake of the development of the general theory in Section~\ref{secformal}, the reader may choose to work in any 
of the standard models ``$\mathbf{Cat}_{\infty}$'' of $\infty$-category theory. More precisely, any
$\infty$-cosmos $\mathbf{Cat}_{\infty}$ of $(\infty,1)$-categories \cite[Section E.2]{riehlverityelements} will do. That is, any
$\infty$-cosmos $\mathbf{Cat}_{\infty}$ such that the underlying quasi-category functor
$U\colon\mathbf{Cat}_{\infty}\rightarrow\mathbf{QCat}$ is a cosmological equivalence. We will furthermore (for the most part implicitly) 
work with an $\infty$-cosmos $\mathbf{CAT}_{\infty}$ of (super)large $\infty$-categories, so that there is an $\infty$-category
$\mathrm{Cat}_{\infty}\in\mathbf{CAT}_{\infty}$ of small (large) $\infty$-categories such that $U(\mathrm{Cat}_{\infty})$ is equivalent to 
the underlying quasi-category of $\mathbf{Cat}_{\infty}$. Similarly, there is an $\infty$-category
$\mathcal{S}\in\mathbf{CAT}_{\infty}$ of small $\infty$-groupoids.
We will omit a formal proof of model independence of the results in this paper, and instead merely appeal to the observation that all 
constructions in Section~\ref{secformal} exist in (the underlying quasi-category of) any given model $\mathbf{Cat}_{\infty}$, and 
further are preserved by the equivalence $U$. Thus, in fact, to use all references to the literature directly, it is simplest and most 
straight-forward to work in the model $\mathbf{Cat}_{\infty}:=\mathbf{QCat}$ of quasi-categories itself (as well as with the
$\infty$-cosmos $\mathbf{CAT}_{\infty}:=\mathbf{QCAT}$ of large quasi-categories accordingly).
In the following, we will do so implicitly so to not over-indulge in ornamented notation by carrying around the underlying 
quasi-category functor $U\colon\mathbf{Cat}_{\infty}\rightarrow\mathbf{QCat}$ at all times.

\begin{notation*}
Accordingly, for an $\infty$-category $\mathcal{C}$ the
$\infty$-cosmos $\mathbf{Fun}(\mathcal{C}^{op},\mathbf{Cat}_{\infty})$ will denote the $\infty$-cosmos
$\mathbf{Fun}(\mathfrak{C}(\mathcal{C})^{op},\mathbf{QCat})$.\\
\end{notation*}

The choice of the homotopy theory of quasi-categories as an ambient theory of $(\infty,1)$-categories is legitimized by the fact that  
it presents (and even qualifies) the homotopy theory of $\infty$-categories.
In the same sense, the homotopy theory of complete Segal spaces presents (and qualifies) the homotopy theory of internal
$\infty$-categories in the $\infty$-category $\mathcal{S}$ of spaces.  Consequently, the homotopy theory of complete Segal objects in an 
$\infty$-category $\mathcal{C}$ as to be studied in Section~\ref{secformal} presents (and qualifies) the theory of internal
$\infty$-categories in $\mathcal{C}$. This is the presentation of the theory of internal $\infty$-categories this paper is primarily 
concerned with. There however may be other (equivalent) models of internal $\infty$-category 
theory within any given ambient $\infty$-cosmos of $(\infty,1)$-categories. One may for instance synthesize other 
models of $\infty$-category theory within a general base $\mathcal{C}$, take the notion of Segal categories 
\cite{dkshcdiags} for example. We therefore will largely stick with the explicit terminology of ``complete Segal objects'' rather than 
that of ``internal $\infty$-categories'' in technical contexts, as there is no apparent reason why all such presentations would 
exhibit all (not necessarily equivalence-invariant) structural results we aim to prove to the same extent. This may be compared to the 
fact that not all presentations of the theory of $(\infty,1)$-categories are equally well-behaved; take the category of simplicial 
categories with the Bergner model structure, or even the category of relative categories equipped with the Barwick-Kan model structure 
\cite{bkrelcat} for example.

\section{The formal theory of internal \texorpdfstring{$\infty$-}{higher }categories}\label{secformal}

In this section we fix an $\infty$-category $\mathcal{C}$. We recall the following associated notions.

\begin{notation}
We denote the $\infty$-category $\mathrm{Fun}(N(\Delta^{op}),\mathcal{C})$ of 
simplicial objects in $\mathcal{C}$ by $s\mathcal{C}$. For $n\geq 0$ and a subset $J\subseteq [n]$ of 
cardinality $j$, we denote by $d^J\colon [j]\rightarrow [n]$ the corresponding inclusion of linear orders with image $J$, and for a simplicial 
object $X\in s\mathcal{C}$, by $d_J\colon X_n\rightarrow X_{j}$ the corresponding simplicial operator. 
\end{notation}

Given that $(\infty,2)$-category theory is the theory of $\mathrm{Cat}_{\infty}$-enriched $\infty$-categories, it is at times 
useful to think of $\infty$-category theory itself as the theory of $\mathcal{S}$-enriched categories. In this context, we briefly recall 
that functors $W\colon\mathcal{I}\rightarrow\mathcal{S}$ out of a small $\infty$-category $\mathcal{I}$ are sometimes referred to as
($\mathcal{I}$-indexed) weights. Given a diagram $F\colon\mathcal{I}\rightarrow\mathcal{C}$, the $\mathcal{S}$-weighted limit
$\{W,F\}\in\mathcal{C}$ is given by the ordinary limit of the composition
\[\mathrm{Un}(W)\overset{\pi_W}{\twoheadrightarrow}\mathcal{I}\xrightarrow{F}\mathcal{C}\]
whenever it exists. Here, the functor $\pi_W$ denotes the left fibration associated to $W$ by way of the Unstraightening construction
\cite[Section 3.2]{luriehtt}. We refer to \cite{rovelliweights} for a more thorough study of $\mathcal{S}$-weighted limits and colimits 
(\cite[Theorem 2.33]{rovelliweights} in particular).

\begin{example}\label{exple_weightedhomspace}
Suppose $\mathcal{I}$ is a small $\infty$-category, and $F,G\colon\mathcal{I}^{op}\rightarrow\mathcal{S}$ are presheaves. Then
\begin{align}
\notag \{F,G\} & \simeq\mathrm{lim}\left( \mathrm{Un}(F)\overset{\pi_F}{\twoheadrightarrow}\mathcal{I}^{op}\xrightarrow{G}\mathcal{C} \right) \\
\label{exple_weightedhomspace_1}
 & \simeq \mathrm{lim}\left(\mathrm{Un}(F)\overset{\pi_F}{\twoheadrightarrow}\mathcal{I}^{op}\xrightarrow{y^{op}}\hat{\mathcal{I}}^{op}
\xrightarrow{\hat{\mathcal{I}}(\phv,G)}\mathcal{C} \right) \\
\label{exple_weightedhomspace_2}
 & \simeq \hat{\mathcal{I}}(\mathrm{colim}\left(\mathrm{Un}(F)^{op}\overset{\pi_F^{op}}{\twoheadrightarrow}\mathcal{I}\xrightarrow{y}\hat{\mathcal{I}}\right),G) \\
\label{exple_weightedhomspace_3}
 & \simeq \hat{\mathcal{I}}(F,G).
\end{align}
Here, Line~(\ref{exple_weightedhomspace_1}) follows directly the Yoneda lemma. Line~(\ref{exple_weightedhomspace_2}) follows from the fact 
that the representable $\hat{\mathcal{I}}(\phv,G)$ preserves limits. In Line~(\ref{exple_weightedhomspace_2}), the functor
$\pi_F^{op}$ is the right fibration associated to $F$ under the Unstraightening construction. In other words, $\mathrm{Un}(F)$ is the
$\infty$-category of elements of $F$. Hence, $F$ is the colimit of the composition $y\circ\pi_F^{op}$ as the Yoneda embedding generates
$\hat{\mathcal{I}}$ under colimits \cite[Lemma 5.1.5.3]{luriehtt}.
Thus, the weighted limit $\{F,G\}\in\mathcal{S}$ computes the mapping space $\hat{\mathcal{I}}(F,G)$ of the presheaf $\infty$-category
$\hat{\mathcal{I}}$.
\end{example}

From Example~\ref{exple_weightedhomspace} one can directly derive the universal property of weighted limits.

\begin{lemma}\label{lemma_weights}
Let $F\colon\mathcal{I}\rightarrow\mathcal{C}$ be a diagram and $W\colon\mathcal{I}\rightarrow\mathcal{C}$ be a weight such that the weighted limit $\{W,F\}\in\mathcal{C}$ exists. Then for every object $C\in\mathcal{C}$ there is a natural equivalence
\[\mathcal{C}(C,\{W,F\})\simeq\hat{\mathcal{I}}(W,\mathcal{C}(C,F(\cdot)))\]
of spaces.
\end{lemma}

\begin{proof}
There are natural equivalences as follows.
\begin{align}
\notag \mathcal{C}(C,\{W,F\}) & \simeq \mathcal{C}(C,\mathrm{lim}\left(\mathrm{Un}(W)\overset{\pi_W}{\twoheadrightarrow}\mathcal{I}\xrightarrow{F}\mathcal{C}\right)) \\
\label{equ_univpropweight1} & \simeq \mathrm{lim}\left(\mathrm{Un}(W)\overset{\pi_W}{\twoheadrightarrow}\mathcal{I}\xrightarrow{F}\mathcal{C}\xrightarrow{\mathcal{C}(C,\phv)}\mathcal{S}\right)\\
\notag  & \simeq\{W,\mathcal{C}(C,F(\cdot))\} \\
\label{equ_univpropweight2} & \simeq \hat{\mathcal{I}}(W,\mathcal{C}(C,F(\cdot)))
\end{align}
Line~(\ref{equ_univpropweight1}) follows from the fact that corepresentable presheaves preserve all limits. 
Line~(\ref{equ_univpropweight2}) is an instance of Example~\ref{exple_weightedhomspace}.
\end{proof}

\begin{example}\label{exmple_ends}
A weighted limit construction that will be central to the calculations in this paper is given by the end
$\int_{\mathcal{I}}F:=\{\mathrm{Tw}(\mathcal{I}),F\}$ associated to a functor
$F\colon\mathcal{I}^{op}\times\mathcal{I}\rightarrow\mathcal{C}$. Here, the corresponding weight is given by the twisted arrow
$\infty$-category $\mathrm{Tw}(\mathcal{I})\twoheadrightarrow\mathcal{I}^{op}\times\mathcal{I}$ of $\mathcal{I}$
\cite[Definition 2.6]{ghnlaxlimits}. For instance, for any pair of functors $F,G\colon\mathcal{I}\rightarrow\mathcal{C}$ between
$\infty$-categories, the hom-space $\mathrm{Fun}(\mathcal{I},\mathcal{C})(F,G)$ of natural transformations from $F$ to $G$ is is naturally 
equivalent to the end $\int_{\mathcal{I}}\mathcal{C}(F(\cdot),G(\cdot))$ \cite[Proposition 5.1]{ghnlaxlimits}.
\end{example}

\subsection{The \texorpdfstring{$\infty$-}{higher }category of internal \texorpdfstring{$\infty$-}{higher }categories}\label{subsecinf1}

In the following we give a straight-forward generalization of the definition of complete Segal objects in left exact
$\infty$-categories from \cite{rasekhunivalence, rasekhcart}. Therefore, let $\sigma_n\colon S^n\hookrightarrow\Delta^n$ be the
$n$-spine \cite{jtqcatvsss}. Considering $\sigma_n$ as a natural transformation of presheaves over the category $\Delta$, we obtain a 
fibered inclusion
\[\xymatrix{
\mathrm{El}(S_n)\ar@{^(->}[rr]^{(\sigma_n)_{\ast}}\ar@/_/@{->>}[dr]_{p_{S_n}} & & \mathrm{El}(\Delta^n)\ar@/^/@{->>}[dl]^{p_{\Delta_n}} \\
 & \Delta & 
}\]
of the corresponding categories of elements over $\Delta$. Since $\Delta^n$ is the representable at $[n]\in\Delta$, the fibration
$p_{\Delta^n}$ is just the domain fibration $s\colon\Delta_{/[n]}\twoheadrightarrow\Delta$. In particular
$\mathrm{id}_{[n]}\in\mathrm{El}(\Delta^n)$ is terminal, and so for every functor $X\colon\Delta\rightarrow\mathcal{C}^{op}$ we obtain a 
colimiting cocone $X\circ p_{\Delta^n}\rightarrow c(X_n)$ in $\mathcal{C}^{op}$. Here,
$c(X_n)\colon \mathrm{El}(\Delta^n)\rightarrow\mathcal{C}^{op}$ denotes the constant functor with value $X_n$.
Restriction of this cocone along the inclusion $(\sigma_n)_{\ast}$ (and taking opposites) yields a cone 
$X_n\rightarrow X\circ p_{S_n}^{op}$ for every simplicial object $X\colon\Delta^{op}\rightarrow\mathcal{C}$.

\begin{definition}\label{defSegalobjects}
A Segal object in an $\infty$-category $\mathcal{C}$ is a simplicial object $X\in s\mathcal{C}$ such that its associated cone
$X_n\rightarrow X\circ p_{S_n}^{op}$ is a limit cone. We let $\mathrm{S}(\mathcal{C})\subset s\mathcal{C}$ denote the full
sub-$\infty$-category of Segal objects in $\mathcal{C}$.
\end{definition}

Equivalently, a simplicial object $X\in\mathcal{C}^{\Delta^{op}}$ is a Segal object if the cone $X_n\rightarrow X\circ p_{S_n}^{op}$ 
exhibits $X_n$ as the $\mathcal{S}$-weighted limit $\{S_n,X\}$ for the weight
$S_n\colon\Delta^{op}\rightarrow\mathrm{Set}\hookrightarrow\mathcal{S}$. 

In the same vein, let $\mathrm{Equiv}\colon\Delta^{op}\rightarrow\mathrm{Set}$ be the following version of the free living biinvertible 
arrow (with no implicit equations imposed on its edges).
\[\xymatrix{
 & 1\ar@/^/[dr]\ar@/^/@{=}[rr] & & 1\\
0\ar@/_/@{=}[rr]\ar@/^/[ur]\ar@{-->}[urrr] & & 0\ar@/_/[ur] & 
}\]
The terminal map $!\colon\mathrm{Equiv}\rightarrow\Delta^0$ again induces a functor of associated fibrations as follows.
\[\xymatrix{
\mathrm{El}(\mathrm{Equiv})\ar[rr]^{(!)_{\ast}}\ar@/_/@{->>}[dr]_{p_{\mathrm{Equiv}}} & & \mathrm{El}(\Delta^0)\ar@/^/@{->>}[dl]^{p_{\Delta_0}} \\
 & \Delta & 
}\]
For every simplicial object $X\colon\Delta^{op}\rightarrow\mathcal{C}$, we obtain a cone $X_0\rightarrow X\circ p_{\mathrm{Equiv}}^{op}$.

\begin{definition}\label{defcompleteness}
A Segal object $X$ in an $\infty$-category $\mathcal{C}$ is complete if its associated cone $X_0\rightarrow X\circ p_{\mathrm{Equiv}}^{op}$ is a limit cone. 
We let $\mathrm{Cat}_{\infty}(\mathcal{C})\subset s\mathcal{C}$ denote the full
sub-$\infty$-category of complete Segal objects in $\mathcal{C}$.
\end{definition}

Equivalently, a Segal object $X\in\mathcal{C}^{\Delta^{op}}$ is complete if the cone
$X_0\rightarrow X\circ p_{\mathrm{Equiv}}^{op}$ exhibits $X_0$ as the $\mathcal{S}$-weighted limit 
$\{\mathrm{Equiv},X\}$ for the weight $\mathrm{Equiv}\colon\Delta^{op}\rightarrow\mathrm{Set}\hookrightarrow\mathcal{S}$. 

\begin{remark}\label{remcsspb}
Whenever $\mathcal{C}$ has pullbacks, this definition of complete Segal objects coincides with the more familiar definition of complete 
Segal objects e.g.\ from \cite[Definition 7.91]{rasekhcart}. Indeed, in this case the morphisms
$\{\sigma_n,X\}\colon\{\Delta^n,X\}\rightarrow\{S_n,X\}$ between weighted limits are precisely the Segal maps 
$X_n\rightarrow X_1\times_{X_0}\dots\times_{X_0} X_1$ associated to $X$. Furthermore, let $\mathrm{Zig}\text{-}\mathrm{zag}$ be the 
category representing the free living zig-zag of the following form.
\begin{align*}
\mathrm{Zig}\text{-}\mathrm{zag}:= & 
\begin{gathered}
\xymatrix{
 & 1\ar@/^/[dr]\ar@/^/[rr] & & 3\\
0\ar@/_/[rr] & & 2 & 
}
\end{gathered}
\end{align*}
Whenever $\mathcal{C}$ has pullbacks, to every Segal object $X$ in $\mathcal{C}$ (in fact to every $X\in s\mathcal{C}$) 
we may associate, first, the object
$\mathrm{Zig}\text{-}\mathrm{zag}(X):=\{\mathrm{Zig}\text{-}\mathrm{zag},X\}$ given by the pullback
$X_1\prescript{d_1}{}{\times}_{X_0}^{d_1}X_1\prescript{d_0}{}{\times}_{X_0}^{d_0}X_1$ of zig-zags in $X$, and second, the object
$\mathrm{Equiv}(X)\subseteq X_3$ of internal equivalences in $X$ (or more precisely the object of edges together with a left and a right 
inverse in $X$) defined as the following pullback.
\[\xymatrix{
\mathrm{Equiv}(X)\ar[r]\ar[d]\ar@{}[dr]|(.3){\pbs} & X_3\ar[d]^{(d_{\{0,2\}},d_{\{1,2\}},d_{\{1,3\}})} \\
X_1\ar[r]_(.35){(s_0d_0,1,s_0d_1)} & \mathrm{Zig}\text{-}\mathrm{zag}(X)
}\]
Then $\mathrm{Equiv}(X)\simeq\{\mathrm{Equiv},X\}$, and the cone in Definition~\ref{defcompleteness} corresponds precisely to the 
accordingly factorized degeneracy $s_0\colon X_0\rightarrow\mathrm{Equiv}(X)$.
\end{remark}

For all what follows, we will assume that the $\infty$-category $\mathcal{C}$ is locally small.
In this case, the Yoneda embedding $y\colon\mathcal{C}\rightarrow\hat{\mathcal{C}}$ induces 
a functor $sy\colon s\mathcal{C}\rightarrow s\hat{\mathcal{C}}$ by postcomposition. The $\infty$-category $s\hat{\mathcal{C}}$ of 
simplicial objects in turn is equivalent to the $\infty$-category $\mathrm{Fun}(\mathcal{C}^{op},s\mathcal{S})$ of $\mathcal{C}$-indexed 
simplicial spaces simply by virtue of Currying. As the Yoneda embedding is left exact, it restricts to functors
$sy\colon\mathrm{S}(\mathcal{C})\rightarrow\mathrm{S}(\hat{\mathcal{C}})$ and
$sy\colon\mathrm{Cat}_{\infty}(\mathcal{C})\rightarrow\mathrm{Cat}_{\infty}(\hat{\mathcal{C}})$.

\begin{lemma}\label{lemmasegalinternal}
Let $X\in s\mathcal{C}$ be a simplicial object. Then the following are equivalent.
\begin{enumerate}
\item  $X$ is a (complete) Segal object in $\mathcal{C}$.
\item $sy(X)$ is a (complete) Segal object in $\hat{\mathcal{C}}$.
\item $X$ is representably a (complete) Segal space, i.e.\ for all $C\in\mathcal{C}$ the simplicial space $sy(X)(C)$ is a (complete) Segal 
space.
\end{enumerate}
\end{lemma}
\begin{proof}
This following directly from the fact that the Yoneda embedding is both left exact and conservative, together with the fact that 
both limits and equivalences in functor $\infty$-categories are computed pointwise.

\end{proof}

\begin{definition}
A Segal groupoid in $\mathcal{S}$ is a Segal space $X$ such that the structure map $d_{\{1,2\}}\colon\mathrm{Equiv}(X)\rightarrow X_1$ 
is an equivalence in $\mathcal{S}$. A Segal groupoid in $\mathcal{C}$ is a Segal object $X$ in
$\mathcal{C}$ such that each $sy(X)(C)$ is a Segal groupoid in $\mathcal{S}$. We denote the full sub-$\infty$-category of Segal groupoids
in $\mathrm{S}(\mathcal{C})$ by $\mathrm{G}(\mathcal{C})$. We denote the full sub-$\infty$-category of complete Segal groupoids (say, 
internal $\infty$-groupoids) in $\mathrm{S}(\mathcal{C})$ by $\mathrm{Gpd}_{\infty}(\mathcal{C})$. 
\end{definition}

Again by virtue of left exactness and conservativity of the Yoneda embedding, a (complete) Segal object $X$ in a locally small
$\infty$-category $\mathcal{C}$ with pullbacks is a (complete) Segal groupoid if and only if the morphism
$d_{\{1,2\}}\colon\mathrm{Equiv}(X)\rightarrow X_1$ in $\mathcal{C}$ is an equivalence.

Given any simplicial set $K$, the $\infty$-category $s\mathcal{C}$ of simplicial objects in $\mathcal{C}$ has all $K$-shaped limits 
whenever $\mathcal{C}$ does so. Furthermore, these limits are computed pointwise in $\mathcal{C}$ \cite[Corollary 5.1.2.3]{luriehtt}. In 
particular, $s\mathcal{C}$ is left exact whenever $\mathcal{C}$ is.

\begin{lemma}\label{lemmascolimits}
The full sub-$\infty$-categories $\mathrm{S}(\mathcal{C})$, $\mathrm{Cat}_{\infty}(\mathcal{C})$, $\mathrm{G}(\mathcal{C})$ and
$\mathrm{Gpd}_{\infty}(\mathcal{C})$ of $s\mathcal{C}$ are each closed under all limits that exist in $s\mathcal{C}$. 
In particular, they are all left exact (complete) whenever $\mathcal{C}$ is left exact (complete).
\end{lemma}
\begin{proof}
Each of these full sub-$\infty$-categories of $s\mathcal{C}$ is spanned by a class of objects defined by requiring that a certain cone 
defined in terms of structure maps coming from $\Delta^{op}$ is limiting in $\mathcal{C}$. Since limits in $s\mathcal{C}$ are computed 
pointwise, and limits in $\mathcal{C}$ commute with one another, these objects are closed under all limits in $s\mathcal{C}$.

Alternatively, one can argue that each of these notions can be defined representably, and that the 
full sub-$\infty$-categories $\mathrm{S}(\mathcal{S})$, $\mathrm{Cat}_{\infty}(\mathcal{S})$, $\mathrm{G}(\mathcal{S})$ and
$\mathrm{Gpd}_{\infty}(\mathcal{S})$ of $s\mathcal{S}$ are all reflective and hence complete.
\end{proof}

\subsection{The \texorpdfstring{$(\infty,2)$-}{higher 2-}category of internal \texorpdfstring{$\infty$-}{higher }categories}\label{subsecext}

We continue assuming that $\mathcal{C}$ is locally small. Then the Yoneda embedding
$sy\colon s\mathcal{C}\rightarrow s\hat{\mathcal{C}}$ further restricts to functors
\begin{align}\label{equ_ext_yon}
sy\colon\mathrm{Cat}_{\infty}(\mathcal{C})\rightarrow\mathrm{Cat}_{\infty}(\hat{\mathcal{C}})\simeq\mathrm{Fun}(\mathcal{C}^{op},
\mathrm{Cat}_{\infty}(\mathcal{S}))
\end{align}
from complete Segal objects  in $\mathcal{C}$ to complete Segal objects in $\hat{\mathcal{C}}$ (or, equivalently, to
$\mathcal{C}$-indexed complete Segal spaces, respectively). The same applies to (complete) Segal groupoids.
Lastly, the $\infty$-category $\mathrm{Cat}_{\infty}(\mathcal{S})$ of complete Segal spaces exhibits an equivalence to the
$\infty$-category $\mathrm{Cat}_{\infty}$ of $\infty$-categories given by the ``underlying $\infty$-category'' functor
\[U\colon\mathrm{Cat}_{\infty}(\mathcal{S})\xrightarrow{\simeq}\mathrm{Cat}_{\infty}.\]
It was constructed in \cite[Section 4]{jtqcatvsss} as an accordingly right derived horizontal projection of simplicial spaces.
The same horizontal projection induces an equivalence
$U\colon\mathrm{Gpd}_{\infty}(\mathcal{S})\xrightarrow{\simeq}\mathcal{S}$ between complete Segal groupoids in
$\infty$-groupoids and $\infty$-groupoids, see \cite[Section 6]{bergner2} and \cite[Theorem 6.6]{rs_bspaces}.

\begin{definition}\label{defext}
The externalization functor $\mathrm{Ext}\colon\mathrm{Cat}_{\infty}(\mathcal{C})\rightarrow\mathrm{Fun}(\mathcal{C}^{op},\mathrm{Cat}_{\infty})$ associated to $\mathcal{C}$ is given by the composition
\[\mathrm{Cat}_{\infty}(\mathcal{C})\xrightarrow{sy}\mathrm{Fun}(\mathcal{C}^{op},\mathrm{Cat}_{\infty}(\mathcal{S}))\underset{\simeq}{\xrightarrow{U_{\ast}}}\mathrm{Fun}(\mathcal{C}^{op},\mathrm{Cat}_{\infty}).\]
\end{definition}

The externalization functor associated to $\mathcal{C}$ extends the Yoneda embedding associated to $\mathcal{C}$ in the sense that the 
following diagram can be shown to commute in both directions.

\begin{align}\label{diagyonedaext}
\begin{gathered}
\xymatrix{
\mathcal{C}\ar@{^(->}[r]^(.35)y\ar@/_/@{^(->}[d]_{c} & \mathrm{Fun}(\mathcal{C}^{op},\mathcal{S})\ar@/_/@{^(->}[d] \\
\mathrm{Cat}_{\infty}(\mathcal{C})\ar[r]_(.35){\mathrm{Ext}}\ar@/_/[u]_{(\cdot)_0} & \mathrm{Fun}(\mathcal{C}^{op},\mathrm{Cat}_{\infty})\ar@/_/[u]_{(\cdot)^{\simeq}}
}
\end{gathered}
\end{align}
Here, the left vertical inclusion $c$ assigns to an object $C\in\mathcal{C}$ the constant simplicial object in $\mathcal{C}$ with value
$C$, and the right vertical inclusion is the push-forward with the canonical inclusion
$\mathcal{S}\hookrightarrow\mathrm{Cat}_{\infty}$. The inclusion
$c\colon\mathcal{C}\hookrightarrow \mathrm{Cat}_{\infty}(\mathcal{C})$ furthermore factors through an equivalence
$c\colon\mathcal{C}\xrightarrow{\sim}\mathrm{Gpd}_{\infty}(\mathcal{C})$, given that a simplicial object
$X\in\mathcal{C}^{\Delta^{op}}$ is a complete Segal groupoid if and only if it is constant.

\begin{definition}\label{defsmallness}
A $\mathcal{C}$-indexed $\infty$-category $\mathcal{E}$ is \emph{small} if it lies in the essential image of the externalization 
functor, i.e.\ if there is a complete Segal object $X\in\mathrm{Cat}_{\infty}(\mathcal{C})$ such that
$\mathrm{Ext}(X)\simeq\mathcal{E}$. We denote the full sub-$(\infty,2)$-category of
$\mathbf{Fun}(\mathcal{C}^{op},\mathbf{Cat}_{\infty})$ spanned by the small $\mathcal{C}$-indexed $\infty$-categories by
$\mathbf{Cat}_{\infty}(\mathcal{C})$, and refer to it as the $(\infty,2)$-category of internal $\infty$-categories in $\mathcal{C}$.
\end{definition}

We recall that if $\mathcal{C}$ is large then the $(\infty,2)$-category $\mathbf{Fun}(\mathcal{C}^{op},\mathbf{Cat}_{\infty})$ is defined 
by way of Example~\ref{exple_size}. Definition~\ref{defsmallness} is justified by the following straight-forward lemma.

\begin{lemma}\label{corsegalyonedaff}
The externalization functor
$\mathrm{Ext}\colon\mathrm{Cat}_{\infty}(\mathcal{C})\rightarrow\mathrm{Fun}(\mathcal{C}^{op},\mathrm{Cat}_{\infty})$ is fully faithful 
and preserves all limits that exist in $\mathrm{Cat}_{\infty}(\mathcal{C})$. In particular,
$\mathrm{Cat}_{\infty}(\mathcal{C})$ is the underlying $\infty$-category of $\mathbf{Cat}_{\infty}(\mathcal{C})$ as defined in 
Section~\ref{secsubformalunderqcat}. 
\end{lemma}
\begin{proof}
The Yoneda embedding $y\colon\mathcal{C}\rightarrow\hat{\mathcal{C}}$ is fully faithful and hence so is the push-forward
$y_{\ast}\colon\mathrm{Fun}(I,\mathcal{C})\rightarrow\mathrm{Fun}(I,\hat{\mathcal{C}})$ for any $\infty$-category $I$. Both
$\mathrm{Cat}_{\infty}(\mathcal{C})\subset s\mathcal{C}$ and $\mathrm{Fun}(\mathcal{C}^{op},\mathrm{Cat}_{\infty}(\mathcal{S}))\subset\mathrm{Fun}(\mathcal{C}^{op},s\mathcal{S})$ are full sub-$\infty$-categories, and so the restriction
$sy\colon\mathrm{Cat}_{\infty}(\mathcal{C})\rightarrow\mathrm{Fun}(\mathcal{C}^{op},\mathrm{Cat}_{\infty}(\mathcal{S}))$ is fully 
faithful, too. Clearly, postcomposition with equivalences preserves fully faithfulness, and so $\mathrm{Ext}$ is fully faithful itself.
\end{proof}

\begin{remark}
In the homotopy theory of complete Segal spaces as a model for $\infty$-category theory, the fact that the 
externalization functor is fully faithful is effectively stated in \cite[Theorem 5.34]{rasekhcart}
(see \cite[Remark 5.8]{rs_comp}). It also can be regarded as a corollary of the Segal--Yoneda Lemma (Proposition~\ref{lemmaextyoneda}) 
stated below, just as fully faithfulness of the Yoneda embedding is a corollary of the Yoneda lemma.
\end{remark}

\begin{remark}\label{rem_intcatstandard}
The $\infty$-category $\mathrm{Cat}_{\infty}(\mathcal{S})$ of complete Segal spaces is the underlying $\infty$-category of an
$(\infty,2)$-category $\mathbf{Cat}_{\infty}(\mathcal{S})$ such that the equivalence
$U\colon\mathrm{Cat}_{\infty}(\mathcal{S})\xrightarrow{\simeq}\mathrm{Cat}_{\infty}$ gives rise to an equivalence
\[U\colon\mathbf{Cat}_{\infty}(\mathcal{S})\xrightarrow{\simeq}\mathbf{Cat}_{\infty}\]
of $(\infty,2)$-categories \cite[Proposition E.2.2]{riehlverityelements}. Because $U$ is in fact a right Quillen equivalence of
$(\mathbf{S},\mathrm{QCat})$-enriched model categories, it induces an equivalence 
\[U_{\ast}\colon\mathbf{Fun}(\mathcal{C}^{op},\mathbf{Cat}_{\infty}(\mathcal{S}))\xrightarrow{\simeq}\mathbf{Fun}(\mathcal{C}^{op},\mathbf{Cat}_{\infty})\]
of associated diagram $(\infty,2)$-categories for any $\infty$-category $\mathcal{C}$. This is standard theory whenever $\mathcal{C}$ is 
small; if $\mathcal{C}$ is large, this holds still via the detour in Example~\ref{exple_size}. Hence,
$\mathbf{Cat}_{\infty}(\mathcal{C})$ might just as well be defined directly as the full sub-$(\infty,2)$-category of
$\mathbf{Fun}(\mathcal{C}^{op},\mathbf{Cat}_{\infty}(\mathcal{S}))$ spanned by the image of the embedding (\ref{equ_ext_yon}).
\end{remark}

\begin{example}\label{exple_smallspaces}
For $\mathcal{C}=\mathcal{S}$, the $(\infty,2)$-category
$\mathbf{Cat}_{\infty}(\mathcal{S})$ of complete Segal spaces obtained by way of Definition~\ref{defsmallness} is the standard
$(\infty,2)$-category of complete Segal spaces as obtained from Remark~\ref{rem_intcatstandard} applied to the bottom Grothendieck universe 
of the hierarchy of Notation~\ref{notation_univ}. Indeed, the diagram
\[\xymatrix{
\mathrm{Cat}_{\infty}(\mathcal{S})\ar[r]^(.4){sy} \ar[d]_U^{\simeq} & \mathrm{Fun}(\mathcal{S}^{op},\mathrm{Cat}_{\infty}(\mathcal{S}))\ar[r]^(.55){U_{\ast}}_(.55){\simeq} & \mathrm{Fun}(\mathcal{S}^{op},\mathrm{Cat}_{\infty})\ar[d]^{\mathrm{res_{\ast}}} \\
\mathrm{Cat}_{\infty}\ar[rr]_{\simeq} & & \mathrm{Fun}(\ast,\mathrm{Cat}_{\infty}) 
}\]
of $\infty$-categories commutes trivially, where $\mathrm{res}_{\ast}$ is restriction along the Yoneda embedding
$\{\ast\}\colon\ast\rightarrow\mathcal{S}^{op}$. Both the right vertical functor $\mathrm{res}_{\ast}$ as well as the bottom horizontal embedding enhance canonically to yield functors of the corresponding $(\infty,2)$-categories. 
In particular, the $(\infty,2)$-categorical structure $\mathbf{Cat}_{\infty}(\mathcal{S})$ from Definition~\ref{defsmallness} enhances the 
vertical forgetful functor $U$ to an equivalence
$U\colon\mathbf{Cat}_{\infty}(\mathcal{S})\rightarrow\mathbf{Cat}_{\infty}$ of $(\infty,2)$-categories as well.
An analogous argument shows that for any small $\infty$-category $\mathcal{C}$ the canonical equivalence
$\mathrm{Cat}_{\infty}(\hat{\mathcal{C}})\simeq\mathrm{Fun}(\mathcal{C}^{op},\mathrm{Cat}_{\infty}(\mathcal{S}))$ induces an equivalence
$\mathbf{Cat}_{\infty}(\hat{\mathcal{C}})\simeq\mathbf{Fun}(\mathcal{C}^{op},\mathbf{Cat}_{\infty})$ of $(\infty,2)$-categories. Here, the 
structure on the left hand side is again obtained by way of Definition~\ref{defsmallness}, the structure on the right hand side is the 
standard pointwise induced one.
\end{example}

The ordinary categorical externalization of a locally small 1-category is defined and studied in \cite[Section 4]{streetintcat}, and later 
in \cite[Section 7.3]{jacobsttbook} and \cite[Section B.2.3]{elephant}. By \cite[Proposition 5.10]{rs_comp}, the $\infty$-categorical 
externalization functor in Definition~\ref{defext} is an essentially faithful generalization thereof. Indeed, we will see that the 
properties of the ordinary categorical externalization functor shown in \cite[Paragraphs (4.4) and (4.5)]{streetintcat} (and later in
in \cite[Proposition 7.3.8]{jacobsttbook}) generalize accordingly, and that they in fact can be strengthened. 
For instance, in \cite[Paragraph (4.4)]{streetintcat} it is stated that the externalization functor creates
(finite) 2-limits whenever $\mathcal{C}$ is complete (left exact). We already have seen that $\mathrm{Ext}$ is left exact, and that
$\mathrm{Cat}_{\infty}(\mathcal{C})$ itself is left exact whenever $\mathcal{C}$ is so. Moreover, we have the following proposition.

\begin{proposition}\label{proppowering}
Suppose $\mathcal{C}$ has all small (finite) limits. Then the embedding
$\mathrm{Ext}\colon\mathbf{Cat}_{\infty}(\mathcal{C})\hookrightarrow\mathbf{Fun}(\mathcal{C}^{op},\mathbf{Cat}_{\infty})$ creates
cotensors with all small (finite) $\infty$-categories. That is to say, the full sub-$(\infty,2)$-category
$\mathbf{Cat}_{\infty}(\mathcal{C})\subset\mathbf{Fun}(\mathcal{C}^{op},\mathbf{Cat}_{\infty})$ is closed under cotensors with all  
small (finite) $\infty$-categories. For any small (finite) $\infty$-category $J$ and any complete Segal object
$X\in\mathrm{Cat}_{\infty}(\mathcal{C})$, the cotensor $X^J\in\mathrm{Cat}_{\infty}(\mathcal{C})$ is given by the ``d\'{e}calage'' 
\[\{((J\times\Delta^{\ast})^{\Delta^{\bullet}})^{\simeq},X\}\]
in $\mathcal{C}$. That is, at level $n$, the
$\mathcal{S}$-weighted limit $(X^J)_n\simeq\{((J\times\Delta^{n})^{\Delta^{\bullet}})^{\simeq},X\}$ of $X$ at the weight
$((J\times\Delta^{n})^{\Delta^{\bullet}})^{\simeq}\colon\Delta^{op}\rightarrow\mathcal{S}$. 
\end{proposition}
\begin{proof}
If $J$ is small (finite) then the bisimplicial space $((J\times\Delta^{\ast})^{\Delta^{\bullet}})^{\simeq}$ is again pointwise small 
(finite); it follows that the simplicial object $\{((J\times\Delta^{\ast})^{\Delta^{\bullet}})^{\simeq},X\}$ in $\mathcal{C}$ exists as 
well. We show that $\{((J\times\Delta^{\ast})^{\Delta^{\bullet}})^{\simeq},X\}$ is a complete Segal object, and show that it satisfies the 
universal property of the cotensor $X^J$ in $\mathbf{Cat}_{\infty}(\mathcal{C})$ by way of a standard structural induction. The fact that 
$X^J$ is created by the embedding $\mathrm{Ext}$ will then follow automatically.

First, as noted in Example~\ref{exple_smallspaces}, the forgetful functor
$U\colon\mathrm{Cat}_{\infty}(\mathcal{S})\rightarrow\mathrm{Cat}_{\infty}$ induces an equivalence
$U\colon\mathbf{Cat}_{\infty}(\mathcal{S})\rightarrow\mathbf{Cat}_{\infty}$ of $(\infty,2)$-categories. Its inverse maps an
$\infty$-category $J$ to the complete Segal space $(J^{\Delta^{\ast}})^{\simeq}$ \cite[Section 4]{jtqcatvsss}. This 
equips the $\infty$-category $\mathrm{Cat}_{\infty}(\mathcal{S})$ with a cotensor $X^J$ for any complete Segal space $X$ and any
$\infty$-category $J$ given as follows:
\begin{align*}
X^J & \simeq ((U(X)^J)^{\Delta^{\ast}})^{\simeq}\\
& \simeq (U(X)^{J\times\Delta^{\ast}})^{\simeq} \\
& \simeq\mathrm{Cat}_{\infty}(J\times\Delta^{\ast},U(X)) \\
& \simeq\mathrm{Cat}_{\infty}(\mathcal{S})(((J\times\Delta^{\ast})^{\Delta^{\bullet}})^{\simeq},X)\\
& \simeq s\mathcal{S}(((J\times\Delta^{\ast})^{\Delta^{\bullet}})^{\simeq},X)\\
& \simeq \{((J\times\Delta^{\ast})^{\Delta^{\bullet}})^{\simeq},X\}.
\end{align*}
The last line is an instance of Example~\ref{exple_weightedhomspace}. We note that this chain of equivalences is natural in all parameters.
This proves the statement for $\mathcal{C}=\mathcal{S}$. 

Second, for a general $\infty$-category $\mathcal{C}$ we recall that
\[U_{\ast}\colon\mathbf{Fun}(\mathcal{C}^{op},\mathbf{Cat}_{\infty}(\mathcal{S}))\xrightarrow{\simeq}\mathbf{Fun}(\mathcal{C}^{op},\mathbf{Cat}_{\infty})\]
is an equivalence of $(\infty,2)$-categories (Remark~\ref{rem_intcatstandard}). For the same reason, the  
associated small cotensors in $\mathbf{Fun}(\mathcal{C}^{op},\mathbf{Cat}_{\infty}(\mathcal{S}))$ exist and are computed pointwise in
$\mathbf{Cat}_{\infty}(\mathcal{S})$. The same applies to the small weighted limits in
$s\hat{\mathcal{C}}\simeq\mathrm{Fun}(\mathcal{C},s\mathcal{S})$. That is to say, the evaluation functors
$\mathrm{ev}_C\colon\mathbf{Fun}(\mathcal{C}^{op},\mathbf{Cat}_{\infty}(\mathcal{S}))\rightarrow\mathbf{Cat}_{\infty}(\mathcal{S})$ and
$\mathrm{ev}_C\colon\mathrm{Fun}(\mathcal{C}^{op},s\mathcal{S})\rightarrow s\mathcal{S}$ for $C\in\mathcal{C}$ preserve all small cotensors 
and small limits, respectively. We further note that the evaluation functors
$\mathrm{ev}_C$ themselves are natural in $C\in\mathcal{C}$. Thus, let $X$ be a $\mathcal{C}$-indexed complete Segal space, and let $J$ be 
a small $\infty$-category. As both $X^J$ and $\{((J\times\Delta^{\ast})^{\Delta^{\bullet}})^{\simeq},X\}$ exist in
$\mathbf{Fun}(\mathcal{C}^{op},\mathbf{Cat}_{\infty}(\mathcal{S}))$ and in $\mathrm{Fun}(\mathcal{C}^{op},s\mathcal{S})$, respectively, we 
are only left to show that they are equivalent in $\mathrm{Fun}(\mathcal{C}^{op},s\mathcal{S})$; the latter is then necessarily a
$\mathcal{C}$-indexed complete Segal space as well. It suffices to do so representably. And indeed, via Example~\ref{exmple_ends}, for all 
$Y\in\mathrm{Fun}(\mathcal{C}^{op},s\mathcal{S})$ we obtain natural equivalences of mapping spaces as follows:
\begin{align*}
\mathrm{Fun}(\mathcal{C}^{op},s\mathcal{S})(Y,X^J) & \simeq \int_{C\in\mathcal{C}}s\mathcal{S}(Y(C),X(C)^J)\\
 & \simeq \int_{C\in\mathcal{C}}s\mathcal{S}(Y(C),\{((J\times\Delta^{\ast})^{\Delta^{\bullet}})^{\simeq},X(C)\}) \\
 & \simeq \int_{C\in\mathcal{C}}s\mathcal{S}(Y(C),\{((J\times\Delta^{\ast})^{\Delta^{\bullet}})^{\simeq},X\}(C)) \\
 & \simeq\mathrm{Fun}(\mathcal{C}^{op},s\mathcal{S})(Y,\{((J\times\Delta^{\ast})^{\Delta^{\bullet}})^{\simeq},X\}).
\end{align*}
Third, for any complete Segal object $X$ in $\mathcal{C}$ and any small (finite) $\infty$-category $J$, the simplicial object
$\{((J\times\Delta^{\ast})^{\Delta^{\bullet}})^{\simeq},X\}$ in 
$\mathcal{C}$ exists by assumption. The functor $sy(X)$ is a $\mathcal{C}$-indexed complete Segal space by Lemma~\ref{lemmasegalinternal}. 
The embedding $sy$ preserves limits, and in the second step we have seen that the limit
$\{((J\times\Delta^{\ast})^{\Delta^{\bullet}})^{\simeq},sy(X)\}$ is a $\mathcal{C}$-indexed complete Segal space as well. It follows that 
the simplicial object $\{((J\times\Delta^{\ast})^{\Delta^{\bullet}})^{\simeq},X\}$ is a complete Segal object in $\mathcal{C}$ again by 
Lemma~\ref{lemmasegalinternal}. Furthermore, we have
\[sy(\{((J\times\Delta^{\ast})^{\Delta^{\bullet}})^{\simeq},X)\})\simeq sy(X)^J\]
by the second step as well. Thus, $X^J:=\{((J\times\Delta^{\ast})^{\Delta^{\bullet}})^{\simeq},X)\}$ is the cotensor of $X$ by $J$ in
$\mathbf{Cat}_{\infty}(\mathcal{C})$ (Example~\ref{exple_sub_cotensors}). We further have
\[\mathrm{Ext}(X^J)\simeq U_{\ast}(sy(X^J))\simeq U_{\ast}(sy(X)^J)\simeq U_{\ast}(sy(X))^J\simeq\mathrm{Ext}(X)^J,\]
because $U_{\ast}$ is an equivalence of $(\infty,2)$-categories and hence preserves cotensors.
\end{proof}

\begin{remark}\label{remcotensorptwdescr}
The fact that the full sub-$(\infty,2)$-category
$\mathbf{Cat}_{\infty}(\mathcal{C})\subset\mathbf{Fun}(\mathcal{C}^{op},\mathbf{Cat}_{\infty})$ is closed under the 
stated simplicial cotensors follows directly from \cite[Corollary 5.18]{rs_comp} and \cite[Corollary 5.26]{rs_comp}. It however will be 
useful to have an explicit formula in terms of limits in $\mathcal{C}$ to compute them. In fact, in case $\mathcal{C}$ is a (bi)complete
1-category, an explicit simplicial cotensor construction associated to the ``categorical simplicial structure'' on $s\mathcal{C}$ has been 
described already in \cite{duggersimp}. The formula $X^J\simeq((J\times\Delta^{\ast})^{\Delta^{\bullet}})^{\simeq},X\}$
is a homotopy-coherent version thereof as can be made precise whenever $\mathcal{C}$ is a suitable model category, see 
Section~\ref{secsubcosmosofcsos}.
\end{remark}

\begin{remark}\label{rem2catstrdirect}
We note that the assignment of (finite) $\infty$-categorical cotensors in $\mathbf{Cat}_{\infty}(\mathcal{C})$ is functorial. 
For every $Y\in\mathrm{Cat}_{\infty}(\mathcal{C})$ we obtain a simplicial object $Y^{\Delta^{\bullet}}$ in
$\mathrm{Cat}_{\infty}(\mathcal{C})$ that can be shown to be a complete Segal object in $\mathrm{Cat}_{\infty}(\mathcal{C})$ itself. By 
Lemma~\ref{lemmasegalinternal}, it follows that the simplicial spaces
$\mathrm{Cat}_{\infty}(\mathcal{C})(X,Y^{\Delta^{\bullet}})$ are complete Segal spaces as well. Thus, via Proposition~\ref{proppowering} we 
can define an enhanced mapping $\infty$-category functor $\mathrm{Map}_{\mathrm{Cat}_{\infty}(\mathcal{C})}(\phv,\phv)$ on
$\mathrm{Cat}_{\infty}(\mathcal{C})$ as recalled in Definition~\ref{defenhmapqcats} via
\begin{align}\label{equ_rem2catstrdirect}
\mathrm{Map}_{\mathrm{Cat}_{\infty}(\mathcal{C})}(X,Y):=U(\mathrm{Cat}_{\infty}(\mathcal{C})(X,Y^{\Delta^{\bullet}}))
\end{align}
for $X,Y\in\mathrm{Cat}_{\infty}(\mathcal{C})$ as mapping $\infty$-categories. It then follows that the externalization functor 
induces an equivalence between the pair $(\mathrm{Cat}_{\infty}(\mathcal{C}),\mathrm{Map}_{\mathrm{Cat}_{\infty}(\mathcal{C})})$ and
$N_{\Delta}(\underline{\mathbf{Cat}_{\infty}(\mathcal{C})})$ equipped with its canonical enhanced mapping $\infty$-category functor.
In other words, the assignment (\ref{equ_rem2catstrdirect}) computes the mapping $\infty$-categories
$\mathbf{Cat}_{\infty}(\mathcal{C})(\phv,\phv)$. This recovers the higher cells between small $\mathcal{C}$-indexed $\infty$-categories 
more directly in terms of internal diagrams in $\mathcal{C}$ as we describe more explicitly in the next remark.
\end{remark}

\begin{remark}\label{rem3catstrdirect}
We can describe the $(\infty,2)$-categorical structure on internal $\infty$-categories in $\mathcal{C}$ in terms 
of internal diagrams and their internal natural transformations more explicitly via the enhanced mapping
$\infty$-category functor from Remark~\ref{rem2catstrdirect}. We therefore notationally identify an internal $\infty$-category 
$X\in\mathrm{Cat}_{\infty}(\mathcal{C})$ with its image $\mathrm{Ext}(X)\in\mathrm{Cat}_{\infty}(\mathcal{C})$ for the sake of 
readability.

First, the fact that $\mathrm{Cat}_{\infty}(\mathcal{C})$ is the underlying $\infty$-category of
$\mathbf{Cat}_{\infty}(\mathcal{C})$ means that for any two internal $\infty$-categories $X$ and $Y$ in
$\mathcal{C}$, the mapping space $\mathbf{Cat}_{\infty}(\mathcal{C})(X,Y)^{\simeq}$ is up to natural 
equivalence given by the space $\mathrm{Cat}_{\infty}(\mathcal{C})(X,Y)=s\mathcal{C}(X,Y)$ of natural transformations 
between simplicial objects (i.e.\ the space of ``internal functors'') from $X$ to $Y$. That means, an internal 
functor from $X$ to $Y$ is but a commutative diagram $f\colon X\rightarrow Y$ of the form
\[\xymatrix{
& \\
X_2\ar[r]^{f_2}\ar@{}[u]|{\vdots}\ar@<2ex>[d]\ar[d]\ar@<-2ex>[d] & Y_2\ar@{}[u]|{\vdots}\ar@<2ex>[d]\ar[d]\ar@<-2ex>[d]\\
X_1\ar@<1ex>[d]\ar@<-1ex>[d]\ar@<1ex>[u]\ar@<-1ex>[u]\ar[r]^{f_1} & Y_1\ar@<1ex>[d]\ar@<-1ex>[d]\ar@<1ex>[u]\ar@<-1ex>[u] \\
X_0\ar[r]^{f_0}\ar[u] & Y_0\ar[u]
}\]
in $\mathcal{C}$. Furthermore, given two such internal functors $f,g\colon X\rightarrow Y$, the space
$(\mathbf{Cat}_{\infty})_{Y\times  Y}(X,Y^{\Delta^1})^{\simeq}$ of 2-cells from $f$ to $g$ is up to natural equivalence 
given by the space of internal functors $\mathrm{Cat}_{\infty}(\mathcal{C})_{Y\times  Y}(X,Y^{\Delta^1})$ over $Y\times Y$, 
where $Y^{\Delta^1}\rightarrow Y\times Y$ is the canonical projection and $X\rightarrow Y\times Y$ is given by the 
pair $f,g$. That is, the space of ``internal natural transformations'' from $f$ to $g$. An internal natural 
transformation from $f$ to $g$ is thus a diagram
\[\xymatrix{
 & Y^{\Delta^1}\ar[d]^{(s,t)} \\
X\ar[ur]^{\alpha}\ar[r]_(.4){(f,g)} & Y\times Y
}\]
in $\mathrm{Cat}_{\infty}(\mathcal{C})$. For visualization, at level $0$, we recall that the object $(Y^{\Delta^1})_0\in\mathcal{C}$ 
represents the presheaf $\mathrm{Ext}(Y^{\Delta^1})^{\simeq}\simeq(\mathrm{Ext}(Y)^{\Delta^1})^{\simeq}$. It hence is equivalent to the 
object $Y_1$ (over $Y_0\times Y_0$) by \cite[Proposition 5.14.1]{rs_comp}. Under this equivalence, the morphism
$\alpha_0\colon X_0\rightarrow Y_1$ assigns to every object $x$ in $X$ a morphism $\alpha_0(x)\colon f_0(x)\rightarrow g_0(x)$. At level 
$1$, the object $(Y^{\Delta^1})_1$ is up to equivalence (over $Y_1\times Y_1$) the object
$(Y^{\Delta^1\times\Delta^1})_0\simeq Y_2\times_{Y_1} Y_2$ of squares in $Y$ for essentially the same reason. Thus, the morphism
$\alpha_1\colon X_1\rightarrow Y_2\times_{Y_1} Y_2$ assigns to every morphism $a\colon x\rightarrow z$ contained 
in $X_1$ a square in $Y$ of the form
\[\xymatrix{
f_0(x)\ar[r]^{\alpha_0(x)}\ar[d]_{f_1(a)}\ar@{}[dr]|{\alpha_1(a)} & g_0(x)\ar[d]^{g_1(a)} \\
f_0(z)\ar[r]_{\alpha_0(y)} & g_0(z).
}\]
\end{remark}

We thus have seen that the canonical enrichment of $\mathcal{C}$ in the $\infty$-category of spaces always induces 
a (generally locally non-discrete) enrichment of $\mathrm{Cat}_{\infty}(\mathcal{C})$ in the $\infty$-category of quasi-categories. The following proposition shows that an enrichment of a suitably complete
$\infty$-category $\mathcal{C}$ in itself (considered as a cartesian $\infty$-category) likewise induces 
an enrichment of $\mathrm{Cat}_{\infty}(\mathcal{C})$ in itself.

\begin{proposition}\label{propcartclosed}
Suppose $\mathcal{C}$ has countable limits and is cartesian closed. Then so is $\mathrm{Cat}_{\infty}(\mathcal{C})$, and
$\mathrm{Ext}\colon\mathrm{Cat}_{\infty}(\mathcal{C})\rightarrow\mathrm{Fun}(\mathcal{C}^{op},\mathrm{Cat}_{\infty})$ 
preserves exponentials. 
\end{proposition}
\begin{proof}
Since $\mathrm{Fun}(\mathcal{C}^{op},\mathrm{Cat}_{\infty})$ is cartesian closed and 
$\mathrm{Cat}_{\infty}(\mathcal{C})\subset\mathrm{Fun}(\mathcal{C}^{op},\mathrm{Cat}_{\infty})$ is a full
sub-$\infty$-category which is closed under products, by Lemma~\ref{corsegalyonedaff} it suffices to show that 
the sub-$\infty$-category $\mathrm{Cat}_{\infty}(\mathcal{C})$ is closed under exponentials. Therefore, let $X$, 
$Y$ be a pair of $\infty$-categories in $\mathcal{C}$. To show that the exponential
$\mathrm{Ext}(Y)^{\mathrm{Ext}(X)}$ is again small, it suffices to show that the presheaves
$(\mathrm{Ext}(Y)^{\mathrm{Ext}(X)})^{\simeq}$ and $((\mathrm{Ext}(Y)^{\mathrm{Ext}(X)})^{\Delta^1})^{\simeq}$ in $\hat{\mathcal{C}}$ 
are representable \cite[Corollary 3.32, Theorem 5.15]{rs_comp}. For the former, we observe that there are natural equivalences as follows.
\begin{align*}
(\mathrm{Ext}(Y)^{\mathrm{Ext}(X)})^{\simeq}(C) & \simeq \mathrm{Fun}(\mathcal{C}^{op},\mathrm{Cat}_{\infty})(yC\times\mathrm{Ext}(X),\mathrm{Ext}(Y)) \\
 &\simeq\mathrm{Fun}(\mathcal{C}^{op},\mathrm{Cat}_{\infty})(\mathrm{Ext}(c(C)\times X),\mathrm{Ext}(Y))  \\
 &\simeq s\mathcal{C}(c(C)\times X,Y) \\
 &\simeq\int\limits_{n\in\Delta}\mathcal{C}((c(C)\times X)_n,Y_n)\\
 &\simeq\int\limits_{n\in\Delta}\mathcal{C}(C\times X_n,Y_n) \\
 &\simeq\int\limits_{n\in\Delta}\mathcal{C}(C,Y_n^{X_n})\\
 &\simeq \mathcal{C}(C,\int\limits_{n\in\Delta}Y_n^{X_n})
\end{align*}
The equivalence in the first line is given by the Yoneda lemma of \cite[Theorem 5.7.3]{riehlverityelements}; it also follows more 
directly from Remark~\ref{remextyoneda}. The expression of the mapping spaces of $s\mathcal{C}$ as a corresponding end construction is again 
Example~\ref{exmple_ends}. It follows that $(\mathrm{Ext}(Y)^{\mathrm{Ext}(X)})^{\simeq}$ is represented by the end
$\int_{n\in\Delta}Y_n^{X_n}$ in $\mathcal{C}$. This end exists as it is a $\mathrm{Tw}(\Delta)$-indexed (and hence 
countable) limit in $\mathcal{C}$. The same argument shows that
$((\mathrm{Ext}(Y)^{\mathrm{Ext}(X)})^{\Delta^1})^{\simeq}$ is represented by the end $\int_{n\in\Delta}(Y^{\Delta^1})_n^{X_n}$ in
$\mathcal{C}$, where $Y^{\Delta^1}$ is the corresponding cotensor constructed in Proposition~\ref{proppowering}.
\end{proof}

\begin{remark}
It may be worth to point out that the analogue to Proposition~\ref{propcartclosed} in the ordinary categorical case only requires left 
exactness of $\mathcal{C}$ rather than countably infinite completeness thereof, see \cite[Paragraph (4.5)]{streetintcat} (or
\cite[Proposition 7.2.2]{jacobsttbook}). 
Essentially, this is because the same proof in this context computes any given exponential of internal categories via the 1-truncated 
simplex category $\Delta_{\leq 1}=(\xymatrix{[0]\ar[r] & [1]\ar@<1ex>[l]\ar@<-1ex>[l])}$ instead of the entire simplex category 
$\Delta$. This however is finite, and hence so is its twisted arrow category. Indeed, note that the explicit formula for the exponential 
of two internal categories given in the proof of \cite[Proposition 7.2.2]{jacobsttbook} is just an explicit computation of exactly this end.
\end{remark}

\begin{remark}
For any $\infty$-category $\mathcal{C}$ and any pair $X,Y\in\mathbf{Cat}_{\infty}(\mathcal{C})$ there is a binatural equivalence
\begin{align*}
\mathbf{Cat}_{\infty}(X,Y) & := \mathbf{Fun}(\mathcal{C}^{op},\mathbf{Cat}_{\infty})(\mathrm{Ext}(X),\mathrm{Ext}(Y)) \\
 & \simeq\mathbf{Fun}(\mathcal{C}^{op},\mathbf{Cat}_{\infty})(\ast,\mathrm{Ext}(Y)^{\mathrm{Ext}(Y)})
\end{align*}
as $\mathbf{Fun}(\mathcal{C}^{op},\mathbf{Cat}_{\infty})$ is cartesian closed. Thus, whenever $\mathcal{C}$ is countably complete and 
cartesian closed itself, the mapping $\infty$-categories $\mathbf{Cat}_{\infty}(X,Y)$ are binaturally equivalent to the underlying
$\infty$-category $\mathrm{Ext}(Y^X)(\ast)$ of the exponential $Y^X\in\mathbf{Cat}_{\infty}(\mathcal{C})$ by 
Proposition~\ref{propcartclosed} (again via the Yoneda Lemma either in \cite[Theorem 5.7.3]{riehlverityelements}, or, alternatively, in 
Proposition~\ref{lemmaextyoneda}). This is up to terminology exactly the $\infty$-categorical enrichment of
$\mathrm{Cat}_{\infty}(\mathcal{C})$ defined and studied in \cite[Section 3]{martini_yoneda} whenever $\mathcal{C}$ is an $\infty$-topos.
\end{remark}

\begin{theorem}\label{corcosmos1}
The $(\infty,2)$-category $\mathbf{Cat}_{\infty}(\mathcal{C})$ is
\begin{enumerate}
\item a full finitary sub-$\infty$-cosmos of $\mathbf{Fun}(\mathcal{C}^{op},\mathbf{Cat}_{\infty})$ whenever $\mathcal{C}$ is left exact. 
It thus defines an $\infty$-cosmos in the sense of \cite{rvyoneda}.
\item a full sub-$\infty$-cosmos of $\mathbf{Fun}(\mathcal{C}^{op},\mathbf{Cat}_{\infty})$ whose underlying $\infty$-category is closed under 
all limits (and exponentials) whenever $\mathcal{C}$ is complete (and cartesian closed). It thus defines a (cartesian closed) $\infty$-cosmos 
in the sense of \cite{riehlverityelements}.
\end{enumerate}
\end{theorem}
\begin{proof}
We recall that $\mathbf{Fun}(\mathcal{C}^{op},\mathbf{Cat}_{\infty})$ admits all $\infty$-cosmological structure from 
Example~\ref{expleqcatiscosmos} and Example~\ref{exple_size} (it only fails to be an $\infty$-cosmos if $\mathcal{C}$ is large as it 
then is not locally small itself). We further recall Definition~\ref{defsubcosmoses}, and note that small products, pullbacks of fibrations 
and sequential limits of fibrations in $\mathbf{Fun}(\mathcal{C}^{op},\mathbf{Cat}_{\infty})$ compute homotopy limits altogether. As such 
they represent the corresponding limits in the underlying $\infty$-category $\mathrm{Fun}(\mathcal{C}^{op},\mathrm{Cat}_{\infty})$. 
Similarly, exponentials in $\mathbf{Fun}(\mathcal{C}^{op},\mathbf{Cat}_{\infty})$ represent the corresponding exponentials in
$\mathrm{Fun}(\mathcal{C}^{op},\mathrm{Cat}_{\infty})$ by Lemma~\ref{lemmacartclosed}. Thus, as the inclusion of
$\mathbf{Cat}_{\infty}(\mathcal{C})$ in $\mathbf{Fun}(\mathcal{C}^{op},\mathbf{Cat}_{\infty})$ is replete by definition,
the theorem follows directly from Lemma~\ref{lemmascolimits}, Proposition~\ref{proppowering} and Proposition~\ref{propcartclosed}.
\end{proof}

We hence can develop the formal $\infty$-category theory of small $\mathcal{C}$-indexed $\infty$-categories using 
the theory of \cite{riehlverityelements} as referred to in the Corollary of Section~\ref{secintro}. We point out that the higher
non-invertible structure on $\mathbf{Cat}_{\infty}(\mathcal{C})$ is only representably internal (via Remark~\ref{rem2catstrdirect} and 
Remark~\ref{rem3catstrdirect}). In Section~\ref{secmodcat} we will see that whenever $\mathcal{C}$ can be presented by a model category, 
then there is an $\infty$-cosmological structure on $\mathrm{Cat}_{\infty}(\mathcal{C})$ that can be described explicitly internally, and 
which is equivalent to the one of Theorem~\ref{corcosmos1} at least whenever $\mathcal{C}$ is presentable.

In fact, while the $\infty$-cosmological structure on $\mathrm{Cat}_{\infty}(\mathcal{C})$ is (defined so to be) compatible with
the $\infty$-cosmological structure on $\mathrm{Fun}(\mathcal{C},\mathrm{Cat}_{\infty})$, the $(\infty,2)$-category
$\mathbf{Fun}(\mathcal{C}^{op},\mathbf{Cat}_{\infty})$ also admits small $\infty$-categorical tensors. In the following, we show that
$\mathbf{Cat}_{\infty}(\mathcal{C})$ also admits some $\infty$-categorical tensors whenever $\mathcal{C}$ has enough well behaved colimits; 
these however are generally not preserved by the embedding of 
$\mathbf{Cat}_{\infty}(\mathcal{C})$ in $\mathbf{Fun}(\mathcal{C}^{op},\mathbf{Cat}_{\infty})$. This parallels the fact that the Yoneda 
embedding $\mathcal{C}\rightarrow\hat{\mathcal{C}}$ generally does not preserve $\mathcal{S}$-tensors either.
Therefore, we introduce the following terminology, following a dual definition of Gray \cite[Section 5]{graymidwest}.

\begin{definition}\label{defcorep}
Say an $(\infty,2)$-category $\mathbf{C}$ is \emph{strongly corepresentable} if there is a cosimplicial object
\[\otimes\colon\Delta\rightarrow\mathrm{Fun}(\mathbf{C},\mathbf{C})\]
such that for all $n\geq 0$, the 
endofunctor $\Delta^n\otimes(\cdot)\colon\mathbf{C}\rightarrow\mathbf{C}$ computes a functorial tensoring with $\Delta^n$ in $\mathbf{C}$.
\end{definition}

We recall that an $\infty$-category is said to be finitary lextensive if it has finite limits and finite disjoint and universal 
coproducts. In ordinary category theory finitary lextensiveness is a standard notion \cite{clwextdist}; for
$\infty$-categories it can be defined completely analogously \cite[Definition 4.1]{rs_hgst}.

\begin{proposition}\label{proptensors}
Suppose $\mathcal{C}$ is finitary lextensive. Then $\mathrm{Cat}_{\infty}(\mathcal{C})$ is strongly corepresentable.
\end{proposition}

\begin{proof}

The category $\mathrm{Set}^{\mathrm{fin}}$ of finite sets is the free finite coproduct completion of the trivial $\infty$-category $\ast$. 
That is to say, for any $\infty$-category $\mathcal{C}$ with finite coproducts, the embedding
$\ast\hookrightarrow\mathrm{Set}^{\mathrm{fin}}$ of the terminal object induces an equivalence
$\mathrm{Fun}^{\sqcup}(\mathrm{Set}^{\mathrm{fin}},\mathcal{C})\xrightarrow{\sim}\mathcal{C}$, where the left hand 
side denotes the $\infty$-category of finite coproduct preserving functors. Its inverse can be curried to give a functor of the form
\begin{align}\label{equdefplaspecialcase}
\otimes_0\colon\mathrm{Set}^{\mathrm{fin}}\times\mathcal{C}\rightarrow\mathcal{C}
\end{align}
which assigns to a pair $(C,X)$ the $X$-fold copower $X\otimes_0 C$ binaturally.

Furthermore, we can show that the functor (\ref{equdefplaspecialcase}) preserves pullbacks; for this it suffices to show that it 
preserves pullbacks in both variables (by way of the extensivity axioms satisfied by $\mathcal{C}$). Thus, first, for any given object 
$C\in\mathcal{C}$, the $\iota$-left adjoint $(\cdot)\otimes_0 C\colon\mathrm{Set}^{\mathrm{fin}}\rightarrow\mathcal{C}$ preserves 
pullbacks. To verify this, it suffices to show that the canonical factorization
$(\cdot)\otimes_0 C\colon\mathrm{Set}^{\mathrm{fin}}\rightarrow\mathcal{C}_{/C}$ is left exact (since the canonical projection
$\mathcal{C}_{/C}\rightarrow\mathcal{C}$ is pullback-preserving). The functor
$(\cdot)\otimes_0 C\colon\mathrm{Set}^{\mathrm{fin}}\rightarrow\mathcal{C}_{/C}$ evaluates a finite set $X$ at the $X$-fold copower of 
the terminal object $1_C\in\mathcal{C}_{/C}$, 
and so $\ast\otimes_0 C\simeq\coprod_{\ast}1_C$ is a terminal object in $\mathcal{C}_{/C}$ by construction. One can 
show that $(\cdot)\otimes_0 C\colon\mathrm{Set}^{\mathrm{fin}}\rightarrow\mathcal{C}_{/C}$ furthermore preserves binary products and 
equalizers by hand (virtually following the ordinary categorical case as in \cite[Section VII.1, p.350]{mlmsheaves} for example, using 
that $\mathcal{C}$ is finitary extensive). As the $\infty$-category $\mathrm{Set}^{\mathrm{fin}}$ is left exact, it follows that
$(\cdot)\otimes_0 C\colon\mathrm{Set}^{\mathrm{fin}}\rightarrow\mathcal{C}_{/C}$ preserves all finite limits via
\cite[Proposition 4.4.3.2]{luriehtt}.

Second, for any given finite set $X$, the functor $X\otimes_0(\cdot)\colon\mathcal{C}\rightarrow\mathcal{C}$ preserves pullbacks again 
by finitary extensiveness of $\mathcal{C}$. 

Thus, we may consider the push-forward $\otimes_0^{s}:=s\otimes_0\colon s(\mathrm{Set}^{\mathrm{fin}}\times\mathcal{C})\rightarrow 
s\mathcal{C}$ which as a consequence preserves pullbacks as well. It hence descends to the respective full sub-$\infty$-categories
$\mathrm{Cat}_{\infty}(\mathrm{Set}^{\mathrm{fin}}\times\mathcal{C})\simeq\mathrm{Cat}_{\infty}(\mathrm{Set}^{\mathrm{fin}})\times\mathrm{Cat}_{\infty}(\mathcal{C})$ 
and $\mathrm{Cat}_{\infty}(\mathcal{C})$ of complete Segal objects. Since the category $\Delta$ is a locally finite category, its 
Yoneda embedding factors to give a left exact inclusion $y\colon\Delta\rightarrow s\mathrm{Set}^{\mathrm{fin}}$  (whose codomain 
literally denotes the category of simplicial objects in finite sets; not to be confused with the $\infty$-category of finite spaces). 
Furthermore, each $y(n)=\Delta^n\in s\mathrm{Set}^{\mathrm{fin}}$ is a Segal object in $\mathrm{Set}^{\mathrm{fin}}$; and as each 
$n\in\Delta$ is a posetal category, each $\Delta^n\in s\mathrm{Set}^{\mathrm{fin}}$ is in fact a complete Segal object
\cite[Example 3.9]{rasekhunivalence}. The Yoneda embedding of $\Delta$ thus further factors to give a left exact inclusion
$y\colon\Delta\rightarrow\mathrm{Cat}_{\infty}(\mathrm{Set}^{\mathrm{fin}})$. We thus obtain a diagram
\[\xymatrix{
\Delta\times\mathrm{Cat}_{\infty}(\mathcal{C})\ar[r]^{y\times\iota}\ar@/_1pc/[dr]_(.4){y\times 1} & s\mathrm{Set}^{\mathrm{fin}}\times s\mathcal{C}\ar[r]^(.6){\otimes_0^s} & s\mathcal{C} \\
 & \mathrm{Cat}_{\infty}(\mathrm{Set}^{\mathrm{fin}})\times\mathrm{Cat}_{\infty}(\mathcal{C})\ar[r]_(.65){\otimes_0^s}\ar@{^(->}[u] & \mathrm{Cat}_{\infty}(\mathcal{C})\ar@{^(->}[u]
}\]
and denote the induced bottom composition by
\[\otimes\colon\Delta\times\mathrm{Cat}_{\infty}(\mathcal{C})\rightarrow\mathrm{Cat}_{\infty}(\mathcal{C}).\]
We are left to show that for every $n\geq 0$ and every $X\in\mathrm{Cat}_{\infty}(\mathcal{C})$, the complete Segal object
$\Delta^n\otimes X$ is a tensor of $X$ with $\Delta^n$ in $\mathbf{Cat}_{\infty}(\mathcal{C})$. In terms of the enhanced mapping
$\infty$-category functor from Remark~\ref{rem2catstrdirect}, this means we are to show that for every
$Y\in\mathrm{Cat}_{\infty}(\mathcal{C})$ there is an equivalence
\[U(\mathrm{Cat}_{\infty}(\mathcal{C})(\Delta^n\otimes X,Y^{\Delta^{\bullet}}))\simeq U(\mathrm{Cat}_{\infty}(\mathcal{C})(X,Y^{\Delta^{\bullet}}))^{\Delta^n}\]
of $\infty$-categories natural in both the arguments $X,Y\in\mathrm{Cat}_{\infty}(\mathcal{C})$. By virtue of the existence of (functorial) 
cotensors, the fact that cotensors commute with one another, and that the equivalence $U$ preserves cotensors, this follows from the 
existence of a binatural equivalence
\[\mathrm{Cat}_{\infty}(\mathcal{C})(\Delta^n\otimes X,Y^{\Delta^{\bullet}})\simeq\mathrm{Cat}_{\infty}(\mathcal{C})(X,(Y^{\Delta^n})^{\Delta^{\bullet}})\]
of complete Segal spaces. For this it in turn suffices to construct a binatural (!) equivalence
\begin{align}\label{equmaptensdef}
\mathrm{Cat}_{\infty}(\mathcal{C})(\Delta^n\otimes X,Y)\simeq\mathrm{Cat}_{\infty}(\mathcal{C})(X,Y^{\Delta^{n}})
\end{align}
of spaces. We hence finish the proof with a construction of such an equivalence. We compute
\begin{align}
\label{equproptensors1} \mathrm{Cat}_{\infty}(\mathcal{C})(X,Y^{\Delta^{n}}) &
\simeq \int\limits_{m\in\Delta^{op}}\mathcal{C}(X_m,\{((\Delta^n\times\Delta^m)^{\Delta^{\bullet}})^{\simeq},Y\})\\
\label{equproptensors2} &\simeq \int\limits_{m\in\Delta^{op}}s\mathcal{S}(((\Delta^n\times\Delta^m)^{\Delta^{\bullet}})^{\simeq},sy(Y)(X_m)) \\
\notag &\simeq \int\limits_{(m,l)\in\Delta^{op}}\mathcal{S}(((\Delta^n\times\Delta^m)^{\Delta^{l}})^{\simeq},\mathcal{C}(X_m,Y_l))\\
\label{equproptensors3} &\simeq\int\limits_{(m,l)\in\Delta^{op}}\mathcal{C}(((\Delta^n\times\Delta^m)^{\Delta^{l}})^{\simeq}\otimes_0 X_m,Y_l) \\
\label{equproptensors4} &\simeq\int\limits_{m\in\Delta^{op}}s\mathcal{C}(((\Delta^n\times\Delta^m)^{\Delta^{\bullet}})^{\simeq}\otimes_0 X_m,Y) \\
\notag &\simeq s\mathcal{C}(\int\limits^{m\in\Delta^{op}}((\Delta^n\times\Delta^m)^{\Delta^{\bullet}})^{\simeq}\otimes_0 X_m,Y) \\
\label{equproptensors5}  &\simeq s\mathcal{C}(\Delta^n\otimes X,Y)
\end{align}
binaturally. Here, Line (\ref{equproptensors1}) follows from the pointwise description of the 
cotensor $Y^{\Delta^n}$ as a simplicial collection of accordingly weighted limits in $\mathcal{C}$ (Proposition~\ref{proppowering}). 
Line (\ref{equproptensors2}) is an instance of the universal property of weighted limits (Lemma~\ref{lemma_weights}). In Line 
(\ref{equproptensors3}) we use the universal property of the coproduct functor
$\otimes_0$, together with the fact that the spaces $((\Delta^n\times\Delta^m)^{\Delta^{l}})^{\simeq}$ are in fact sets (since the 
categories $\Delta^n\times\Delta^m$ have no non-trivial isomorphisms). In Line (\ref{equproptensors4}), the domain
$((\Delta^n\times\Delta^m)^{\Delta^{\bullet}})^{\simeq}\otimes_0 X_m$ is defined by pointwise application of $(\cdot)\otimes_0 X_m$ on 
the simplicial finite set $((\Delta^n\times\Delta^m)^{\Delta^{\bullet}})^{\simeq}$. Lastly, Line (\ref{equproptensors5}) is essentially an application of the Yoneda lemma: First, we have a natural equivalence
\begin{align}
\notag \int\limits^{m\in\Delta^{op}}((\Delta^n\times\Delta^m)^{\Delta^{\bullet}})^{\simeq}\otimes_0 X_m & \simeq ((\Delta^n)^{\Delta^{\bullet}})^{\simeq}\otimes_0^s\int\limits^{m\in\Delta^{op}}((\Delta^m)^{\Delta^{\bullet}})^{\simeq}\otimes_0 X_m \\
\label{equproptensors6} &\simeq \Delta^n \otimes \int\limits^{m\in\Delta^{op}}\Delta^m(\bullet)\otimes_0 X_m
\end{align}
in $\mathcal{C}$. And second, the right component in (\ref{equproptensors6}) is the coend of $X$ with the corepresentables
$\Delta^{op}(m,\phv)$ over $\Delta^{op}$. This can be shown to compute $X$ itself by the usual Yoneda lemma argument.
\end{proof}

\begin{remark}\label{remlkanexttensor}
By construction, the tensor $\Delta^n\otimes X$ defined in Proposition~\ref{proptensors} is a simplicial object in $\mathcal{C}$ which 
evaluates at $[m]\in\Delta^{op}$ the coproduct $\coprod_{\Delta(m,n)}X_m$. In line with Remark~\ref{remcotensorptwdescr}, this is a general
homotopy-coherent version of the simplicial tensors constructed in \cite{duggersimp}. Furthermore, by essentially the same proof one can 
show that whenever all corepresentables
\[\mathcal{C}(C,\phv)\colon\mathcal{C}\rightarrow\mathcal{S}\] 
have a pullback-preserving left adjoint (relative to the full $\infty$-category of finite spaces/$\kappa$-small spaces/\dots), then 
there is a functor
\[\otimes\colon\mathrm{Cat}_{\infty}^{(\mathrm{fin}/\kappa/\dots)}\rightarrow\mathrm{Fun}(\mathrm{Cat}_{\infty}(\mathcal{C}),\mathrm{Cat}_{\infty}(\mathcal{C}))\]
which computes tensors with all (finite/$\kappa$-small/\dots) $\infty$-categories in $\mathbf{Cat}_{\infty}(\mathcal{C})$. For example, the 
corepresentables $\mathcal{C}(C,\phv)$ have pullback-preserving left adjoints 
whenever $\mathcal{C}$ is an $\infty$-topos, as each global sections functor $\Gamma\colon\mathcal{C}_{/C}\rightarrow\mathcal{S}$ has a 
left exact left adjoint. Hence, if $\mathcal{C}$ is an $\infty$-topos then $\mathbf{Cat}_{\infty}(\mathcal{C})$ is also tensored over
$\mathrm{Cat}_{\infty}$ as stated in \cite[Section 3.4]{martini_yoneda}.
\end{remark}

\begin{remark}\label{remmappingqcatalt2}
Let $\mathcal{C}$ be a finitary lextensive $\infty$-category. Then functoriality of the tensoring $\otimes$ in 
Proposition~\ref{proptensors} together with binaturality of Equation (\ref{equmaptensdef}) gives an alternative description of the 
enhanced mapping $\infty$-category functor from Remark~\ref{remmappingqcatalt} in terms of its tensoring with the $\infty$-categories
$\Delta^n$ rather than its cotensoring:
\[\mathrm{Map}_{\mathrm{Cat}_{\infty}(\mathcal{C})}(X,Y)\simeq U(\mathrm{Cat}_{\infty}(\mathcal{C})(\Delta^{\bullet}\otimes X,Y)).\]
\end{remark}

\begin{remark}
While all cotensors in $\mathbf{Cat}_{\infty}(\mathcal{C})$ which exist are automatically preserved by the canonical inclusion
$\mathbf{Cat}_{\infty}(\mathcal{C})\subset\mathbf{Fun}(\mathcal{C}^{op},\mathbf{Cat}_{\infty})$ by Proposition~\ref{proppowering}, the 
tensors generally are not. Case in point, if $\mathcal{C}$ is finitary lextensive and $\ast$ is terminal in $\mathcal{C}$, then
$\mathrm{Ext}(\Delta^1\otimes c(\ast))$ is generally not the tensor
$\Delta^1\otimes\mathrm{Ext}(\ast)\simeq c(\Delta^1)\times \ast\simeq c(\Delta^1)$ in
$\mathbf{Fun}(\mathcal{C}^{op},\mathbf{Cat}_{\infty})$. Here, $c(\Delta^1)\colon\mathcal{C}^{op}\rightarrow\mathrm{Cat}_{\infty}$ 
denotes the constant functor with value $\Delta^1$. Indeed, if we denote by $\mathbbm{2}$ the coproduct $\ast\sqcup\ast$ in
$\mathcal{C}$, then $c(\Delta^1)(\mathbbm{2})$ is just $\Delta^1$ by definition. However
$\mathrm{Ext}(\Delta^1\otimes c(\ast))(\mathbbm{2})$ is the category $\Delta^1\sqcup\Delta^0\sqcup\Delta^0$. Explicitly, its objects are given by the four 
morphisms of type $\mathbbm{2}\rightarrow\mathbbm{2}$ in $\mathcal{C}$ together with one arrow from the pair
$(\mathrm{incl},\mathrm{incl})\colon\mathbbm{2}\rightarrow\mathbbm{2}$ of left inclusions to the pair
$(\mathrm{incr},\mathrm{incr})\colon\mathbbm{2}\rightarrow\mathbbm{2}$ of right inclusions.
\end{remark}

\begin{corollary}\label{cortensors}
Suppose $\mathcal{C}$ is finitary lextensive. Then there is a cosimplicial object
\[\otimes\colon\Delta\rightarrow\mathrm{Fun}(\mathcal{C},\mathrm{Cat}_{\infty}(\mathcal{C}))\]
such that for all objects $C\in\mathcal{C}$ and for all $Y\in\mathrm{Cat}_{\infty}(\mathcal{C})$ there is a natural equivalence
\[\mathrm{Cat}_{\infty}(\mathcal{C})(\Delta^{\bullet}\otimes c(C),Y)\simeq sy(Y)(C)\]
of $\mathcal{C}$-indexed complete Segal spaces. In particular, there is a natural equivalence
\[U(\mathrm{Cat}_{\infty}(\mathcal{C})(\Delta^{\bullet}\otimes c(C),Y))\simeq\mathrm{Ext}(Y)(C)\]
of underlying $\infty$-categories, and for every $n\in\Delta$ there is a binatural equivalence
\[\mathrm{Cat}_{\infty}(\mathcal{C})(\Delta^{n}\otimes c(C),Y)\simeq\mathcal{C}(C,Y_n)\]
of spaces. Thus, the functor $\Delta^n\otimes c\colon\mathcal{C}\rightarrow\mathrm{Cat}_{\infty}(\mathcal{C})$ is left adjoint to evaluation at 
$n$.
\end{corollary}
\begin{proof}
The inclusion $c\colon\mathcal{C}\rightarrow\mathrm{Cat}_{\infty}(\mathcal{C})$ induces a restriction
\[\otimes\colon\Delta\rightarrow\mathrm{Fun}(\mathcal{C},\mathrm{Cat}_{\infty}(\mathcal{C}))\]
of the cosimplicial object in Proposition~\ref{proptensors}. For every $C\in\mathcal{C}$ and every
$Y\in\mathrm{Cat}_{\infty}(\mathcal{C})$, we obtain a chain of natural equivalences of $\mathcal{C}$-indexed complete Segal spaces as 
follows.
\begin{align}
\label{equcortensors1} \mathrm{Cat}_{\infty}(\mathcal{C})(\Delta^{\bullet}\otimes c(C),Y) &\simeq \mathrm{Cat}_{\infty}(\mathcal{C})(c(C),Y^{\Delta^{\bullet}}) \\
\label{equcortensors2}  & \simeq\mathcal{C}(C,\mathrm{ev}_0(Y^{\Delta^{\bullet}})) \\
\label{equcortensors3}  & \simeq\mathcal{C}(C,Y_{\bullet}) \\ 
\notag &\simeq sy(Y)(C)
\end{align}
Here, Line~(\ref{equcortensors1}) is Equation (\ref{equmaptensdef}) after restriction along
$c\colon\mathcal{C}\rightarrow\mathrm{Cat}_{\infty}(\mathcal{C})$. Line~(\ref{equcortensors2}) is part of the adjunction $c\adj \mathrm{ev}_0$, 
Line~(\ref{equcortensors3}) follows from the fact that for any given $n\geq 0$  the presheaf
$\mathrm{Ext}(Y^{\Delta^n})^{\simeq}\simeq(\mathrm{Ext}(Y)^{\Delta^n})^{\simeq}$ is represented by $Y_n$
\cite[Proposition 5.14]{rs_comp}. These equivalences are natural in $\Delta^{op}$ given that the simplicial object 
$Y^{\Delta^{\bullet}}\colon\Delta^{op}\rightarrow\mathrm{Cat}_{\infty}(\mathcal{C})$ was defined in terms of the simplicial object
$\mathrm{Ext}(Y)^{\Delta^{\bullet}}$ and the fact that
$\mathrm{Ext}\colon\mathrm{Cat}_{\infty}(\mathcal{C})\rightarrow\mathrm{Fun}(\mathcal{C}^{op},\mathrm{Cat}_{\infty})$ is fully faithful (and 
so has an inverse on its image). The other equivalences in the statement are all immediate consequences.
\end{proof}

\begin{remark}
Some individual examples of internal categories of the form
$J\otimes c(\ast)$ in a 1-category $\mathcal{C}$ with a terminal object $\ast$ -- explicitly the four internal categories $\Delta^0\otimes c(\ast)$, $\Delta^1\otimes c(\ast)$, $(\Delta^0\sqcup\Delta^0)\otimes c(\ast)$ and 
$N(\xymatrix{0\ar@<.5ex>[r]\ar@<-.5ex>[r] & 1})\otimes c(\ast)$ -- have been described in \cite[Example 7.1.2.(iii)]{jacobsttbook}.
A general theory of tensors is neither addressed here or in \cite{streetintcat} however.
\end{remark}

Our results to this point yield the following hierarchy of structural richness, which expands its ordinary categorical analogues shown 
in \cite[Section 4]{streetintcat}.
Therefore, we recall that a simplicial diagram $W\colon\mathbf{I}\rightarrow\mathbf{S}$, i.e.\ a simplicial weight, is said to be flexible 
if it is a projectively cofibrant cell complex (with respect to either $(\mathbf{S},\mathrm{QCat})$ or $(\mathbf{S},\mathrm{Kan})$, given 
that they have the same cofibrations). A flexibly weighted limit in a simplicial category $\mathbf{C}$ is the weighted limit
$\{W,F\}\in\mathbf{C}$ of a simplicial diagram $F\colon\mathbf{I}\rightarrow\mathbf{C}$ associated to a flexible weight $W$
\cite[Section 6.2]{riehlverityelements}.


\begin{theorem}\label{corcosmos2}
Let $\mathcal{C}$ be an $\infty$-category and let
$\mathrm{Ext}\colon\mathbf{Cat}_{\infty}(\mathcal{C})\hookrightarrow\mathbf{Fun}(\mathcal{C}^{op},\mathbf{Cat}_{\infty})$ be the canonical 
embedding of $(\infty,2)$-categories.
\begin{enumerate}
\item Suppose $\mathcal{C}$ is (finitely) complete. Then $\mathbf{Cat}_{\infty}(\mathcal{C})$ has all flexibly weighted limits (with 
pointwise finite weights of finite index). Furthermore, the $\infty$-bicategory canonically associated to
$\mathbf{Cat}_{\infty}(\mathcal{C})$ is (finitely) $(\infty,2)$-complete in the sense of \cite[Section 6]{ghl2fib}, and $\mathrm{Ext}$ preserves these 
limits.
\item Suppose $\mathcal{C}$ is countably complete and cartesian closed. Then the underlying $\infty$-category
$\mathrm{Cat}_{\infty}(\mathcal{C})$ is cartesian closed and $\mathrm{Ext}$ preserves exponentials.
\item Suppose $\mathcal{C}$ is finitary lextensive. Then $\mathbf{Cat}_{\infty}(\mathcal{C})$ is strongly corepresentable.
\end{enumerate}
\end{theorem}
\begin{proof}
For Part 1, we can use that (finitary) sub-$\infty$-cosmoses are generally closed under flexibly (finitely) weighted limits via
\cite[Proposition 6.2.8]{riehlverityelements}. Thus, closure of $\mathbf{Cat}_{\infty}(\mathcal{C})$ under (finite) flexibly weighted 
limits follows from Theorem~\ref{corcosmos1}. As every (finite) simplicial weight $W\colon\mathbf{I}\rightarrow\mathbf{S}$ has a (finite) 
flexible cofibrant replacement, it follows that for every (finite) simplicial diagram
$F\colon\mathbf{I}\rightarrow\mathbf{Cat}_{\infty}(\mathcal{C})$ and for every (pointwise finite) simplicial weight
$W\colon\mathbf{I}\rightarrow\mathbf{S}$, the $W$-weighted homotopy limit $\mathrm{holim}^W(\mathrm{Ext}(F))$ of 
the composition $\mathrm{Ext}(F)\colon\mathbf{I}\rightarrow\mathbf{Fun}(\mathcal{C}^{op},(\mathbf{S},\mathrm{QCat}))_{\mathrm{inj}}$ is 
again contained in $\mathbf{Cat}_{\infty}(\mathcal{C})$.
To deduce closure under general $(\infty,2)$-limits, let $(\mathbf{S}^+,\mathrm{Cart})$ denote the marked model structure on the category
$\mathbf{S}^+$ of marked simplicial sets \cite[Section 3.1]{luriehtt}. By a combination of \cite[Corollary 5.3.11]{ghl2fib},
\cite[Remark 0.0.4]{lgood}, a series of suitable fibrant replacements in the canonically induced model structure on the category of
$\mathbf{S}^+$-enriched categories, and the fact that the fibration categories $\mathbf{QCat}$ and $(\mathbf{S}^+,\mathrm{Cart})^f$ are 
isomorphic, one reduces (finite) $(\infty,2)$-completeness of $\mathbf{Cat}_{\infty}(\mathcal{C})$  (that is to say, of its canonically 
associated $\infty$-bicategory as referred to in Remark~\ref{remscbicat}) to the fact that $\mathbf{Cat}_{\infty}(\mathcal{C})$ is closed 
under all (finitely) weighted homotopy limits in $\mathbf{Fun}(\mathcal{C}^{op},(\mathbf{S},\mathrm{QCat}))_{\mathrm{inj}}$.

Parts 2 and 3 are covered by Proposition~\ref{proptensors} and Proposition~\ref{propcartclosed}.
%
\end{proof}

\subsection{A Segal--Yoneda lemma}

We finish this section with a version of the Yoneda lemma for small indexed $\infty$-categories and some of its applications. 

\begin{lemma}\label{lemmapreyoneda}
Let $X\in s\mathcal{C}$ be a complete Segal object in $\mathcal{C}$. Let
$F\colon\mathcal{C}^{op}\rightarrow s\mathcal{S}$ be a $\mathcal{C}$-indexed simplicial space. Then there is a binatural 
equivalence
\[\mathrm{Fun}(\mathcal{C}^{op},\mathrm{Cat}_{\infty}(\mathcal{S}))(sy(X),F)\xrightarrow{\simeq}\int\limits_{n\in\Delta^{op}}F_n(X_n)\]
of spaces which is pointwise induced by the Yoneda lemma for presheaves.
\end{lemma}
\begin{proof}
By way of Example~\ref{exmple_ends}, the Yoneda lemma induces a sequence of binatural equivalences as follows.
\begin{align*}
\mathrm{Fun}(\mathcal{C}^{op},s\mathcal{S}))(sy(X),F) \simeq s\hat{\mathcal{C}}(sy(X),F) \simeq \int\limits_{n\in\Delta^{op}}\hat{\mathcal{C}}(sy(X)_n,F_n) \simeq \int\limits_{n\in\Delta^{op}}F_n (X_n).
\end{align*}
\end{proof}

\begin{proposition}[Segal--Yoneda Lemma]\label{lemmaextyoneda}
Let $X\in\mathrm{Cat}_{\infty}(\mathcal{C})$ be a complete Segal object in $\mathcal{C}$. For every $\mathcal{C}$-indexed $\infty$-category $F\colon\mathcal{C}^{op}\rightarrow\mathrm{Cat}_{\infty}$ there is a binatural 
equivalence
\[\mathbf{Fun}(\mathcal{C}^{op},\mathbf{Cat}_{\infty})(\mathrm{Ext}(X),F)\xrightarrow{\simeq}\int\limits_{n\in\Delta^{op}}F(X_n)^{\Delta^n}\]
of $\infty$-categories. In particular, there is a binatural equivalence
\[\mathrm{Fun}(\mathcal{C}^{op},\mathrm{Cat}_{\infty})(\mathrm{Ext}(X),F)\xrightarrow{\simeq}\int\limits_{n\in\Delta^{op}}(F(X_n)^{\Delta^n})^{\simeq}\]
of underlying spaces. 
\end{proposition}
\begin{proof}
We recall that the inverse of the underlying $\infty$-category functor
$U\colon\mathrm{Cat}_{\infty}(\mathcal{S})\rightarrow\mathrm{Cat}_{\infty}$ is given by the functor
$((\cdot)^{\Delta^{\bullet}})^{\simeq}\colon\mathrm{Cat}_{\infty}\rightarrow\mathrm{Cat}_{\infty}(\mathcal{S})$. It thus suffices to 
construct a binatural equivalence
\[\left(\mathbf{Fun}(\mathcal{C}^{op},\mathbf{Cat}_{\infty})(\mathrm{Ext}(X),F)^{\Delta^{\bullet}}\right)^{\simeq}\xrightarrow{\simeq}\left(\left(\int\limits_{n\in\Delta^{op}}F(X_n)^{\Delta^n}\right)^{\Delta^{\bullet}}\right)^{\simeq}\]
of complete Segal spaces. Since $\mathbf{Fun}(\mathcal{C}^{op},\mathbf{Cat}_{\infty})$ is (functorially) cotensored over
$\mathrm{Cat}_{\infty}$, we first get a binatural equivalence
\[\left(\mathbf{Fun}(\mathcal{C}^{op},\mathbf{Cat}_{\infty})(\mathrm{Ext}(X),F)^{\Delta^{\bullet}}\right)^{\simeq}\xrightarrow{\simeq}\mathbf{Fun}(\mathcal{C}^{op},\mathbf{Cat}_{\infty})(\mathrm{Ext}(X),F^{\Delta^{\bullet}})^{\simeq}\]
of complete Segal spaces. By definition, the codomain is just the simplicial object of hom-spaces
$\mathrm{Fun}(\mathcal{C}^{op},\mathrm{Cat}_{\infty})(\mathrm{Ext}(X),F^{\Delta^{\bullet}})$.
Now, the $\mathcal{C}$-indexed $\infty$-category $\mathrm{Ext}(X)$ is by definition the (pointwise) underlying
$\infty$-category $U(sy(X))$ of the $\mathcal{C}$-indexed complete Segal space $sy(X)$. Thus,
\[sy(X)\simeq(\mathrm{Ext}(X)^{\Delta^{\bullet}})^{\simeq},\]
and we obtain a binatural equivalence
\[\mathrm{Fun}(\mathcal{C}^{op},\mathrm{Cat}_{\infty})(\mathrm{Ext}(X),F^{\Delta^{\bullet}})\xrightarrow{\simeq}\mathrm{Fun}(\mathcal{C}^{op},\mathrm{Cat}_{\infty}(\mathcal{S}))(sy(X),((F^{\Delta^{\bullet}})^{\Delta^{(\cdot)}})^{\simeq})\]
of complete Segal spaces. By Lemma~\ref{lemmapreyoneda}, the right hand side is binaturally equivalent to the end
\[\int\limits_{n\in\Delta^{op}}((F^{\Delta^{\bullet}}(X_n))^{\Delta^n})^{\simeq}\]
of spaces. But $\infty$-categorical cotensors of indexed $\infty$-categories are computed pointwise and commute with one another, while
ends commute both with $\infty$-categorical cotensors and the core construction. We thus obtain a binatural equivalence
\[\int\limits_{n\in\Delta^{op}}((F^{\Delta^{\bullet}}(X_n))^{\Delta^n})^{\simeq}\simeq\left(\left(\int\limits_{n\in\Delta^{op}}F(X_n)^{\Delta^n}\right)^{\Delta^{\bullet}}\right)^{\simeq}.\]
Concatenation of this sequence of binatural equivalences of complete Segal spaces finishes the proof.
\end{proof}


\begin{remark}\label{remextyoneda}
It is easy to see that the Segal--Yoneda lemma recovers the standard Yoneda lemma for the base $\mathcal{C}$ when applied to internal
$\infty$-groupoids: For any object $C\in\mathcal{C}$ and any presheaf $F$, by commutativity of Diagram~(\ref{diagyonedaext}) we have
$\mathrm{Ext}(c(C))\simeq y(C)$ and hence
\[\hat{\mathcal{C}}(yC,F)\simeq\mathrm{Fun}(\mathcal{C}^{op},\mathrm{Cat}_{\infty})(\mathrm{Ext}(c(C)),F)\simeq\int\limits_{n\in\Delta^{op}}F(C)^{\Delta^n}.\]
As each $F(C)$ is a space, the end on the right hand side is the totalization (and hence the homotopy-limit) of the constant simplicial 
space $F(C)$. This however is equivalent to $F(C)$ itself. It follows more generally that
$\mathbf{Fun}(\mathcal{C}^{op},\mathbf{Cat}_{\infty})(yC,F)\simeq F(C)$ for any given general $\mathcal{C}$-indexed $\infty$-category 
$F$.  
\end{remark}

The Segal--Yoneda lemma states that the small indexed $\infty$-categories $\mathrm{Ext}(X)$ are freely generated by the diagram of 
identities $1_{X_n}\in\mathrm{Ext}(X)(X_n)$ for $n\geq 0$. Indeed, informally, the equivalence in Proposition~\ref{lemmaextyoneda} maps a 
natural transformation
\[\alpha\colon\mathrm{Ext}(X)\rightarrow F\]
to the tuple
\[(\alpha_{X_{n}}(1_{X_{n}})|n\geq 0)\in\int\limits_{n\in\Delta^{op}}(F(X_n)^{\Delta^n})^{\simeq}.\]
Formally, Proposition~\ref{lemmaextyoneda} equivalently states that the following.

\begin{corollary}\label{corextyoneda}
Let $\mathcal{C}$ be an $\infty$-category and $X\in\mathrm{Cat}_{\infty}(\mathcal{C})$. Then there is a 2-cell
\[\xymatrix{
\Delta\ar[r]^{\Delta^{\bullet}}\ar[d]_{X^{op}} & \mathrm{Cat}_{\infty} \\
\mathcal{C}^{op}\ar@/_1pc/[ur]_{\mathrm{Ext}(X)}\ar@{}[ur]|(.5){\Downarrow 1_{X_{\bullet}}} &
}\]
that exhibits $\mathrm{Ext}(X)$ as the left Kan extension $(X^{op})_!(\Delta^{\bullet})$ of the canonical inclusion $\Delta^{\bullet}\colon\Delta\hookrightarrow\mathrm{Cat}_{\infty}$ along $X^{op}\colon\Delta\rightarrow\mathcal{C}$.
\end{corollary}
\begin{proof}
The $\infty$-category $\mathrm{Cat}_{\infty}$ is cocomplete, and so there is a global left Kan extension functor
$(X^{op})_!\colon\mathrm{Fun}(\Delta,\mathrm{Cat}_{\infty})\rightarrow\mathrm{Fun}(\mathcal{C}^{op},\mathrm{Cat}_{\infty})$.
For every $F\colon\mathcal{C}^{op}\rightarrow\mathrm{Cat}_{\infty}$ we obtain a binatural equivalence
\begin{align*}
\mathrm{Fun}(\mathcal{C}^{op},\mathrm{Cat}_{\infty})((X^{op})_!(\Delta^{\bullet}),F) & \simeq\mathrm{Fun}(\Delta,\mathrm{Cat}_{\infty})((\Delta^{\bullet},(X^{op})^{\ast}F) \\
 & \simeq\int\limits_{n\in\Delta}(F(X_n)^{\Delta^n})^{\simeq}
\end{align*}
again via Example~\ref{exmple_ends}. This end (despite the subtle difference in the indexing domain) is exactly the end that 
computes binaturally the hom-space $\mathrm{Fun}(\mathcal{C}^{op},\mathrm{Cat}_{\infty})(\mathrm{Ext}(X),F)$ by Lemma~\ref{lemmaextyoneda}.
Indeed, the diagram
\[\xymatrix{
\mathrm{Tw}(\Delta^{op})\ar@{->>}[d]\ar[r]^{\simeq} & \mathrm{Tw}(\Delta)\ar@{->>}[d]\\
\Delta\times\Delta^{op}\ar[r]^{\simeq}_{(\pi_2,\pi_1)}\ar[d]_{(X^{op},(F(\cdot)^{\Delta^{\bullet}})^{\simeq})} & \Delta^{op}\times\Delta\ar[d]^{(\Delta^{\bullet},F\circ X^{op})} \\
\mathcal{C}^{op}\times\hat{\mathcal{C}}\ar[d]_{\mathrm{ev}} & \mathrm{Cat}_{\infty}^{op}\times \mathrm{Cat}_{\infty}\ar@/^1pc/[dl]^{\mathrm{Hom}} \\
\mathcal{S}
}\]
commutes. Here, the top horizontal equivalence is given in \cite[Remark 8.1.1.7]{kerodon}. It follows that
$\mathrm{Ext}(X)\simeq (X^{op})_!(\Delta^{\bullet})$. The corresponding 2-cell $1_{X_{\bullet}}$ in the statement is given by the element in
$\mathrm{Fun}(\Delta,\mathrm{Cat}_{\infty})((\Delta^{\bullet},(X^{op})^{\ast}\mathrm{Ext}(X))$ corresponding to the identity $1_{\mathrm{Ext}(X)}$ via Proposition~\ref{lemmaextyoneda}.
\end{proof}

Corollary~\ref{corextyoneda} generalizes the fact that the ordinary Yoneda lemma equivalently states that the 2-cell
\[\xymatrix{
\ast\ar[r]^{\{\ast\}}\ar[d]_{\{C\}} & \mathcal{S} \\
\mathcal{C}^{op}\ar@/_1pc/[ur]_{yC}\ar@{}[ur]|(.5){\Downarrow 1_{C}}  &
}\]
is a left Kan extension for every $\infty$-category $\mathcal{C}$ and every object $C\in\mathcal{C}$.

\begin{remark}[Relation to the generalized Yoneda lemma of Riehl and Verity]\label{remarkrvyoneda}
The category theoretical literature is not short of formulations and generalizations of the Yoneda lemma in various contexts. One 
such generalization very much relevant for the context at hand is the generalized Yoneda lemma for cartesian 
fibrations of Riehl and Verity \cite[Theorem 5.7.3, Corollary 5.7.19]{riehlverityelements} in the standard case of 
an $\infty$-cosmos of $(\infty,1)$-categories. Applied to a simplicial diagram
$X\colon\Delta^{op}\rightarrow\mathcal{C}$ and a cartesian fibration $p$ over $\mathcal{C}$, it yields an 
equivalence
\[\mathrm{Fun}_{\mathcal{C}}^{\mathrm{cart}}(\mathcal{C}\downarrow X,p)\simeq\mathrm{Fun}_{\mathcal{C}}(X,p)\]
from the quasi-category of cartesian functors between the comma-object $\mathcal{C}\downarrow X$ and $p$ over
$\mathcal{C}$, and the quasi-category of functors between $X$ and $p$ over $\mathcal{C}$. Via
\cite[Proposition 6.9, Proposition 7.1]{ghnlaxlimits}, for every simplicial object $X\in s\mathcal{C}$ and every 
$F\colon\mathcal{C}^{op}\rightarrow\mathrm{Cat}_{\infty}$ it equivalently states the existence of a binatural 
equivalence
\[\mathbf{Fun}(\mathcal{C}^{op},\mathbf{Cat}_{\infty})(\mathrm{St}(\mathcal{C}\downarrow X),F)\simeq\int\limits_{n\in\Delta^{op}}F(X_n)^{\Delta_{/n}}\]
of quasi-categories, where the right hand side computes the oplax limit of the composition
$F\circ X^{op}\colon\Delta\rightarrow\mathrm{Cat}_{\infty}$ \cite[Definition 2.9]{ghnlaxlimits}.

In comparison to Proposition~\ref{lemmaextyoneda}, we note that the overcategories $\Delta_{/n}$ are very 
different from the categories $\Delta^n$, and generally the associated ends differ as well. Indeed, the 
externalization $\mathrm{Ext}(X)\twoheadrightarrow\mathcal{C}$ (considered as cartesian fibration over
$\mathcal{C}$ via its Unstraightening) and the comma-object $\mathcal{C}\downarrow X$ associated to $X$ are 
generally different concepts. For example, if $\mathcal{C}$ is the terminal $\infty$-category $\ast$ and 
$X=c(\ast)$ is the unique simplicial object in $\mathcal{C}$, one computes that
$\mathcal{C}\downarrow X=\Delta^{op}$ and $\mathrm{Ext}(X)=\ast$ in $\mathrm{Cart}(\ast)=\mathrm{Cat}_{\infty}$.
Indeed, for objects $C\in\mathcal{C}$ \emph{when considered as constant simplicial objects $c(C)$ in $\mathcal{C}$}, the comma-object 
$\mathcal{C}\downarrow c(C)$ does generally not even compute the representable $y(C)$. That is, because the comma-object is overloaded 
with non-contractible structure coming from $\Delta^{op}$; one rather has to consider $C$ in the contractible context
$\{C\}\colon\ast\rightarrow\mathcal{C}$ for the comma-object $\mathcal{C}\downarrow\{C\}$ to recover 
the corresponding representable. As we have seen earlier, $\mathrm{Ext}(c(C))$ on the other hand is equivalent to 
$y(C)$. In summary, we see that the two Yoneda lemmas generalize the classical Yoneda lemma to a 
classification of cartesian functors out of two different classes of fibrations. Both reduce to the same statement 
when applied to the subclass of representable right fibrations however.
\end{remark}

We finish this section with two applications of Lemma~\ref{lemmaextyoneda}. First, for later use in Section~\ref{secmodcat} we recall 
that if $\mathcal{C}$ is a complete $\infty$-category, then the externalization functor
$\mathrm{Ext}\colon\mathrm{Cat}_{\infty}(\mathcal{C})\rightarrow\mathrm{Fun}(\mathcal{C}^{op},\mathrm{Cat}_{\infty})$ induces an 
equivalence
\begin{align}\label{equcompleteextequiv}
\mathrm{Ext}\colon\mathrm{Cat}_{\infty}(\mathcal{C})\rightarrow\mathrm{RAdj}(\mathcal{C}^{op},\mathrm{Cat}_{\infty})
\end{align}
into the full sub-$\infty$-category spanned by the right adjoint functors by corestriction as shown in \cite[Proposition 5.23]{rs_comp}. 
Indeed, $\mathrm{Ext}(X)$ is the nerve of the cosimplicial object $X^{op}\colon\Delta\rightarrow\mathcal{C}^{op}$ whose 
left adjoint is in this case given by the left Kan extension of $X^{op}$ along the generating inclusion
$y\colon\Delta\rightarrow\mathrm{Cat}_{\infty}$. The adjunction $y_!(X^{op})\adj\mathrm{Ext}(X)$ for any given 
$X\in\mathrm{Cat}_{\infty}(\mathcal{C})$ restricts the equivalence
$y_!\colon \mathrm{Fun}(\Delta,\mathcal{C}^{op})\xrightarrow{\simeq}\mathrm{LAdj}(s\mathcal{S},\mathcal{C}^{op})$ to an 
equivalence
\[y_!\colon \mathrm{Cat}_{\infty}(\mathcal{C})^{op}\xrightarrow{\simeq}\mathrm{LAdj}(\mathrm{Cat}_{\infty}(\mathcal{S}),\mathcal{C}^{op})\xrightarrow{\simeq}\mathrm{LAdj}(\mathrm{Cat}_{\infty},\mathcal{C}^{op}).\]
The description of externalization as a right adjoint under the given assumption thereby holds not only pointwise but functorially
(as to be expected) as the following proposition shows.

\begin{proposition}\label{propextasnerve}
Suppose $\mathcal{C}$ is a complete $\infty$-category. Then the triangle
\begin{align}\label{diagpropextasnerve}
\begin{gathered}
\xymatrix{
\mathrm{Cat}_{\infty}(\mathcal{C})\ar[dr]_{\mathrm{Ext}}\ar[r]^(.4){y_!^{op}} & \mathrm{LAdj}(\mathrm{Cat}_{\infty},\mathcal{C}^{op})^{op}\ar[d]_{\simeq}^{\rho} \\
 & \mathrm{RAdj}(\mathcal{C}^{op},\mathrm{Cat}_{\infty})\\
}
\end{gathered}
\end{align}
commutes up to equivalence. Here, the vertical equivalence $\rho$ from left to right adjoints is as defined in 
Diagram~(\ref{equdefpla}) for $F$ the identity.
\end{proposition}
\begin{proof}
We compute that the space of natural transformations from $\mathrm{Ext}$ to the composition
$\rho\circ y_!^{op}$ is equivalent to 
\begin{align*}
\int\limits_{X\in\mathrm{Cat}_{\infty}(\mathcal{C})}\mathrm{Fun}(\mathcal{C}^{op},\mathrm{Cat}_{\infty})(\mathrm{Ext}(X),\rho(y_!^{op}(X)))^{\simeq} & \simeq \int\limits_{X\in\mathrm{Cat}_{\infty}(\mathcal{C})}\int\limits_{n\in\Delta^{op}}(\rho(y_!^{op}(X))(X_n)^{\Delta^n})^{\simeq} \\
 & \simeq \int\limits_{X\in\mathrm{Cat}_{\infty}(\mathcal{C})}\int\limits_{n\in\Delta^{op}}\mathrm{Cat}_{\infty}(\Delta^n,\rho(y_!^{op}(X))(X_n)) \\
& \simeq \int\limits_{X\in\mathrm{Cat}_{\infty}(\mathcal{C})}\int\limits_{n\in\Delta^{op}}\mathcal{C}^{op}(y_!^{op}(X)(\Delta^n),X_n) \\
& \simeq \int\limits_{X\in\mathrm{Cat}_{\infty}(\mathcal{C})}\int\limits_{n\in\Delta^{op}}\mathcal{C}^{op}(X_n,X_n) \\
& \simeq \int\limits_{X\in\mathrm{Cat}_{\infty}(\mathcal{C})}\int\limits_{n\in\Delta^{op}}\mathcal{C}(X_n,X_n) \\
& \simeq \int\limits_{X\in\mathrm{Cat}_{\infty}(\mathcal{C})}\mathrm{Cat}_{\infty}(\mathcal{C})(X,X) \\
& \simeq \mathrm{Fun}(\mathrm{Cat}_{\infty}(\mathcal{C}),\mathrm{Cat}_{\infty}(\mathcal{C}))(\mathrm{id},\mathrm{id}).
\end{align*}
Let $\alpha\colon\mathrm{Ext}\Rightarrow\rho\circ y_!^{op}$ be the essentially unique natural transformation that corresponds to the 
identity on $\mathrm{id}$ under this chain of equivalences. We are to show that this natural transformation is a (pointwise) equivalence. 
Therefore, let $X\in\mathrm{Cat}_{\infty}(\mathcal{C})$. 
We have seen in Corollary~\ref{corextyoneda} that $\mathrm{Ext}(X)$ is the left Kan extension
$X^{op}_!(\Delta^{\bullet})$. 
\[\xymatrix{
\Delta\ar[rr]^{\Delta^{\bullet}}\ar[d]_{X^{op}} & & \mathrm{Cat}_{\infty} \\
\mathcal{C}^{op}\ar@/_1pc/[urr]|{\mathrm{Ext}(X)}\ar@{}[urr]|(.5){\Downarrow 1_{X_{\bullet}}}\ar@/_3pc/[urr]_{\rho y_!^{op}(X)}^{\Downarrow \alpha_X} &
}\]
Both $\mathrm{Ext}(X)$ and $\rho y_!^{op}(X)$ are a right adjoint to the left Kan extension $y_!(X^{op})$ each by construction, and hence 
they are equivalent to one another. Under this equivalence and the adjunction $X^{op}_!\adj(X^{op})^{\ast}$, the natural transformation
$\alpha_X\colon X^{op}_!(\Delta^{\bullet})\rightarrow X^{op}_!(\Delta^{\bullet})$ corresponds precisely to the unit
\[1_{X_{\bullet}}\colon\Delta^{\bullet}\rightarrow (X^{op})^{\ast}X^{op}_!(\Delta^{\bullet})\]
by definition of $\alpha_X$. It follows that $\alpha_X$ is a natural equivalence. 
\end{proof} 

\begin{remark}\label{rem_pressheaves}
Whenever $\mathcal{C}$ is presentable (and hence in particular complete), the Adjoint Functor Theorem and the equivalence 
(\ref{equcompleteextequiv}) together imply that an indexed $\infty$-category $F\colon\mathcal{C}^{op}\rightarrow\mathrm{Cat}_{\infty}$ is 
small if and only if it is small limit preserving. If $\mathcal{C}$ is in fact an $\infty$-topos, this identifies the $(\infty,2)$-category 
of $\infty$-categories internal to $\mathcal{C}$ with the full sub-$(\infty,2)$-category of sheaves of $\infty$-categories over
$\mathcal{C}$. This latter characterization is stated in \cite[Section 3.5]{martini_yoneda}. In loc.\ cit.\ the author also computes that 
\[\mathrm{Ext}(X)\simeq U(\mathcal{C}(\phv,(X^{\Delta^{\bullet}})_0))\simeq\mathbf{Cat}_{\infty}(\mathcal{C})(c(\cdot),X)\]
for all complete Segal objects $X$ in an $\infty$-topos $\mathcal{C}$. By way of the enhanced mapping $\infty$-category functor of 
Remark~\ref{rem2catstrdirect}, this can be more generally shown to hold for all finitely complete $\infty$-categories $\mathcal{C}$.
\end{remark}

Second, Lemma~\ref{lemmaextyoneda} can also be used to characterize the $\infty$-category of internal presheaves 
over an internal $\infty$-category $X$ in any $\infty$-category $\mathcal{C}$ with pullbacks. Therefore we denote 
the canonical indexing of $\mathcal{C}$ over itself -- defined as the Straightening of the target fibration
$t\colon\mathrm{Fun}(\Delta^1,\mathcal{C})\twoheadrightarrow\mathcal{C}$ -- by
$\mathcal{C}_{/(\cdot)}\colon\mathcal{C}^{op}\rightarrow\mathrm{Cat}_{\infty}$ (or with codomain
$\mathrm{CAT}_{\infty}$ in case $\mathcal{C}$ is itself large). 

\begin{definition}
Let $X\in\mathrm{Cat}_{\infty}(\mathcal{C})$. The $\infty$-category $[X,\mathcal{C}]$ of internal covariant presheaves over 
$X$ in $\mathcal{C}$ is the $\infty$-category
$\mathbf{Fun}(\mathcal{C}^{op},\mathbf{Cat}_{\infty})(\mathrm{Ext}(X),\mathcal{C}_{/(\cdot)})$.
\end{definition}

\begin{construction}\label{defintcore}
For every $X\in\mathrm{Cat}_{\infty}(\mathcal{C})$ there is an $X^{op}\in\mathrm{Cat}_{\infty}(X)$ defined in such a way that the externalization of $X^{op}$ 
computes the composition
\[\mathcal{C}^{op}\xrightarrow{\mathrm{Ext}(X)}\mathrm{Cat}_{\infty}\xrightarrow{(\cdot)^{op}}\mathrm{Cat}_{\infty}.\]
This can be seen quickly for instance from \cite[Theorem 5.15]{rs_comp} together with the fact that the core of the opposite of an
$\infty$-category $\mathcal{D}$ is naturally equivalent to the core of $\mathcal{D}$. The opposite $X^{op}$ also can be defined directly 
as the precomposition of $X$ with $(\cdot)^{op}\colon\Delta\rightarrow\Delta$ as in \cite[Section 1.2.1]{luriehtt}.
\end{construction}

\begin{definition}
Let $X\in\mathrm{Cat}_{\infty}(\mathcal{C})$. The $\infty$-category of internal (contravariant) presheaves over $X$ in $\mathcal{C}$ is the 
$\infty$-category $[X^{op},\mathcal{C}]$.
\end{definition}

\begin{example}
If we represent an $\infty$-category $\mathcal{C}$ as a complete Segal object $(\mathcal{C}^{\Delta^{\bullet}})^{\simeq}$ in
$\mathcal{S}$, then $\mathrm{Ext}((\mathcal{C}^{\Delta^{\bullet}})^{\simeq})$ is just the ``naive'' indexing
$\mathcal{C}^{(-)}\colon\mathcal{S}^{op}\rightarrow\mathrm{CAT}_{\infty}$ of
$\mathcal{C}$ as shown in \cite[Example 5.12]{rs_comp}. More precisely, the functor
$(\cdot)^{(-)}\colon\mathrm{CAT}_{\infty}\rightarrow\mathrm{Fun}(\mathcal{S}^{op},\mathrm{CAT}_{\infty})$ -- which assigns to an
$\infty$-category $\mathcal{C}$ the $\mathcal{S}$-indexed $\infty$-category $\mathcal{C}^{(-)}$ -- is the composition
\[\mathrm{CAT}_{\infty}\xrightarrow{((\cdot)^{\Delta^{\bullet}})^{\simeq}}\mathrm{Cat}_{\infty}(\mathcal{S})\xrightarrow{\mathrm{Ext}}\mathrm{Fun}(\mathcal{S}^{op},\mathrm{CAT}_{\infty})\]
of the equivalence $((\cdot)^{\Delta^{\bullet}})^{\simeq}$ with the externalization functor on $\mathcal{S}$. It thus is fully faithful 
by Lemma~\ref{corsegalyonedaff}, and yields a functor
$(\cdot)^{(-)}\colon\mathbf{QCAT}\rightarrow\mathbf{Fun}(\mathcal{S}^{op},\mathbf{QCAT})$ that induces 
equivalences of $\infty$-hom-categories as it preserves $\infty$-categorical cotensors as well. 
The canonical indexing $\mathcal{S}_{/(\cdot)}$ of $\mathcal{S}$ over itself is naturally equivalent to the naive indexing
$\mathcal{S}^{(\cdot)}\colon\mathcal{S}^{op}\rightarrow\mathrm{CAT}_{\infty}$ as well (by the unmarked Straightening 
construction, see \cite[Theorem 2.2.1.2]{luriehtt}). It follows that the $\infty$-category
$[(\mathcal{C}^{\Delta^{\bullet}})^{\simeq},\mathcal{S}]$ of $\mathcal{S}$-internal covariant presheaves over
$(\mathcal{C}^{\Delta^{\bullet}})^{\simeq}$ is (as to be expected) just the $\infty$-category
\[\mathbf{Fun}(\mathcal{S}^{op},\mathbf{QCAT})(\mathcal{C}^{(\cdot)},\mathcal{S}^{(\cdot)})\simeq\mathbf{QCAT}(\mathcal{C},\mathcal{S})=\mathrm{Fun}(\mathcal{C},\mathcal{S})\]
of covariant presheaves over $\mathcal{C}$. 
The same goes for contravariant presheaves over $\mathcal{C}$.
\end{example}

\begin{corollary}
Let $X\in\mathrm{Cat}_{\infty}(\mathcal{C})$. Then there are natural equivalences
\[[X,\mathcal{C}]\simeq\int\limits_{n\in\Delta^{op}}(\mathcal{C}_{/X_n})^{\Delta^n},\]
\[[X^{op},\mathcal{C}]\simeq\int\limits_{n\in\Delta^{op}}(\mathcal{C}_{/(X^{op})_n})^{\Delta^n}\]
of $\infty$-categories.\qed 
\end{corollary}

\begin{corollary}\label{corintpretopoi}
If $\mathcal{C}$ is an $\infty$-topos, then for every $X\in\mathrm{Cat}_{\infty}(\mathcal{C})$, the $\infty$-category
$[X,\mathcal{C}]$ of internal covariant presheaves over $X$ is again an $\infty$-topos. Trivially, the same 
applies to $[X^{op},\mathcal{C}]$.
\end{corollary}
\begin{proof}
If $\mathcal{C}$ is an $\infty$-topos, the diagram
\begin{align}\label{equcorintpretopoi}
\Delta^{op}\times\Delta & \rightarrow\mathrm{CAT}_{\infty} \\
\notag ([n],[m]) & \mapsto(\mathcal{C}_{/X_m})^{\Delta^n}
\end{align}
factors through the $\infty$-category $\mathrm{LTop}\subset\mathrm{CAT}_{\infty}$, where $\mathrm{LTop}$ denotes 
the opposite $\infty$-category of $\infty$-toposes and geometric morphisms \cite[Definition 6.3.1.5]{luriehtt}. But the inclusion
$\mathrm{LTop}\subset\mathrm{CAT}_{\infty}$ creates all small limits \cite[Proposition 6.3.2.3]{luriehtt}, and so
$[X,\mathcal{C}]\simeq\mathrm{lim}((\ref{equcorintpretopoi})\circ p_{\Delta})$ is an $\infty$-topos, where
$p_{\Delta}\colon\mathrm{Tw}(\Delta)\twoheadrightarrow\Delta^{op}\times\Delta$ denotes the canonical projection.
\end{proof}

\begin{remark}
The argument in Corollary~\ref{corintpretopoi} applies more generally to every (potentially super-large) countably complete
$\infty$-category $K$ of $\infty$-categories (whose inclusion $K\subset\mathrm{CAT}_{\infty}$ is fully faithful on $n$-cells for
$n\geq 2$) such that $\mathcal{C}\in K$ implies
\begin{enumerate}
\item $\mathcal{C}^{\Delta^n}\in K$ for all $n\geq 0$ and $\mathcal{C}_{/C}\in K$ for all $C\in\mathcal{C}$,
\item $\alpha^{\ast}\colon\mathcal{C}^{\Delta^m}\rightarrow\mathcal{C}^{\Delta^n}$ is in $K$ for all $\alpha\colon n\rightarrow m$, and
\item $f^{\ast}\colon\mathcal{C}_{/D}\rightarrow\mathcal{C}_{/C}$ is contained in $K$ for all morphisms
$f\colon C\rightarrow D$ in $\mathcal{C}$.
\end{enumerate}
It thus applies for instance also to the $\infty$-category of presentable $\infty$-categories with universal colimits (and cocontinuous 
functors between them).
\end{remark}

\section{Model categorical presentations}\label{secmodcat}

Whenever a given $\infty$-category $\mathcal{C}$ is the homotopy $\infty$-category of a model category $\mathbb{M}$, we can
relate the main results of Section~\ref{secformal} such as Theorem~\ref{corcosmos1} and Proposition~\ref{proptensors} to the classical 
theory of model categories. Notably, every model category $\mathbb{M}$ induces a Reedy 
model structure on the category $s\mathbb{M}$ of simplicial objects on $\mathbb{M}$, and bicompleteness of $\mathbb{M}$ furthermore 
induces a simplicial enrichment of the category $s\mathbb{M}$ (explicitly presented in \cite{duggersimp}). While the Reedy model 
structure generally does not cohere with this enrichment, we will see that the induced 
Reedy fibration category structure on the full subcategory $\mathrm{Cat}_{\infty}(\mathbb{M})\subset s\mathbb{M}$ of complete Segal 
objects in $\mathbb{M}$ is always $(\mathbf{S},\mathrm{QCat})$-enriched, and that the induced fibration category structure on the full 
subcategory $\mathrm{Gpd}_{\infty}(\mathbb{M})\subset s\mathbb{M}$ of complete Segal groupoids in $\mathbb{M}$ is always
$(\mathbf{S},\mathrm{Kan})$-enriched. The latter are the ``frames'' in the sense of \cite{hovey} and the ``complete Bousfield-Segal 
objects'' in the sense of \cite{bergner2, rs_bspaces}. The two fibration categories in particular each form an $\infty$-cosmos in the 
weaker sense of \cite{rvyoneda}. By virtue of Example~\ref{expleenrmodtofibcat}, this specialises to the fact that the associated model 
categories $(s\mathbb{M},\mathrm{Cat}_{\infty})$ \cite{rvyoneda} and $(s\mathbb{M},\mathrm{Gpd}_{\infty})$ \cite{duggersimp} are
$(\mathbf{S},\mathrm{QCat})$-enriched and $(\mathbf{S},\mathrm{Kan})$-enriched, respectively, in case $\mathbb{M}$ is combinatorial and 
left proper. 

We further define a simplicially enriched model categorical right derived externalization functor
\[\mathbb{R}\mathbf{Ext}\colon\mathbf{Cat}_{\infty}(\mathbb{M})\rightarrow\mathbf{Fun}(\mathcal{L}_{\Delta}(\mathbb{M},W)^{op},\mathbf{QCat}),\]
elaborate on instances of this functor that arise in classical simplicial homotopy theory, and 
show that it is cosmological in a suitable sense.
We show that $\mathbb{R}\mathbf{Ext}$ lifts the $\infty$-categorical externalization construction of Section~\ref{secformal} under 
additional assumptions on $\mathbb{M}$, and use this to show that $\mathrm{Cat}_{\infty}(\mathcal{C})$ is an $(\infty,1)$-localic
$(\infty,2)$-topos as to be defined in Section~\ref{subsecrdevext} whenever $\mathcal{C}$ is an $\infty$-topos.\\

We fix a model category $\mathbb{M}$ for this entire section. The homotopy $\infty$-category $\mathbb{M}[\mathcal{W}^{-1}]$ of
$\mathbb{M}$ \cite[Definition 1.3.4.15]{lurieha} will be denoted by $\mathcal{C}$.

\subsection{Complete Segal objects and model categorical externalization}\label{secsubcsointro}

We recall that the derived mapping spaces $\mathbb{M}(A,B)^h$ of $\mathbb{M}$ compute the mapping spaces $\mathcal{C}(A,B)$ and may be 
calculated in terms of coframes on its cofibrant objects $A$, or frames on its fibrant objects $B$ \cite[Chapter 5]{hovey}. 
Here, a frame on a fibrant object $B$ is a homotopically constant Reedy fibrant simplicial object $X$ in $\mathbb{M}$ 
such that $X_0\cong B$. A frame is the same thing as a complete Bousfield-Segal object in the sense of \cite[Section 6]{bergner2} by
\cite[Lemma 6.4]{rs_bspaces}. We will refer to such as \emph{complete Segal groupoids} in $\mathbb{M}$ to cohere with the
$\infty$-categorical terminology. We denote by $\mathrm{Gpd}_{\infty}(\mathbb{M})\subseteq s\mathbb{M}$ the full subcategory of complete 
Segal groupoids in $\mathbb{M}$. 
In \cite[Corollary 5.4.4]{hovey}, Hovey shows that for every fibrant object $B$ in $\mathbb{M}$ 
and every complete Segal groupoid $X$ on $B$, the left Kan extension of $X^{op}\colon\Delta\rightarrow\mathbb{M}$ along the Yoneda 
embedding $y\colon\Delta\rightarrow\mathbf{S}$ yields a Quillen adjunction
\[\xymatrix{
y_!(X^{op})\colon(\mathbf{S},\mathrm{Kan})\ar@<1ex>[r]\ar@{}[r]|(.55){\rotatebox[origin=c]{90}{$\vdash$}} & \mathbb{M}^{op}\colon\mathbb{M}(\phv,X_{\bullet})\ar@<1ex>[l]
}\]
such that $\mathbb{M}(A,X_{\bullet})\simeq\mathbb{M}(A,B)^h$ for all cofibrant objects $A$. Thus, we can assign to every 
simplicial object $X$ in $\mathbb{M}$ a right adjoint $\mathbb{M}^{op}\rightarrow\mathbf{S}$ which is a right 
Quillen functor whenever $X$ is a complete Segal groupoid. Since every fibrant object in $\mathbb{M}$ is the object component of some 
homotopically unique complete Segal groupoid \cite[Theorem 5.2.8]{hovey}, we may think of these right Quillen functors
\[\mathbb{M}(\phv,X_{\bullet})\colon\mathbb{M}^{op}\rightarrow(\mathbf{S},\mathrm{Kan})\] 
as representable $\mathbb{M}$-indexed Kan complexes.
It is not hard to show that in fact every right Quillen functor of type $\mathbb{M}^{op}\rightarrow(\mathbf{S},\mathrm{Kan})$ is exactly 
of the form $\mathbb{M}(\phv,X_{\bullet})$ for some complete Segal groupoid $X$ in $\mathbb{M}$ (see Proposition~\ref{propextrQf} 
below).

More generally, the functor $\mathbb{M}(\phv,X_{\bullet})\colon\mathbb{M}^{op}\rightarrow\mathbf{S}$ is well-defined for every 
simplicial object in $\mathbb{M}$ and features prominently under the name ``$X/\blank$ '' in the Joyal-Tierney calculus developed in
\cite[Section 7]{jtqcatvsss}. Indeed, for a simplicial object $X\in s\mathbb{M}$, the left Kan extension
\[\xymatrix{
\Delta\ar[r]^{X^{op}}\ar[d]_{y} & \mathbb{M}^{op} \\
\mathbf{S}\ar@{-->}[ur]_(.6){y_!(X^{op})},
}\]
of $X^{op}$ along the Yoneda embedding always exhibits a right adjoint, explicitly given by the nerve
$\mathbb{M}^{op}(X_{\bullet},\phv)\colon\mathbb{M}^{op}\rightarrow\mathbf{S}$. By definition, for every $X\in s\mathbb{M}$, we have natural
isomorphisms $y_!(X^{op})(\Delta^n)\cong X_n$. The left adjoint $y_!(X^{op})$ takes the $n$-th boundary 
inclusion $\delta^n\colon\partial\Delta^n\hookrightarrow\Delta^n$ to the $n$-th matching object $X_n\rightarrow M_nX$ in
$\mathbb{M}$, and the $n$-th spine inclusion $S_n\hookrightarrow\Delta^n$ to the $n$-th Segal map
$X_n\rightarrow\{S_n,X\}=X_1\times_{X_0}\dots\times_{X_0}X_1$.

\begin{definition}\label{defsegalobj}
A simplicial object $X$ in $\mathbb{M}$ is Reedy fibrant if the matching object $X_n\rightarrow M_nX$ is a fibration for all $n\geq 0$. 
A Reedy fibrant simplicial object $X$ in $\mathbb{M}$ is a Segal object if the Segal maps
\[X_n\rightarrow X_1\times_{X_0}\dots\times_{X_0}X_1\]
are homotopy equivalences for all $n\geq 2$. A Segal object $X$ in $\mathbb{M}$ is complete if the functor
$y_!(X^{op})\colon\mathbf{S}^{op}\rightarrow\mathbb{M}$ takes the endpoint inclusion
$\Delta^0\rightarrow I\Delta^1$ to an acyclic fibration. The full subcategory of complete Segal objects in $\mathbb{M}$ will be denoted 
by $\mathrm{Cat}_{\infty}(\mathbb{M})\subset s\mathbb{M}$.
\end{definition}

The definition of completeness is a straight-forward generalization of Rezk's original definition of completeness for 
Segal spaces. It is used in \cite{jtqcatvsss} to define a Quillen equivalence from Rezk's model structure
$(s\mathbf{S},\mathrm{Cat}_{\infty})$ for complete Segal spaces to the Joyal model structure
$(\mathbf{S},\mathrm{QCat})$. A comparison to the internal definition of completeness as given in Definition~\ref{defcompleteness}, and 
as used e.g.\ in \cite{lgood} and \cite{rasekhcart}, is given in \cite{rs_uc}. Definition~\ref{defsegalobj} is chosen in such a 
way that we immediately obtain the following characterization. 

\begin{lemma}\label{lemmarQuillenchar}
The left Kan extension $y_!(X^{op})\colon\mathbf{S}^{op}\rightarrow\mathbb{M}$ takes
\begin{itemize}
\item boundary inclusions (and hence all monomorphisms) in $\mathbf{S}$ to fibrations in $\mathbb{M}$ if and only if $X$ 
is Reedy fibrant;
\item furthermore inner horn inclusions (and hence all inner anodyne morphisms) in $\mathbf{S}$ to trivial fibrations if 
and only if $X$ is a Segal object;
\item furthermore the endpoint inclusion $\Delta^0\rightarrow I\Delta^1$ to a trivial fibration if and only if $X$ 
is a complete Segal object.
\end{itemize}
\end{lemma}
\begin{proof}
This is the dual to \cite[Proposition 3.2]{rs_uc}. In a nutshell, Part 1 follows directly from a saturation argument and the fact that
$y_!(X^{op})$ takes the boundary inclusions to the corresponding matching objects. Similarly, a Reedy fibrant simplicial object 
$X$ in $\mathbb{M}$ is a Segal object if and only if the functor $y_!(X^{op})\colon\mathbf{S}\rightarrow\mathbb{M}^{op}$ 
takes the spine inclusions to acyclic cofibrations in $\mathbb{M}^{op}$. The latter holds if and only if $y_!(X^{op})$ 
takes the class of inner horn inclusions to acyclic fibrations by \cite[Lemma 1.21, Lemma 3.5]{jtqcatvsss}. 
\end{proof}

The global left Kan extension functor $y_!\colon \mathrm{Fun}(\Delta,\mathbb{M}^{op})\rightarrow\mathrm{Fun}(\mathbf{S},\mathbb{M}^{op})$ 
corestricts to equivalences
\begin{align}\label{equlanequiv}
s\mathbb{M} & \xrightarrow{\simeq}\mathrm{LAdj}(\mathbf{S},\mathbb{M}^{op})^{op}\xrightarrow{\cong}\mathrm{RAdj}(\mathbb{M}^{op},\mathbf{S}) \\
\notag X & \mapsto y_!(X^{op}) \mapsto\mathbb{M}(\phv,X_{\bullet})
\end{align}
into the opposite of the full subcategory of left adjoints in $\mathrm{Fun}(\mathbf{S}^{op},\mathbb{M})$, and the full subcategory of 
right adjoints in $\mathrm{Fun}(\mathbb{M}^{op},\mathbf{S})$, respectively. This is noted already in \cite[Proposition 3.1.5]{hovey}.

\begin{definition}\label{defextmc}
Restriction of the composition (\ref{equlanequiv}) to the full subcategory $\mathrm{Cat}_{\infty}(\mathbb{M})\subset s\mathbb{M}$ of 
complete Segal objects in $\mathbb{M}$ defines the model categorical externalization functor
\[\mathrm{Ext}\colon\mathrm{Cat}_{\infty}(\mathbb{M})\rightarrow\mathrm{Fun}(\mathbb{M}^{op},\mathbf{S}).\]
We say that a functor $\mathbb{M}^{op}\rightarrow\mathbf{S}$ is \emph{small} if it is naturally isomorphic
to the model categorical externalization of a complete Segal object in $\mathbb{M}$. 
\end{definition}

Against the background of Definition~\ref{defextmc} we will refer to the functor
$\mathbb{M}(\phv, X_{\bullet})\colon\mathbb{M}^{op}\rightarrow\mathbf{S}$ of a general simplicial object $X$ in $\mathbb{M}$ as the 
externalization $\mathrm{Ext}(X)$ of $X$ at times as well to simplify notation.

\begin{example}
Every complete Segal groupoid $X$ in $\mathbb{M}$ is a complete Segal object in $\mathbb{M}$. For such $X$, the functor
$\mathrm{Ext}(X)$ is a model for the derived mapping space functor $\mathbb{M}(\phv,X_0)^h\colon\mathbb{M}^{op}\rightarrow\mathbf{S}$ 
since $X$ is a simplicial frame on $X_0$.
\end{example}

\begin{example}\label{exple1ext}
If $\mathbb{C}$ is an internal category in (the underlying category of) $\mathbb{M}$, 
we can construct the internal nerve $N(\mathbb{C})\in s\mathbb{M}$ in $\mathbb{M}$, see e.g.\ \cite[Remark B.2.3.2]{elephant}. This will 
generally not be Reedy fibrant, let alone complete. If by $h$ we denote the left adjoint to the ($\mathrm{Set}$-internal) nerve functor
$N\colon\mathrm{Cat}\rightarrow\mathbf{S}$, then the push-forward
$h\circ\mathbb{M}(\phv,N(\mathbb{C})_{\bullet})\colon\mathbb{M}^{op}\rightarrow\mathbf{S}\rightarrow\mathrm{Cat}$ is exactly the 
externalization of the category object $\mathbb{C}$ in the 1-categorical sense \cite[Section 7.3]{jacobsttbook}. 
\end{example}

\begin{notation}
For any model category $\mathbb{N}$, let
$\mathrm{Qrf}(\mathbb{M}^{op},\mathbb{N})\subseteq\mathrm{Fun}(\mathbb{M}^{op},\mathbb{N})$ denote the full 
subcategory of right Quillen functors. 
\end{notation}

\begin{proposition}\label{propextrQf}
The functor $\mathrm{Ext}\colon\mathrm{Cat}_{\infty}(\mathbb{M})\rightarrow\mathrm{Fun}(\mathbb{M}^{op},\mathbf{S})$ restricts to equivalences 
\begin{enumerate}
\item $\mathrm{Ext}\colon\mathrm{Cat}_{\infty}(\mathbb{M})\xrightarrow{\simeq}\mathrm{Qrf}(\mathbb{M}^{op},(\mathbf{S},\mathrm{QCat}))$, and
\item $\mathrm{Ext}\colon\mathrm{Fr}(\mathbb{M})\xrightarrow{\simeq}\mathrm{Qrf}(\mathbb{M}^{op},(\mathbf{S},\mathrm{Kan}))$.
\end{enumerate}
\end{proposition}
\begin{proof}
Since the map (\ref{equlanequiv}) is an equivalence, for Part 1 we only have to show that a simplicial object $X$ in
$\mathbb{M}$ is a complete Segal object if and only if its associated functor
$\mathrm{Ext}(X)$ is a right Quillen functor with respect to the Joyal model structure. This follows directly from 
Lemma~\ref{lemmarQuillenchar}.
For Part 2 we are left to show that a simplicial object $X$ in $\mathbb{M}$ is a complete Segal groupoid if and only if its 
associated functor $\mathrm{Ext}(X)$ is a right Quillen functor with respect to the Quillen model structure for Kan complexes. But a 
Reedy fibrant $X$ is homotopically constant if and only if all boundary maps $y_!(X^{op})(d^i)\colon X_{n+1}\rightarrow X_n$ 
are acyclic fibrations (since the degeneracies are sections thereof). 
By an application of \cite[Lemma 3.7]{jtqcatvsss} this in turn holds if and only if the left adjoint $y_!(X^{op})$ 
takes all anodyne maps in $\mathbf{S}$ to acyclic cofibrations in $\mathbb{M}^{op}$. Thus, it follows that $X$ is a 
complete Segal groupoid if and only if the functor $y_!(X^{op})\colon(\mathbf{S},\mathrm{Kan})\rightarrow\mathbb{M}^{op}$ is 
a left Quillen functor.
\end{proof}

\begin{remark}
Proposition~\ref{propextrQf} states that the small (representable) $\mathbb{M}$-indexed simplicial sets are exactly the contravariant 
right Quillen functors from $\mathbb{M}$ into the Joyal (Kan) model structure on $\mathbf{S}$. This is a model categorical analogon to 
the more general $\infty$-categorical fact exploited in Proposition~\ref{propextasnerve}: Whenever $\mathcal{C}$ is a complete
$\infty$-category, the small $\mathcal{C}$-indexed $\infty$-categories (presheaves) are exactly the contravariant right adjoint functors 
from $\mathcal{C}$ to $\mathrm{Cat}_{\infty}$ (to $\mathcal{S}$).
\end{remark}

This model categorical externalization construction of complete Segal objects can be dualized to a nerve construction of (not 
necessarily reduced) interval objects in the sense of To\"{e}n \cite[Definition 3.4]{toeninftycats}. This comprises many examples 
throughout the literature.

\begin{example}\label{explenervecat}
When the category $\mathrm{Cat}$ of small categories is equipped with the canonical model structure, the ordinary nerve
\[N\colon\mathrm{Cat}\rightarrow(\mathbf{S},\mathrm{QCat})\]
is a right Quillen functor. Hence, it is a small $\mathrm{Cat}^{op}$-indexed simplicial set and as such it is given by the 
externalization of the complete Segal object $\Delta^{\bullet}$ in $\mathrm{Cat}^{op}$. Completeness of the Segal object
$\Delta^{\bullet}$ corresponds exactly to the fact that the endpoint inclusion into the freely walking isomorphism is an acyclic 
cofibration in $\mathrm{Cat}$. The simplicial object $\Delta^{\bullet}$ itself is the internal nerve of 
the internal category $\Delta_{\leq 1}$ in $\mathrm{Cat}^{op}$ given by the ``free co-composition'' on the cograph
$\xymatrix{[0]\ar@<1ex>[r]\ar@<-1ex>[r] & [1]\ar[l]}$ in $\mathrm{Cat}$. That is, the functor $[1]\rightarrow [1]\cup_{[0]}[1]$ which 
maps the non-degenerate edge to the edge freely added between the two outer endpoints. Thus, $N\cong\mathrm{Ext}(N(\Delta_{\leq 1}))$.

The internal category $\Delta_{\leq 1}$ in $\mathrm{Cat}^{op}$ can also be externalized in the ordinary categorical sense 
(Example~\ref{exple1ext}). The corresponding $\mathrm{Cat}^{op}$-indexed category
$\mathrm{Ext}(\Delta_{\leq 1})\colon\mathrm{Cat}\rightarrow\mathrm{Cat}$ is isomorphic to the identity on $\mathrm{Cat}$. Its 
Grothendieck construction $p_{\Delta_{\leq 1}}\colon\int\mathrm{Ext}(\Delta_{\leq 1})\twoheadrightarrow\mathrm{Cat}$ is 
given by pointed categories and oplax-pointed functors between them. The opfibration $p_{\Delta_{\leq 1}}$ is the universal opfibration 
in the sense that every opfibration with fibres in $\mathrm{Cat}$ is equivalent to the pullback of $p_{\Delta_{\leq 1}}$ 
along its associated 1-categorical indexing.
\end{example}

\begin{example}
Similarly, the homotopy-coherent nerve 
\[N_{\Delta}\colon(\mathbf{S}\text{-Cat},\mathrm{Bergner})\rightarrow(\mathbf{S},\text{QCat})\]
with its left adjoint $\mathfrak{C}$ yields a Quillen equivalence. In particular, $N_{\Delta}$ is the small
$(\mathbf{S}\text{-Cat}^{op})$-indexed simplicial set given by externalization of the complete Segal object
$\mathfrak{C}|_{\Delta^{op}}$ in $\mathbf{S}\text{-Cat}^{op}$. When considered as a cosimplicial object in $\mathbf{S}\text{-Cat}$, this 
is known to be the Reedy cofibrant replacement of the diagram $\Delta^{\bullet}\in\mathbf{S}\text{-Cat}^{\Delta}$. That means dually
$\mathfrak{C}|_{\Delta^{op}}$ is the Reedy fibrant replacement of $\Delta^{\bullet}$ in $s(\mathbf{S}\text{-Cat}^{op})$. As the co-Segal 
maps of $\Delta^{\bullet}\in\mathbf{S}\text{-Cat}^{\Delta}$ are still isomorphisms, the Segal object $\Delta^{\bullet}$ is the internal 
nerve of the (levelwise locally discrete) internal category $\Delta_{\leq 1}$ in $\mathbf{S}\text{-Cat}$. It follows that
$\mathfrak{C}\simeq\mathrm{Ext}(\mathbb{R}(N(\Delta_{\leq 1}))$ where $\mathbb{R}$ denotes a corresponding Reedy fibrant replacement 
functor.

Analogously one can show that Dwyer and Kan's $\bar{W}$-construction for simplicial groupoids \cite[Section 3]{dksgrpds} is 
equivalent to $\mathrm{Ext}(\mathbb{R}(N(I\Delta_{\leq 1}))$ for $I\Delta_{\leq 1}$ in $\mathbf{S}\text{-Gpd}^{op}$ given by the 
levelwise push-forward of $\Delta_{\leq 1}$ by the free groupoid functor $I(\cdot)\colon\mathrm{Cat}\rightarrow\mathrm{Gpd}$ left 
adjoint to the canonical inclusion.
\end{example}

\begin{example}\label{exmplenervetoen}
More generally, in \cite{toeninftycats} To\"en considers criteria on a model category $\mathbb{M}$ for the existence 
of an indexed simplicial space $(\mathbb{M}^{op})^{op}\rightarrow(s\mathbf{S},\mathrm{Cat}_{\infty})$ which is the right part of a 
Quillen equivalence with respect to Rezk's model structure for complete Segal spaces. Therefore, he gives a general construction of 
functors $S_X\colon\mathbb{M}\rightarrow s\mathbf{S}$ associated to cosimplicial objects $X$ in $\mathbb{M}$ which 
is part of a Quillen equivalence if and only if the pair $(\mathbb{M},X)$ is a ``theory of $(\infty,1)$-categories'' 
\cite[Definition 4.2]{toeninftycats}. One essential part 
of this definition is that $X$ is a complete Segal object in $\mathbb{M}^{op}$ with contractible base $X_0$, there 
called an \emph{interval} in $\mathbb{M}$ \cite[Definition 3.4]{toeninftycats}. Given a theory of $(\infty,1)$-categories
$(\mathbb{M},X)$, he constructs in the proof of \cite[Theorem 5.1]{toeninftycats} the Quillen equivalence
\[S_X\colon\mathbb{M}\rightarrow (s\mathbf{S},\mathrm{Cat}_{\infty})\]
exactly such that its postcomposition with the underlying-quasi-category functor
$U\colon(s\mathbf{S},\mathrm{Cat}_{\infty})\rightarrow(\mathbf{S},\text{QCat})$ yields the externalization of 
the complete Segal object $X$ in $\mathbb{M}^{op}$. Since $U$ is a Quillen equivalence itself, To\"en's theorem can be rephrased stating 
that a pair $(\mathbb{M},X)$ is a theory of $(\infty,1)$-categories if and only if the externalization
$\mathrm{Ext}(X^{op})\colon\mathbb{M}\rightarrow(\mathbf{S},\mathrm{QCat})$ is a Quillen equivalence.
\end{example}

Since every Quillen pair between model categories induces an adjunction on underlying $\infty$-categories
\cite[Theorem 2.1]{mazelgeequadj}, every complete Segal object $X$ in $\mathbb{M}$ induces a $\mathcal{C}$-indexed $\infty$-category
\begin{align}\label{equlemmaext2ext1}
\mathrm{Ho}_{\infty}(\mathrm{Ext}(X))\colon\mathcal{C}^{op}\rightarrow\mathrm{Cat}_{\infty}
\end{align}
via 	Proposition~\ref{propextrQf}. It also induces a Segal object
$\mathrm{Ho}_{\infty}(X)$ in $\mathcal{C}$ by postcomposition of $X\colon\Delta^{op}\rightarrow\mathbb{M}$ with the $\infty$-categorical 
localization functor $\mathbb{M}\rightarrow\mathcal{C}$.
This Segal object in $\mathcal{C}$ is complete and hence an internal $\infty$-category via \cite[Theorem 4.6, Remark 4.8]{rs_uc}\footnote{Theorem 4.6 in \cite{rs_uc} is stated under the assumption of right properness (as made explicit in Remark 4.8). This however 
is not necessary; Remark 4.8, and hence Theorem 4.6, apply to any model category.\label{footnrp}}. 

\begin{proposition}\label{lemmaext2ext}
For every complete Segal object $X$ in a model category $\mathbb{M}$ the functor
\[\mathrm{Ho}_{\infty}(\mathrm{Ext}(X))\colon\mathcal{C}^{op}\rightarrow\mathrm{Cat}_{\infty}\]
is naturally equivalent to the $\infty$-categorical externalization $\mathrm{Ext}(\mathrm{Ho}_{\infty}(X))$.
\end{proposition}
\begin{proof}
Every Quillen pair between model categories induces an adjunction on underlying $\infty$-categories. Thus, to show that the right 
adjoints $\mathrm{Ext}(\mathrm{Ho}_{\infty}(X))$ and $\mathrm{Ho}_{\infty}(\mathrm{Ext}(X))$ are naturally equivalent, it suffices to 
show that so are the left adjoints $y_!(\mathrm{Ho}_{\infty}(X)^{op})$ and $\mathrm{Ho}_{\infty}(y_!(X^{op}))$ which we consider as 
functors of type $\mathrm{Cat}_{\infty}\rightarrow\mathcal{C}^{op}$. The former left adjoint is the 
left Kan extension of $\mathrm{Ho}_{\infty}(X)^{op}\colon\Delta\rightarrow\mathcal{C}^{op}$ along the canonical embedding 
$y\colon\Delta\rightarrow\mathrm{Cat}_{\infty}$ \cite[Proposition 5.23]{rs_comp}. As both left adjoints preserve colimits and the 
embedding $y\colon\Delta\rightarrow\mathrm{Cat}_{\infty}$ generates $\mathrm{Cat}_{\infty}$ under colimits, it suffices to show that the 
two restrictions $y_!(\mathrm{Ho}_{\infty}(X^{op}))\circ y$ and $\mathrm{Ho}_{\infty}(y_!(X^{op}))\circ y$ of type
$\Delta\rightarrow\mathcal{C}^{op}$ are naturally equivalent. The former is naturally equivalent to $\mathrm{Ho}_{\infty}(X^{op})$ since 
$y$ is fully faithful \cite[Section 4.3.2]{luriehtt}. Regarding the latter, we have a commutative diagram of the form
\[\xymatrix{
\Delta\ar[d]_y\ar@/^/[drr]^{X^{op}} & \\
(\mathbf{S},\mathrm{QCat})\ar[rr]_(.6){y_!(X^{op})}\ar[d] & &\mathbb{M}^{op}\ar[d] \\
\mathrm{Cat}_{\infty}\ar[rr]_(.55){\mathrm{Ho}_{\infty}(y_!(X^{op}))}& & \mathcal{C}^{op}
}\]
simply by the definition of the two vertical functors. Here, the two unlabelled vertical arrows denote the respective
$\infty$-categorical localization functors. The left vertical composition is exactly the generating canonical embedding
$y\colon\Delta\rightarrow\mathrm{Cat}_{\infty}$. Thus, the outer square yields an equivalence
$\mathrm{Ho}_{\infty}(y_!(X^{op}))\circ y\simeq\mathrm{Ho}_{\infty}(X^{op})$ as well.
\end{proof}

It may be interesting to note that both the $\infty$-categorical externalization construction in Section~\ref{secformal} as well as the 
model categorical externalization construction in this section are canonically induced from the Yoneda embedding associated to the 
corresponding base structure. However, the Yoneda embedding associated to a plain model category $\mathbb{M}$ has no $\infty$-categorical 
content. In this sense, opposed to its $\infty$-categorical counterpart, the model categorical externalization construction is an
$\infty$-categorical primitive.

\begin{example}\label{explerqcore}
In Section~\ref{secsubformalunderqcat} we recalled the existence of a right Quillen functor
\[k^!\colon(\mathbf{S},\text{QCat})\rightarrow(\mathbf{S},\text{Kan})\]
which comes with a natural weak equivalence to the core functor $(\cdot)^{\simeq}\colon\mathrm{QCat}\rightarrow\mathrm{Kan}$ when 
restricted to the full subcategory of quasi-categories. Postcomposition of a small $\mathbb{M}$-indexed simplicial set
$\mathrm{Ext}(X)\colon\mathbb{M}^{op}\rightarrow(\mathbf{S},\mathrm{QCat})$ with $k^!$ yields a right Quillen functor
$k^!\circ\mathrm{Ext}(X)\colon\mathbb{M}^{op}\rightarrow(\mathbf{S},\mathrm{Kan})$.
Hence, by Proposition~\ref{propextrQf}, there is a complete Segal groupoid $X^{\simeq}$ in $\mathbb{M}$ such that
\[\text{Ext}(X^{\simeq})\cong k^!\circ\text{Ext}(X).\]
It follows from Proposition~\ref{lemmaext2ext} that the two complete Segal groupoids
$\mathrm{Ho}_{\infty}(X^{\simeq})$ and $\mathrm{Ho}_{\infty}(X)^{\simeq}$ from Definition~\ref{defcompleteness} are naturally 
equivalent. One may therefore refer to $X^{\simeq}$ as the core of the complete Segal object $X$ in $\mathbb{M}$. It is explicitly 
constructed in \cite[Lemma 5.6]{rs_bspaces} for $\mathbb{M}=(\mathbf{S},\mathrm{Kan})$.
\end{example}

\begin{remark}\label{remmsforcs}
Dugger showed in \cite{duggersimp} that whenever $\mathbb{M}$ is left proper and combinatorial there is a model structure
$(s\mathbb{M},\mathrm{Gpd}_{\infty})$ whose fibrant objects are exactly the complete Segal groupoids, and such that the inclusion
$\Delta\colon\mathbb{M}\rightarrow(s\mathbb{M},\mathrm{Gpd}_{\infty})$ is the left part of a Quillen equivalence. In particular, the 
composition
\[\mathbb{M}\xrightarrow{\Delta}s\mathbb{M}\xrightarrow{\simeq}\mathrm{RAdj}(\mathbb{M}^{op},\mathbf{S})\] 
with the equivalence (\ref{equlanequiv}) yields a Quillen equivalence from $\mathbb{M}$ to a model structure on the functor category
$\mathrm{RAdj}(\mathbb{M}^{op},\mathbf{S})$ whose fibrant objects are exactly the right Quillen functors into
$(\mathbf{S},\mathrm{Kan})$ by Proposition~\ref{propextrQf}.

The same observation under the same assumptions on $\mathbb{M}$ induces a model structure for right Quillen functors into the Joyal 
model structure $(\mathbf{S},\mathrm{QCat})$ via the model structure for complete Segal objects on $s\mathbb{M}$ constructed for example 
in \cite[Proposition 2.2.9]{rvyoneda}.
\end{remark}

\subsection{The \texorpdfstring{$\infty$}{infinity}-cosmos of complete Segal objects in a model category}\label{secsubcosmosofcsos}

While the existence of a model structure on $s\mathbb{M}$ for complete Segal groupoids and a model structure on $s\mathbb{M}$ for 
complete Segal objects requires additional assumptions on $\mathbb{M}$, both notions always come equipped with a fibrational structure 
automatically. In the following, the fibration category $s\mathbb{M}^f$ will denote the category of Reedy fibrant objects in 
$s\mathbb{M}$. 

\begin{proposition}\label{propfibcatstr}
The full subcategories $\mathrm{Cat}_{\infty}(\mathbb{M})$ and $\mathrm{Gpd}_{\infty}(\mathbb{M})$ of $s\mathbb{M}^f$ are closed under 
small products, pullbacks along Reedy fibrations, transfinite towers of Reedy fibrations, and Reedy cofibrant replacements. They both 
are replete with respect to the class of weak equivalences in $s\mathbb{M}^f$. In particular, they both inherit the fibration category 
structure (with cofibrant replacements) of $s\mathbb{M}^f$ such that the inclusions
$\mathrm{Gpd}_{\infty}(\mathbb{M})\hookrightarrow\mathrm{Cat}_{\infty}(\mathbb{M})\hookrightarrow s\mathbb{M}^f$ are exact.
\end{proposition}
\begin{proof}
We formulate the proof for $\mathrm{Cat}_{\infty}(\mathbb{M})$; it is completely analogous (in fact even more straight-forward) for
$\mathrm{Gpd}_{\infty}(\mathbb{M})$. The fact that $\mathrm{Cat}_{\infty}(\mathbb{M})\subset s\mathbb{M}$ is closed under small products 
follows directly from the fact that the class of trivial fibrations in $\mathbb{M}$ is closed under small products (as well as the fact 
that the Segal map of a product of simplicial objects is the product of corresponding Segal maps). The proofs regarding pullbacks along  
fibrations and transfinite towers of fibrations are similarly straight-forward.

For the fibration category structure on $\mathrm{Cat}_{\infty}(\mathbb{M})$ we define a morphism $f\colon X\rightarrow Y$ between 
complete Segal objects $X$ and $Y$ in $\mathbb{M}$ to be a fibration (weak equivalence) if it is a fibration (weak equivalence) in the 
fibration category $s\mathbb{M}^f$. To verify that this defines the structure of a fibration category, one essentially is only left to 
verify that every morphism $f\colon X\rightarrow Y$ in $\mathrm{Cat}_{\infty}(\mathbb{M})$ factors into a weak equivalence followed by a 
fibration. Thus, given a morphism $f\colon X\rightarrow Y$ in $\mathrm{Cat}_{\infty}(\mathbb{M})$, let $j\colon X\rightarrow Z$ be a 
weak equivalence and $p\colon Z\twoheadrightarrow Y$ be a fibration in $s\mathbb{M}^f$ such that $pj=f$. Then the pair $(j,p)$ is a 
factorization in $\mathrm{Cat}_{\infty}(\mathbb{M})$ as desired if $Z$ is again contained in $\mathrm{Cat}_{\infty}(\mathbb{M})$. The 
simplicial object $Z$ is Reedy fibrant by assumption. As $j\colon X\rightarrow Z$ is a (pointwise) weak equivalence, $Z$ is again a 
complete Segal object in $\mathbb{M}$. Indeed, validation of the Segal conditions is immediate. Validation of completeness follows from
\cite[Lemma 4.4.2]{rs_uc}\footnote{As referred to in Footnote~\ref{footnrp}, the lemma does in fact not make use of the ambient 
assumption of right properness of $\mathbb{M}$ contrary to what is stated there.}. Exactness of the inclusion
$\mathrm{Cat}_{\infty}(\mathbb{M})\subset s\mathbb{M}$ follows trivially.
Closure under cofibrant replacements follows in the same way.
\end{proof}


We recall from \cite[Section 4.1]{duggersimp} that the category $s\mathbb{M}$ is always simplicially enriched and that it is furthermore 
both tensored and cotensored over $\mathbf{S}$. In summary, for $K\in\mathbf{S}$ and $X\in s\mathbb{M}$, define
$K\otimes X\in s\mathbb{M}$ via
\[(K\otimes X)_n:=\coprod_{K_n}X_n,\]
and
\begin{align}\label{defmccotensor}
X^K:=y_!(X^{op})(K\times\Delta^{\bullet}).
\end{align}
The latter is Dugger's original formula up to an explicit pointwise description of the left Kan extension $y_!(X^{op})$. For 
$X$ and $Y$ in $s\mathbb{M}$ this induces the definition of a mapping object
\[\mathrm{Map}_{s\mathbb{M}}(X,Y)=s\mathbb{M}(\Delta^{\bullet}\otimes X,Y)\cong s\mathbb{M}(X,Y^{\Delta^{\bullet}}).\]
We recall that these mapping objects are generally not Kan complexes for Reedy bifibrant simplicial objects $X$ and $Y$. However, they
do induce a canonical 
$(\mathbf{S},\mathrm{QCat})$-enrichment of the fibration category $\mathrm{Cat}_{\infty}(\mathbb{M})$ and a canonical 
$(\mathbf{S},\mathrm{Kan})$-enrichment of the fibration category $\mathrm{Gpd}_{\infty}(\mathbb{M})$ instead.

\begin{proposition}\label{propfibcatstrextra}
The full simplicially enriched subcategories $\mathbf{Cat}_{\infty}(\mathbb{M})$ and $\mathbf{Gpd}_{\infty}(\mathbb{M})$ in
$s\mathbb{M}$ are cotensored over $\mathbf{S}$ as well.
\end{proposition}
\begin{proof}
We again show the case for $\mathbf{Cat}_{\infty}(\mathbb{M})$ only. The case for $\mathbf{Gpd}_{\infty}(\mathbb{M})$ is analogous.
We are to show that for every $X\in\mathrm{Cat}_{\infty}(\mathbb{M})$ and every simplicial set $K\in\mathbf{S}$, the cotensor
$X^K\in s\mathbb{M}$ 
is again a complete Segal object. By Proposition~\ref{propextrQf} we therefore have to show that
\[y_! ((X^K)^{op})\colon(\mathbf{S},\mathrm{QCat})\rightarrow\mathbb{M}^{op}\]
is a left Quillen functor. But $y_! ((X^K)^{op})\cong y_!(X^{op})(K\times(\cdot))$ given that both functors are 
cocontinuous and restrict to the same functor on $\Delta$. The latter is the composition of the left Quillen endofunctor
$K\times(\cdot)$ on $(\mathbf{S},\mathrm{QCat})$ with the left Quillen functor $y_!(X^{op})$. As such it is a left Quillen 
functor itself.
\end{proof}

\begin{remark}\label{remmappingqcatalt}
In analogy to the formula of Remark~\ref{rem2catstrdirect}, it follows that the mapping objects of $\mathbf{Cat}_{\infty}(\mathbb{M})$ 
can be computed directly in $\mathrm{Cat}_{\infty}(\mathbb{M})$ by
\[\mathbf{Cat}_{\infty}(\mathbb{M})(X,Y)=\mathrm{Cat}_{\infty}(\mathbb{M})(X,Y^{\Delta^{\bullet}}).\]
\end{remark}

\begin{corollary}\label{corfibcatenr}
Let $\mathbb{M}$ be a model category.
\begin{enumerate}
\item The fibration category $\mathbf{Cat}_{\infty}(\mathbb{M})$ is $(\mathbf{S},\mathrm{QCat})^c$-enriched.
\item The fibration category $\mathbf{Gpd}_{\infty}(\mathbb{M})$ is $(\mathbf{S},\mathrm{Kan})^c$-enriched.
\end{enumerate}
\end{corollary}
\begin{proof}
We again consider $\mathbf{Cat}_{\infty}(\mathbb{M})$ only. By \cite[Proposition 3.2 and Proposition 4.4]{duggersimp}
\footnote{Proposition 3.2 in \cite{duggersimp} is stated under the assumption of left properness of $\mathbb{M}$. This however is used 
only to reduce the left lifting property against all cofibrations to the left lifting property against all cofibrations between 
cofibrant objects. This in fact is valid in any model category as shown in \cite[Corollary 7.13]{jtqcatvsss}.} we are only left to show 
that for every $X\in\mathrm{Cat}_{\infty}(\mathbb{M})$ and every acyclic cofibration $j\colon A\rightarrow B$ in
$(\mathbf{S},\mathrm{QCat})$ the fibration $X^j\colon X^B\rightarrow X^A$ is trivial in $s\mathbb{M}$. But $X^j\cong y_!(X^{op})(j)$ is 
a trivial fibration because $y_!(X^{op})\colon(\mathbf{S},\mathrm{QCat})\rightarrow\mathbb{M}^{op}$ is a left Quillen functor by 
Proposition~\ref{propextrQf}.
\end{proof}

\begin{corollary}\label{corinftycosmos}
Let $\mathbb{M}$ be a model category. Then the Reedy model structure on $s\mathbb{M}$ equips both simplicially enriched subcategories
$\mathbf{Cat}_{\infty}(\mathbb{M})$ and $\mathbf{Gpd}_{\infty}(\mathbb{M})$ with the structure of an $\infty$-cosmos in the weaker sense 
of \cite{rvyoneda}. Nevertheless, they both have all cosmological limits from \cite[Definition 1.2.1 (i)]{riehlverityelements}, that is, 
all simplicial cotensors (instead finitely presented ones only), small products, pullbacks of fibrations and countable towers 
of fibrations.\qed
\end{corollary}

\begin{remark}\label{remmscofgencase}
Corollary~\ref{corinftycosmos} can be understood as a decompression of the main result in \cite{duggersimp} into two parts. 
First, the fibration category $\mathrm{Gpd}_{\infty}(\mathbb{M})$ comes equipped with an enrichment over
$(\mathbf{S},\mathrm{Kan})$ for every model category $\mathbb{M}$. The evaluation functor
$\mathrm{ev}_0\colon\mathrm{Gpd}_{\infty}(\mathbb{M})\rightarrow\mathbb{M}^f$ is always exact, and in fact it is easy to show that it is 
a weak equivalence of fibration categories in the sense of \cite[Definition 1.7]{szumilococomplqcats}. Thus, every model category can be 
replaced by a simplicially enriched fibration category. And second, this fibration category underlies an (automatically simplicially 
enriched) model structure on $s\mathbb{M}$ whenever the model category $\mathbb{M}$ is furthermore combinatorial and left proper. In 
this case, the weak equivalence $\mathrm{ev}_0\colon\mathrm{Gpd}_{\infty}(\mathbb{M})\rightarrow\mathbb{M}^f$ is the underlying exact 
functor of a Quillen equivalence.

Under the same additional assumptions on $\mathbb{M}$, there is an intermediate (combinatorial and 
left proper) model structure $\mathrm{Cat}_{\infty}$ on $s\mathbb{M}$ obtained by left Bousfield localization of the Reedy model 
structure on $s\mathbb{M}$ as well, such that $(s\mathbb{M},\mathrm{Cat}_{\infty})^f=\mathrm{Cat}_{\infty}(\mathbb{M})$, and such that
$(s\mathbb{M},\mathrm{Cat}_{\infty})$ is a $(\mathbf{S},\mathrm{QCat})$-enriched model category. This is
\cite[Proposition 2.2.9]{rvyoneda}.
If $\mathbb{M}$ even is a Cisinski model category (to be recalled in Remark~\ref{remcisinskimc}), then all simplicial objects in
$\mathbb{M}$ are Reedy cofibrant. It follows that $\mathbf{Cat}_{\infty}(\mathbb{M})$ is an $\infty$-cosmos (of cofibrant objects) as 
defined in \cite{riehlverityelements}, see \cite[Proposition E.3.7]{riehlverityelements}.
\end{remark}

\section{The right derived externalization functor}\label{subsecrdevext}

In this section we prove various exactness properties of the model categorical externalization functor of
Definition~\ref{defextmc}, and show that it recovers the $\infty$-categorical externalization functor from Section~\ref{secformal} 
whenever the model category $\mathbb{M}$ satisfies some suitable additional properties. We give an application in
Proposition~\ref{prop1oc2topos} which shows that $\mathrm{Cat}_{\infty}(\mathcal{C})$ is an $(\infty,1)$-localic $(\infty,2)$-topos as 
defined in Definition~\ref{def1loc2topos} whenever $\mathcal{C}$ is an $\infty$-topos.

\begin{notation}\label{notationrderextV}
Following up on Notation~\ref{notation_univ}, we will assume the existence of a Grothendieck universe $V$ of small categories as well as
a Grothendieck universe $V^+$ of large categories, the latter of which
contains the model category $\mathbb{M}$.
In particular, $\mathbb{M}$ has all small limits, and all results of 
Sections~\ref{secsubcsointro} and \ref{secsubcosmosofcsos} apply both in the context of the cartesian closed model category
$\mathbf{S}$ of small simplicial sets, as well as in the context of the cartesian closed superlarge model category $\mathbf{S}^+$ of large 
simplicial sets. In particular, the fibration category $\mathbf{Cat}_{\infty}(\mathbb{M})$ remains $(\mathbf{S},\mathrm{QCat})^c$-enriched, 
and $\mathbf{Gpd}_{\infty}(\mathbb{M})$ remains $(\mathbf{S},\mathrm{Kan})^c$-enriched. 
Technically, the two Grothendieck universes $V$ and $V^+$ are given by the class of $\nu$-small and $\nu^+$-small sets for suitable 
inaccessible cardinals 
$\nu<\nu^+$. Without loss of generality the cardinals $\nu<\nu^+$ can be chosen in such a way that the underlying $\infty$-category 
of the model categories $(\mathbf{S},\mathrm{QCat})$ and $(\mathbf{S}^+,\mathrm{QCat})$ of $\nu$-small and $\nu^+$-small simplicial sets 
are the $\infty$-categories $\mathrm{Cat}_{\infty}$ and $\mathrm{CAT}_{\infty}$ of $\nu$-small and $\nu^+$-small $\infty$-categories, 
respectively \cite[Corollary 3.16]{rs_small}. For $\mathrm{CAT}_{\infty}$ to be a quasi-category, we will implicitly assume the 
existence of a third even larger Grothendieck universe as we did in Notation~\ref{notation_univ}, but this will not require further 
explicit reference. We denote the composition 
\[\mathrm{Cat}_{\infty}(\mathbb{M})\xrightarrow{\mathrm{Ext}}\mathrm{Fun}(\mathbb{M}^{op},\mathbf{S})\xrightarrow{\iota_{\ast}}\mathrm{Fun}(\mathbb{M}^{op},\mathbf{S}^+)\]
with the push-forward along the canonical inclusion $\iota\colon\mathbf{S}\hookrightarrow\mathbf{S}^+$ by $\mathrm{Ext}$ as well. As usual, 
a functor will be said to be continuous if it preserves all small limits.
\end{notation}

Consider the category $\mathbb{M}$ as a discrete simplicial category and let 
$\mathbf{Fun}(\mathbb{M}^{op},\mathbf{S}^+)$ be the simplicially enriched category of simplicial presheaves. Its underlying category
$\mathbf{Fun}(\mathbb{M}^{op},\mathbf{S}^+)_0$ is exactly $\mathrm{Fun}(\mathbb{M}^{op},\mathbf{S}^+)$. Let
$\mathrm{Fun}(\mathbb{M}^{op},(\mathbf{S}^+,\mathrm{QCat}))_{\mathrm{proj}}$ denote the category
$\mathrm{Fun}(\mathbb{M}^{op},\mathbf{S}^+)$ equipped with the corresponding
projective model structure. Then $\mathbf{Fun}(\mathbb{M}^{op},(\mathbf{S}^+,\mathrm{QCat}))_{\mathrm{proj}}$
is a $(\mathbf{S}^+,\mathrm{QCat})$-enriched model category \cite[Proposition A.3.3.2]{luriehtt}. We denote its underlying
$(\mathbf{S}^+,\mathrm{QCat})^c$-enriched fibration category by
\[\mathbf{Fun}(\mathbb{M}^{op},\mathbf{QCAT})_{\mathrm{proj}}:=\mathbf{Fun}(\mathbb{M}^{op},(\mathbf{S}^+,\mathrm{QCat}))_{\mathrm{proj}}^f.\]
Its full simplicial subcategory spanned by the pointwise small quasi-categories is a $(\mathbf{S},\mathrm{QCat})^c$-enriched fibration 
category and will be denoted by $\mathbf{Fun}(\mathbb{M}^{op},\mathbf{QCat})_{\mathrm{proj}}$. 

Let $\lambda\colon\mathbb{M}\rightarrow\mathbb{M}$ be a functorial cofibrant replacement functor. By
\cite[Proposition A.3.3.7, Example A.3.2.23]{luriehtt} restriction along $\lambda$ induces an $(\mathbf{S}^+,\mathrm{QCat})$-enriched right 
Quillen functor 
\[\lambda^{\ast}\colon\mathbf{Fun}(\mathbb{M}^{op},(\mathbf{S}^+,\mathrm{QCat}))_{\mathrm{proj}}\rightarrow\mathbf{Fun}(\mathbb{M}^{op},(\mathbf{S}^+,\mathrm{QCat}))_{\mathrm{proj}}.\]
In particular, $\lambda^{\ast}$ preserves projectively fibrant presheaves and hence trivially restricts to an
$(\mathbf{S}^+,\mathrm{QCat})$-enriched functor
\[\lambda^{\ast}\colon\mathbf{Fun}(\mathbb{M}^{op},\mathbf{QCat})_{\mathrm{proj}}\rightarrow\mathbf{Fun}(\mathbb{M}^{op},\mathbf{QCat})_{\mathrm{proj}}.\]
The same constructions and conventions apply to $(\mathbf{S}^+,\mathrm{Kan})$ replacing $(\mathbf{S}^+,\mathrm{QCat})$.
We recall the model categorical externalization functor
$\mathrm{Ext}\colon\mathrm{Cat}_{\infty}(\mathbb{M})\rightarrow\mathrm{Fun}(\mathbb{M}^{op},\mathbf{S})$ from Definition~\ref{defextmc}.

\begin{proposition}\label{propextcosm1}
The pointwise right derived externalization $\lambda^{\ast}\circ\mathrm{Ext}$ gives rise to transfinitely $\mathbf{S}$-exact 
functors (Definition~\ref{defCexact})
\begin{align}\label{equpropextcosm1}
\mathbb{R}\mathbf{Ext}\colon\mathbf{Cat}_{\infty}(\mathbb{M})\rightarrow\mathbf{Fun}(\mathbb{M}^{op},\mathbf{QCat})_{\mathrm{proj}}
\end{align}
and
\[\mathbb{R}\mathbf{Ext}\colon\mathbf{Gpd}_{\infty}(\mathbb{M})\rightarrow\mathbf{Fun}(\mathbb{M}^{op},\mathbf{Kan})_{\mathrm{proj}}.\]
\end{proposition}
\begin{proof}
We show the statement for $\mathbf{Cat}_{\infty}(\mathbb{M})$; the case $\mathbf{Gpd}_{\infty}(\mathbb{M})$ is completely 
analogous.

First, let us show that $\mathrm{Ext}\colon\mathrm{Cat}_{\infty}(\mathbb{M})\rightarrow\mathrm{Fun}(\mathbb{M}^{op},\mathbf{S})$ 
preserves small simplicial cotensors. Therefore, let $X\in\mathrm{Cat}_{\infty}(\mathbb{M})$ 
and $S\in\mathbf{S}$. In the proof of Proposition~\ref{propfibcatstrextra} we noted that the left adjoint
$y_!((X^S)^{op})\colon\mathbf{S}\rightarrow\mathbb{M}^{op}$ is naturally isomorphic to the composition
$y_!(X^{op})(S\times(\cdot))$ of left adjoints. It follows that the right adjoint $\mathrm{Ext}(X^S)$ is naturally isomorphic 
to the composition $\mathrm{Ext}(X)^S$ of respective right adjoints. This however computes the corresponding cotensor of $\mathrm{Ext}(X)$ 
in $\mathbf{Fun}(\mathbb{M}^{op},\mathbf{S})$.
In particular, for any two $X,Y\in\mathrm{Cat}_{\infty}(\mathbb{M})$ we obtain a sequence
\begin{align}\label{equpropextcosm2}
\notag \mathbf{Cat}_{\infty}(\mathbb{M})(X,Y) & \cong \mathrm{Cat}_{\infty}(\mathbb{M})(X,Y^{\Delta^{\bullet}}) \\
\notag & \cong\mathrm{Fun}(\mathbb{M}^{op},\mathbf{S})(\mathrm{Ext}(X),\mathrm{Ext}(Y^{\Delta^{\bullet}}))\\ 
\notag  & \cong\mathrm{Fun}(\mathbb{M}^{op},\mathbf{S})(\mathrm{Ext}(X),\mathrm{Ext}(Y)^{\Delta^{\bullet}}) \\
\notag  & \cong\mathbf{Fun}(\mathbb{M}^{op},\mathbf{S})(\mathrm{Ext}(X),\mathrm{Ext}(Y))
\end{align}
of natural isomorphisms of simplicial sets. Here the first isomorphism is given in Remark~\ref{remmappingqcatalt}, and the second one is 
given by the natural action of $\mathrm{Ext}$ on morphisms which is levelwise an isomorphism (e.g.\ by the equivalence 
(\ref{equlanequiv})). The functor
$\mathrm{Ext}\colon\mathrm{Cat}_{\infty}(\mathbb{M})\rightarrow\mathrm{Fun}(\mathbb{M}^{op},\mathbf{S})$ furthermore preserves all 
ordinary categorical limits that exist in $\mathrm{Cat}_{\infty}(\mathbb{M})$, given that it is is naturally isomorphic to the 
composition $\mathrm{Cat}_{\infty}(\mathbb{M})\hookrightarrow s\mathbb{M}\xrightarrow{sy} s\mathrm{Fun}(\mathbb{M}^{op},\mathrm{Set})$ 
of continuous functors.
It follows that $\mathrm{Ext}$ gives rise to a simplicially enriched functor $\mathbf{Ext}$ which preserves all conical limits as well 
as all small simplicial cotensors. Hence, so does the composition $\mathbb{R}\mathbf{Ext}=\lambda^{\ast}\circ\mathbf{Ext}$ in 
(\ref{equpropextcosm1}).

We are left to show that $\mathbb{R}\mathbf{Ext}$ preserves fibrations and trivial fibrations.
Therefore let $p\colon X\twoheadrightarrow Y$ be a (trivial) fibration in $\mathrm{Cat}_{\infty}(\mathbb{M})$. To show that
$\mathbb{R}\mathrm{Ext}(p)\colon\mathbb{R}\mathrm{Ext}(X)\rightarrow\mathbb{R}\mathrm{Ext}(Y)$ is a 
projective (trivial) fibration in $\mathrm{Fun}(\mathbb{M}^{op},(\mathbf{S},\mathrm{QCat}))_{\mathrm{proj}}$, it suffices to show that 
for all cofibrant objects $M\in\mathbb{M}$, the map $\mathrm{Ext}(p)(M)\colon\mathrm{Ext}(X)(M)\rightarrow\mathrm{Ext}(Y)(M)$ is a 
(trivial) fibration in $(\mathbf{S},\mathrm{QCat})$. Thus, let $M\in\mathbb{M}$ be cofibrant and let $j\colon A\hookrightarrow B$ be a 
(trivial) cofibration in $(\mathbf{S},\mathrm{QCat})$. We are to show that the gap map
\[(\mathrm{Ext}(p)(M)^j)_0\colon(\mathrm{Ext}(X)(M)^B)_0\rightarrow(\mathrm{Ext}(X)(M)^A)_0\times_{(\mathrm{Ext}(Y)(M)^A)_0}(\mathrm{Ext}(Y)(M)^B)_0\]
of sets has a section. By the above, this map is isomorphic to the map
\[\mathrm{Ext}(p^j)(M)_0\colon\mathbb{M}(M,(X^B)_0)\rightarrow\mathbb{M}(M,(X^A\times_{Y^A}Y^B)_0).\]
The morphism $p^j\colon X^B\rightarrow X^A\times_{Y^A}Y^B$ is a trivial fibration in $\mathrm{Cat}_{\infty}(\mathbb{M})$ by 
Corollary~\ref{corfibcatenr}. Hence, so is $(p^j)_0\colon (X^B)_0\rightarrow(X^A\times_{Y^A}Y^B)_0$ in $\mathbb{M}$. Given that $M$ 
is cofibrant in $\mathbb{M}$, we obtain the desired section.
\end{proof}

\begin{remark}\label{remcofreplchoice}
When all objects in $\mathbb{M}$ are cofibrant, the cofibrant replacement functor $\lambda\colon\mathbb{M}\rightarrow\mathbb{M}$ can 
without loss of generality be chosen to be the identity. In that case the right derived externalization functor $\mathbb{R}\mathbf{Ext}$ 
is just the externalization
$\mathbf{Ext}\colon\mathbf{Cat}_{\infty}(\mathbb{M})\rightarrow\mathbf{Fun}(\mathbb{M}^{op},\mathbf{QCat})_{\mathrm{proj}}$ itself (and 
hence is an embedding of simplicial categories).
\end{remark}

Let $W\subseteq\mathbb{M}^{\Delta^1}$ denote the class of weak equivalences in $\mathbb{M}$. As the
$(\mathbf{S}^+,\mathrm{QCat})$-enriched model category $\mathbf{Fun}(\mathbb{M}^{op},(\mathbf{S}^+,\mathrm{QCat}))_{\mathrm{proj}}$ is left
proper and combinatorial (in the Grothendieck universe of superlage categories), we may consider its 
$(\mathbf{S}^+,\mathrm{QCat})$-enriched left Bousfield localization
\[\mathcal{L}_{y[W]}\mathbf{Fun}(\mathbb{M}^{op},(\mathbf{S}^+,\mathrm{QCat}))_{\mathrm{proj}},\]
as well a its $(\mathbf{S}^+,\mathrm{Kan})$-enriched left Bousfield localization
$\mathcal{L}_{y[W]}\mathbf{Fun}(\mathbb{M}^{op},(\mathbf{S}^+,\mathrm{Kan}))_{\mathrm{proj}}$.
By the (simplicially enriched) Yoneda lemma, a projectively fibrant simplicial presheaf
$F\colon\mathbb{M}^{op}\rightarrow\mathbf{S}^+$ is $y[W]$-local if and only if it takes weak equivalences in $\mathbb{M}$ to equivalences 
of quasi-categories (Kan complexes). 
The simplicial full subcategory of
\[\mathcal{L}_{y[W]}\mathbf{Fun}(\mathbb{M}^{op},\mathbf{QCAT})_{\mathrm{proj}}:=\mathcal{L}_{y[W]}\mathbf{Fun}(\mathbb{M}^{op},(\mathbf{S}^+,\mathrm{QCat}))_{\mathrm{proj}}^f\]
spanned by the pointwise small quasi-categories is again a $(\mathbf{S},\mathrm{QCat})^c$-enriched fibration 
category which we denote by $\mathcal{L}_{y[W]}\mathbf{Fun}(\mathbb{M}^{op},\mathbf{QCat})_{\mathrm{proj}}$. The same applies to
$(\mathbf{S}^+,\mathrm{Kan})$.

\begin{proposition}\label{propextcosm2}
The pointwise right derived externalization functor factors through transfinitely $\mathbf{S}$-exact functors
\begin{align}\label{equpropextcosm21}
\mathbb{R}\mathbf{Ext}\colon\mathbf{Cat}_{\infty}(\mathbb{M})\rightarrow\mathcal{L}_{y[W]}\mathbf{Fun}(\mathbb{M}^{op},\mathbf{QCat})_{\mathrm{proj}}
\end{align}
and
\begin{align}\label{equpropextcosm212}
\mathbb{R}\mathbf{Ext}\colon\mathbf{Gpd}_{\infty}(\mathbb{M})\rightarrow\mathcal{L}_{y[W]}\mathbf{Fun}(\mathbb{M}^{op},\mathbf{Kan})_{\mathrm{proj}}.
\end{align}
\end{proposition}
\begin{proof}
The identity
\[\mathrm{id}\colon\mathcal{L}_{y[W]}\mathbf{Fun}(\mathbb{M}^{op},(\mathbf{S}^+,\mathrm{QCat}))_{\mathrm{proj}}\rightarrow\mathbf{Fun}(\mathbb{M}^{op},(\mathbf{S}^+,\mathrm{QCat}))_{\mathrm{proj}}\]
induces an inclusion 
\begin{align}\label{equpropextcosm22}
\mathbb{R}\mathrm{id}\colon\mathcal{L}_{y[W]}\mathbf{Fun}(\mathbb{M}^{op},\mathbf{QCAT})_{\mathrm{proj}}\rightarrow\mathbf{Fun}(\mathbb{M}^{op},\mathbf{QCAT})_{\mathrm{proj}}
\end{align}
of a full simplicially enriched subcategory whose $(\mathbf{S}^+,\mathrm{QCat})$-enriched fibration category structure is induced from 
that of $\mathbf{Fun}(\mathbb{M}^{op},\mathbf{QCAT})_{\mathrm{proj}}$. This inclusion hence reflects all corresponding 
limits, all cotensors, and all fibrations and trivial fibrations. By Proposition~\ref{propextcosm1}, we thus only have to show that
$\mathbb{R}\mathbf{Ext}\colon\mathbf{Cat}_{\infty}(\mathbb{M})\rightarrow\mathbf{Fun}(\mathbb{M}^{op},\mathbf{QCat})_{\mathrm{proj}}$ 
factors pointwise trough the inclusion (\ref{equpropextcosm22}).
But for any $X\in\mathrm{Cat}_{\infty}(\mathbb{M})$, the functor
$\mathbb{R}\mathrm{Ext}(X)\colon\mathbb{M}^{op}\rightarrow(\mathbf{S},\mathrm{QCat})$ is the right derived functor of the Quillen right 
adjoint $\mathrm{Ext}(X)$. It hence is projectively fibrant, and takes weak equivalences in $\mathbb{M}^{op}$ to weak equivalences in
$(\mathbf{S},\mathrm{QCat})$. Thus, it is $y[W]$-local.
\end{proof}

For the following, we recall that to every relative category $(\mathbb{C},W)$ one may associate its simplicial localization
$\mathcal{L}_{\Delta}(\mathbb{C},W)$ \cite{dksimploc}. It computes the $\infty$-categorical localization of $\mathbb{C}$ at $W$. 
Furthermore, we recall that weak equivalences in the Bergner model structure on the category of simplicial categories are often referred to 
as DK-equivalences.

\begin{corollary}\label{corcosmemblin}
For every model category $\mathbb{M}$ and every simplicial category $\mathbf{M}$ which is DK-equivalent to the simplicial 
localization $\mathcal{L}_{\Delta}(\mathbb{M},W)$ (of the underlying category of $\mathbb{M}$ at $W$) the pointwise right derived
externalization induces transfinitely $\mathbf{S}$-exact functors
\[\mathbb{R}\mathbf{Ext}\colon\mathbf{Cat}_{\infty}(\mathbb{M})\rightarrow\mathbf{Fun}(\mathbf{M}^{op},\mathbf{QCat})_{\mathrm{proj}}\]
and
\[\mathbb{R}\mathbf{Ext}\colon\mathbf{Gpd}_{\infty}(\mathbb{M})\rightarrow\mathbf{Fun}(\mathbf{M}^{op},\mathbf{Kan})_{\mathrm{proj}}.\]
\end{corollary}
\begin{proof}
We again only treat the case for $\mathbf{Cat}_{\infty}(\mathbb{M})$; the $\infty$-groupoidal case is completely analogous.
By Proposition~\ref{propextcosm2}, externalization gives rise to a transfinitely $\mathbf{S}$-exact functor
\begin{align}\label{equcorcosmemblin1}
\mathbb{R}\mathbf{Ext}\colon\mathbf{Cat}_{\infty}(\mathbb{M})\rightarrow\mathcal{L}_{y[W]}\mathbf{Fun}(\mathbb{M}^{op},\mathbf{QCat})_{\mathrm{proj}}.
\end{align}
Via Dwyer and Kan's work on simplicial localizations of homotopical categories \cite{dksimploc, dkequivs}, as well as Lurie's work on 
enriched model categories \cite[Section A.3.3]{luriehtt},
there is a zig-zag of simplicially enriched right Quillen equivalences between the
$(\mathbf{S}^+,\mathrm{QCat})$-enriched model categories
$\mathcal{L}_{y[W]}\mathbf{Fun}(\mathbb{M}^{op},(\mathbf{S}^+,\mathrm{QCat}))_{\mathrm{proj}}$ and  
$\mathbf{Fun}(\mathcal{L}_{\Delta}(\mathbb{M},W)^{op},(\mathbf{S}^+,\mathrm{QCat}))_{\mathrm{proj}}$.
By assumption, there is a zig-zag of DK-equivalences between $\mathcal{L}_{\Delta}(\mathbb{M},W)$ and $\mathbf{M}$, and so we 
furthermore obtain a zig-zag of simplicially enriched right Quillen equivalences between the $(\mathbf{S}^+,\mathrm{QCat})$-enriched model 
categories $\mathbf{Fun}(\mathcal{L}_{\Delta}(\mathbb{M},W)^{op},(\mathbf{S}^+,\mathrm{QCat}))_{\mathrm{proj}}$ and  
$\mathbf{Fun}(\mathbf{M}^{op},(\mathbf{S}^+,\mathrm{QCat}))_{\mathrm{proj}}$. Every such (composite) zig-zag of simplicially enriched 
Quillen equivalences can be replaced by a single simplicially enriched right Quillen equivalence
\begin{align}\label{equcorcosmemblin2}
\mathcal{L}_{y[W]}\mathbf{Fun}(\mathbb{M}^{op},(\mathbf{S}^+,\mathrm{QCat}))_{\mathrm{proj}}\rightarrow\mathbf{Fun}(\mathbf{M}^{op},(\mathbf{S}^+,\mathrm{QCat}))_{\mathrm{proj}}
\end{align}
by \cite[Proposition 3.14]{rs_small}\footnote{The proposition is phrased for simplicial model categories, but it in fact applies to any 
monoidal model category $\mathcal{V}$ and any pair of combinatorial $\mathcal{V}$-enriched model categories $\mathbb{M}$ and
$\mathbb{N}$.}. 
The resulting postcomposition of the functor (\ref{equcorcosmemblin1}) with the right derivation of this right Quillen 
equivalence yields a functor
\[\mathbb{R}\mathbf{Ext}\colon\mathbf{Cat}_{\infty}(\mathbb{M})\rightarrow\mathbf{Fun}(\mathbf{M}^{op},\mathbf{QCAT})_{\mathrm{proj}},\]
which corestricts to a functor as stated.
\end{proof}

\begin{remark}
Whenever $\mathbb{M}$ is itself an $(\mathbf{S},\mathrm{Kan})$-enriched model category, the simplicial localization
$\mathcal{L}_{\Delta}(\mathbb{M},W)$ is DK-equivalent to the simplicial category $\mathbb{M}^{cf}$ of bifibrant objects in $\mathbb{M}$
\cite[Theorem 1.8]{bksimploc}.
Thus, in this case $\mathbf{M}$ can be taken to be $\mathbb{M}^{cf}$.
\end{remark}

For the following theorem we recall that every $(\mathbf{S},\mathrm{QCat})^c$-exact functor between
$(\mathbf{S},\mathrm{QCat})^c$-enriched fibration categories gives rise to a functor of associated $(\infty,2)$-categories 
(Remark~\ref{remfibcatsinfty2}) whenever the codomain comes equipped with a simplicially enriched cofibrant replacement functor. 
We further note that for any simplicial category $\mathbf{C}$, the simplicial functor category
$\mathbf{Fun}(\mathbf{C}^{op},\mathbf{QCat})_{\mathrm{proj}}$ does have simplicially enriched cofibrant 
replacements indeed. This follows from, first, the fact that $\mathbf{Fun}(\mathbf{C}^{op},(\mathbf{S}^+,\mathrm{QCat}))_{\mathrm{proj}}$ 
is combinatorial (again in the Grothendieck universe of large categories), and so we can apply
\cite[Corollary 13.2.4]{riehlcthty}. And second, every pointwise small presheaf therein has a pointwise small cofibrant replacement by
\cite[Proposition 2.3]{duggersimp}. Thus, the model categorical externalization functor
from Corollary~\ref{corcosmemblin} induces a functor of associated $(\infty,2)$-categories, and so in particular it induces a functor of 
underlying $\infty$-categories (Definition~\ref{defpith}).

\begin{theorem}\label{propglobrepext1}
Suppose $\mathbb{M}$ is a left proper, combinatorial, and $(\mathbf{S},\mathrm{Kan})$-enriched model category. Let $\mathcal{C}$ denote the 
homotopy $\infty$-category $\mathrm{Ho}_{\infty}(\mathbb{M})$ of $\mathbb{M}$, let $\mathbb{M}_0$ denote its underlying ordinary model 
category, and $\mathbf{M}:=\mathbb{M}^{cf}$. Then the model categorical externalization functor
\[\mathbb{R}\mathbf{Ext}\colon\mathbf{Cat}_{\infty}(\mathbb{M}_0)\rightarrow\mathbf{Fun}(\mathbf{M}^{op},\mathbf{QCat})_{\mathrm{proj}}\]
from Corollary~\ref{corcosmemblin} returns the externalization functor 
\begin{align}\label{equpropglobrepextpre}
\mathrm{Ext}\colon\mathrm{Cat}_{\infty}(\mathcal{C})\rightarrow\mathrm{Fun}(\mathcal{C}^{op},\mathrm{Cat}_{\infty})
\end{align}
from Section~\ref{secformal} on underlying $\infty$-categories.
\end{theorem}
\begin{proof}
The cofibrant replacement functor $\lambda$ on $\mathbb{M}_0$ can be 
lifted to a simplicially enriched cofibrant replacement functor $\lambda$ on $\mathbb{M}$ by \cite[Corollary 13.2.4]{riehlcthty} as noted 
above. Furthermore, the simplicial enrichment of $\mathbb{M}$ induces a transfinitely $\mathbf{S}$-exact Yoneda embedding
\[\mathbb{M}\xrightarrow{y}\mathbf{Fun}(\mathbb{M}^{op},(\mathbf{S}^+,\mathrm{Kan}))_{\mathrm{proj}}\xrightarrow{\lambda^{\ast}}\mathbf{Fun}(\mathbb{M}^{op},(\mathbf{S}^+,\mathrm{Kan}))_{\mathrm{proj}}\]
of $(\mathbf{S},\mathrm{Kan})$-enriched model categories. This in turn induces a transfinitely $s\mathbf{S}$-exact embedding
\begin{align}\label{equpropglobrepextpre1}
(s\mathbb{M},\mathrm{Cat}_{\infty})\xrightarrow{sy}\mathbf{Fun}(\mathbb{M}^{op},(s\mathbf{S}^+,\mathrm{Cat}_{\infty}))_{\mathrm{proj}}\xrightarrow{\lambda^{\ast}}\mathbf{Fun}(\mathbb{M}^{op},(s\mathbf{S}^+,\mathrm{Cat}_{\infty}))_{\mathrm{proj}}
\end{align}
of $(s\mathbf{S},\mathrm{Cat}_{\infty})$-enriched model categories. Here,we recall that $(s\mathbf{S}^+,\mathrm{Cat}_{\infty})$ is a left 
Bousfield localization of the Reedy model structure on $s\mathbf{S}^+$. It is a cartesian model structure and as such induces an  
$(s\mathbf{S}^+,\mathrm{Cat}_{\infty})$-enriched model structure on the codomain of (\ref{equpropglobrepextpre1}). The domain of
$(\ref{equpropglobrepextpre1})$ is an analogous Bousfield localization of the Reedy model structure on $s\mathbb{M}$. It is
$(s\mathbf{S}^+,\mathrm{Cat}_{\infty})$-enriched as well by way of the mapping objects
$s\mathbb{M}(X,Y)^{\Delta^{\bullet}}\in s\mathbf{S}$.
The two projections
$(\cdot)_0\colon(s\mathbf{S},\mathrm{Cat}_{\infty})\rightarrow(\mathbf{S},\mathrm{Kan})$ and
$U\colon(s\mathbf{S},\mathrm{Cat}_{\infty})\rightarrow(\mathbf{S},\mathrm{QCat})$ onto the first column and the first row, respectively, 
are right Quillen functors \cite{jtqcatvsss}. They hence induce an $(\mathbf{S},\mathrm{Kan})$-enrichment and an
$(\mathbf{S},\mathrm{QCat})$-enrichment of (\ref{equpropglobrepextpre1}) each, respectively. We will denote the former induced enrichment 
by an uppercase $v$, and the latter by an uppercase $h$, respectively.

Let $\iota\colon\mathbb{M}_0\rightarrow\mathbb{M}$ be the canonical embedding of simplicial categories, where the domain is considered to 
be locally discrete. We obtain commutative squares of simplicial functors as follows (whose vertices are decorated with simplicial model 
structures for later use).

\begin{align}\label{diagpropglobrepextpre1}
\begin{gathered}
\xymatrix{
(s\mathbb{M},\mathrm{Cat}_{\infty})^h\ar[r]^(.35){sy^h} & \mathbf{Fun}(\mathbb{M}^{op},(s\mathbf{S}^+,\mathrm{Cat}_{\infty}))_{\mathrm{proj}}^h\ar[r]^{U_{\ast}} & \mathbf{Fun}(\mathbb{M}^{op},(\mathbf{S}^+,\mathrm{QCat}))_{\mathrm{proj}}\ar[d]^{\iota^{\ast}}\\
\mathbf{Cat}_{\infty}(\mathbb{M}_0)\ar@{^(->}[u]\ar[rr]_{\mathrm{Ext}} & & \mathcal{L}_{y[W]}\mathbf{Fun}(\mathbb{M}^{op}_0,(\mathbf{S}^+,\mathrm{QCat}))_{\mathrm{proj}}
}
\end{gathered}
\end{align}
\begin{align}\label{diagpropglobrepextpre1.5}
\begin{gathered}
\xymatrix{
\mathbf{Fun}(\mathbb{M}^{op},(\mathbf{S}^+,\mathrm{QCat}))_{\mathrm{proj}}\ar[d]^{\iota^{\ast}}\ar[r]^{\lambda^{\ast}} & \mathbf{Fun}(\mathbb{M}^{op},(\mathbf{S}^+,\mathrm{QCat}))_{\mathrm{proj}}\ar[d]^{\iota^{\ast}} \\
\mathcal{L}_{y[W]}\mathbf{Fun}(\mathbb{M}^{op}_0,(\mathbf{S}^+,\mathrm{QCat}))_{\mathrm{proj}}\ar[r]_{\lambda^{\ast}} & \mathcal{L}_{y[W]}\mathbf{Fun}(\mathbb{M}_0^{op},(\mathbf{S}^+,\mathrm{QCat}))_{\mathrm{proj}}
}
\end{gathered}
\end{align}

In (\ref{diagpropglobrepextpre1}), the simplicial category $\mathbf{Cat}_{\infty}(\mathbb{M}_0)$ is precisely the full simplicial 
subcategory of $(s\mathbb{M},\mathrm{Cat}_{\infty})_h$ spanned by the fibrant objects. The vertical left hand side embedding thus denotes 
the canonical inclusion. The square (\ref{diagpropglobrepextpre1}) then commutes by definition of
$\mathrm{Ext}(X):=\mathbb{M}(\phv,X_{\bullet})_0=U_{\ast}(sy_h(X))$. The square (\ref{diagpropglobrepextpre1.5}) commutes by definition as 
well.

Now, if by $\beta\colon\mathbb{M}\rightarrow\mathbb{M}^{cf}$ we denote a simplicially enriched bifibrant replacement 
functor (again by \cite[Corollary 13.2.4]{riehlcthty}),  then the composition $\beta\circ\iota\colon\mathbb{M}_0\rightarrow\mathbb{M}^{cf}$ 
exhibits $\mathbb{M}^{cf}$ up to DK-equivalence as the simplicial localization of $\mathbb{M}_0$ at its class of weak equivalences by
\cite[Proposition A.3.4.13]{luriehtt}. As $\beta$ itself is a DK-equivalence, it follows that in the sequence
\begin{align}\label{equpropglobrepextpre2}
\mathbf{Fun}((\mathbb{M}^{cf})^{op},(\mathbf{S}^+,\mathrm{QCat}))_{\mathrm{proj}}\xrightarrow{\beta^{\ast}}\mathbf{Fun}(\mathbb{M}^{op},(\mathbf{S}^+,\mathrm{QCat}))_{\mathrm{proj}}\xrightarrow{\iota^{\ast}}\mathcal{L}_{y[W]}\mathbf{Fun}(\mathbb{M}_0^{op},(\mathbf{S}^+,\mathrm{QCat}))_{\mathrm{proj}}
\end{align}
of right Quillen functors the left hand side is a right Quillen equivalence. The composition is a right Quillen equivalence as well as 
reasoned in the proof of Corollary~\ref{corcosmemblin}. It follows that the right Quillen functor $\iota^{\ast}$ is a right Quillen 
equivalence, too. Furthermore, via Corollary~\ref{corcosmemblin}, we obtain a lift of
$\mathbb{R}\mathbf{Ext}\colon\mathbf{Cat}_{\infty}(\mathbb{M}_0)\rightarrow\mathcal{L}_{y[W]}\mathbf{Fun}(\mathbb{M}_0^{op},\mathbf{QCAT})$ 
along the composition (\ref{equpropglobrepextpre2}). This in turn yields a lift
\begin{align}\label{diagpropglobrepextpre2}
\begin{gathered}
\xymatrix{
(s\mathbb{M},\mathrm{Cat}_{\infty})^h\ar[rr]^(.4){\lambda^{\ast}\circ U_{\ast}\circ sy^h} & & \mathbf{Fun}(\mathbb{M}^{op},(\mathbf{S}^+,\mathrm{QCat}))_{\mathrm{proj}}\ar[d]^{\iota^{\ast}} \\
\mathbf{Cat}_{\infty}(\mathbb{M}_0)\ar@{^(->}[u]\ar[rr]_(.35){\lambda^{\ast}\circ\mathrm{Ext}}\ar@{-->}[urr]|{\mathbb{R}\mathbf{Ext}} & &  \mathcal{L}_{y[W]}\mathbf{Fun}(\mathbb{M}_0^{op},(\mathbf{S}^+,\mathrm{QCat}))_{\mathrm{proj}}
}
\end{gathered}
\end{align}
to the horizontal composition of the squares (\ref{diagpropglobrepextpre1}) and (\ref{diagpropglobrepextpre1.5}). Here, the lower triangle 
commutes up to homotopy by construction. The upper triangle then commutes up to homotopy as well, because
$\iota^{\ast}$ is a Quillen equivalence. We are to show that the diagonal lift $\mathbb{R}\mathbf{Ext}$ in (\ref{diagpropglobrepextpre2}) 
induces the externalization functor (\ref{equpropglobrepextpre}) on underlying $\infty$-categories. Therefore it suffices to show that so 
does the top horizontal composition $\lambda^{\ast}\circ U_{\ast}\circ sy^h=U_{\ast}\circ\lambda^{\ast}\circ sy^h$.

Thus, let $\mathbb{L}$ be a simplicially enriched projective cofibrant replacement functor on the 
respective simplicial functor categories. If we recall the definition of the underlying $\infty$-category of an $(\infty,2)$-category 
(Definition~\ref{defpith}), we are left to show that the composition
\begin{align}
\label{equpropglobrepextpre3}
N_{\Delta}(\underline{\mathbf{Cat}_{\infty}(\mathbb{M}_0)^{c}}) & \xrightarrow{N_{\Delta}(\underline{\mathbb{L}\lambda^{\ast}sy^h})}N_{\Delta}(\underline{(\mathbf{Fun}(\mathbb{M}^{op},(s\mathbf{S}^+,\mathrm{Cat}_{\infty}))_{\mathrm{proj}}^h)^{cf}}) \\
\label{equpropglobrepextpre3.5}
& \xrightarrow{N_{\Delta}(\mathbb{L}\underline{U_{\ast}})} N_{\Delta}(\underline{\mathbf{Fun}(\mathbb{M}^{op},\mathbf{QCAT})_{\mathrm{proj}}^{c}}) 
\end{align}
is equivalent to the 
composition
\[\mathrm{Cat}_{\infty}(\mathcal{C})\xrightarrow{sy}\mathrm{Fun}(\mathcal{C}^{op},\mathrm{Cat}_{\infty}(\mathcal{S}))\xrightarrow{U_{\ast}}\mathrm{Fun}(\mathcal{C}^{op},\mathrm{Cat}_{\infty}).\]
We recall that $\mathcal{C}:=\mathrm{Ho}_{\infty}(\mathbb{M}_0)$ is generally 
defined as the $\infty$-categorical localization $N(\mathbb{M}_0)[W^{-1}]$. We can naturally identify this $\infty$-category with the
homotopy-coherent nerve $N_{\Delta}(\mathbb{M}^{cf})$ by way of \cite[Proposition 1.3.4.7]{lurieha} for any
$(\mathbf{S},\mathrm{Kan})$-enriched model category $\mathbb{M}$. Furthermore, the bifibrant replacement
$\beta\circ\iota\colon\mathbb{M}^{cf}\rightarrow\mathbb{M}^{cf}$ restricted to the simplicial category $\mathbb{M}^{cf}$ of bifibrant 
objects is homotopic to the identity. It follows that the functor
\[\xymatrix{
N_{\Delta}(\underline{(\mathbf{Fun}((\mathbb{M}^{cf})^{op},(s\mathbf{S}^+,\mathrm{Cat}_{\infty}))_{\mathrm{proj}}^h)^{cf}})\ar[rr]^{N_{\Delta}(\underline{\mathbb{L}U_{\ast}})}\ar[d]_{\beta^{\ast}} & & N_{\Delta}(\underline{\mathbf{Fun}((\mathbb{M}^{cf})^{op},\mathbf{QCAT})_{\mathrm{proj}}^{c}}) \\
N_{\Delta}(\underline{(\mathbf{Fun}(\mathbb{M}^{op},(s\mathbf{S}^+,\mathrm{Cat}_{\infty}))_{\mathrm{proj}}^h)^{cf}})\ar[rr]_{N_{\Delta}(\underline{\mathbb{L}U_{\ast}})} & & N_{\Delta}(\underline{\mathbf{Fun}(\mathbb{M}^{op},\mathbf{QCAT})_{\mathrm{proj}}^{c}})
\ar[u]^{\iota{\ast}}}\]
is the functor
\[U_{\ast}\colon\mathrm{Fun}(\mathcal{C}^{op},\mathrm{Cat}_{\infty}(\mathcal{S}))\rightarrow\mathrm{Fun}(\mathcal{C}^{op},\mathrm{Cat}_{\infty})\]
essentially by definition. Lastly, given a complete Segal space $X$, the homotopy equivalence
$s_0\colon X_0\rightarrow\mathrm{Equiv}(X)$ from Remark~\ref{remcsspb} induces a natural equivalence
$(\cdot)_0\rightarrow\{\mathrm{Equiv},\phv\}$ of functors from the category of complete Segal spaces to the category of
quasi-categories. Furthermore, by \cite[Theorem 6.2]{rezk}, there is a natural equivalence
$\{\mathrm{Equiv},\phv\}\rightarrow U(\cdot)^{\simeq}$. We obtain a composite natural equivalence $(\cdot)_0\rightarrow U(\cdot)^{\simeq}$, 
which in turn induces a natural equivalence $N_{\Delta}((\cdot)^v)\rightarrow N_{\Delta}(\underline{(\cdot)^h})$ of functors from the 
category of $(s\mathbf{S},\mathrm{Cat}_{\infty})^f$-enriched categories to the category of quasi-categories. In particular, we obtain an 
equivalence as follows.
\[\xymatrix{
N_{\Delta}(((s\mathbb{M},\mathrm{Cat}_{\infty})^v)^{cf})\ar[rr]^(.4){N_{\Delta}(\mathbb{L}\lambda^{\ast}sy^v)}\ar[d]_{\simeq} & & N_{\Delta}(\underline{(\mathbf{Fun}(\mathbb{M}^{op},(s\mathbf{S}^+,\mathrm{Cat}_{\infty}))^v)^{cf}})\ar[d]^{\simeq} \\
N_{\Delta}((\underline{(s\mathbb{M},\mathrm{Cat}_{\infty})^h)^{cf}})\ar[rr]_(.4){N_{\Delta}(\underline{\mathbb{L}\lambda^{\ast}sy^h})} & & N_{\Delta}(\underline{(\mathbf{Fun}(\mathbb{M}^{op},(s\mathbf{S}^+,\mathrm{Cat}_{\infty}))^h)^{cf}})
}\]
The simplicial functor $sy^v$ is the restriction of the push-forward
$sy\colon s\mathbb{M}\rightarrow\mathbf{Fun}(\mathbb{M}^{op},s\mathbf{S}^+)$
of the simplicial Yoneda embedding $y\colon\mathbb{M}\rightarrow\mathbf{Fun}(\mathbb{M}^{op},\mathbf{S}^+)$. It follows that the functor 
$N_{\Delta}(\mathbb{L}\lambda^{\ast}sy^v)$ is equivalent to
$sy\colon\mathrm{Cat}_{\infty}(\mathcal{C})\xrightarrow{sy}\mathrm{Fun}(\mathcal{C}^{op},\mathrm{Cat}_{\infty}(\mathcal{S}))$ after 
corestriction to pointwise small Kan complexes by way of \cite[Proposition A.3.4.13]{luriehtt}. This finishes the proof.

\end{proof}

\begin{remark}
Theorem~\ref{propglobrepext1} can also be shown without the assumption of $(\mathbf{S},\mathrm{Kan})$-enrichment of the left 
proper and combinatorial model category $\mathbb{M}$. This can either be done by replacing any such $\mathbb{M}$ suitably by an
$(\mathbf{S},\mathrm{Kan})$-enrichment model category \cite{duggersimp}, or potentially directly by way of the alternative formula for the
$\infty$-categorical externalization functor from Remark~\ref{rem_pressheaves}. In the latter case, one is to show that the simplicially enriched hom-functor
\[\mathrm{Hom}\colon(\mathbf{Cat}_{\infty}(\mathbb{M})^c)^{op}\times\mathbf{Cat}_{\infty}(\mathbb{M})^c\rightarrow\mathbf{QCat}\]
from Section~\ref{secsubcosmosofcsos} returns the enhanced mapping $\infty$-category functor of Remark~\ref{rem2catstrdirect} on underlying 
$\infty$-categories directly.
\end{remark}

\begin{corollary}\label{cormcextff}
Suppose $\mathbb{M}$ is a left proper, combinatorial and $(\mathbf{S},\mathrm{Kan})$-enriched model category. Let $\mathbb{M}_0$ denote its 
underlying ordinary model category, and suppose $\mathbf{M}$ is a simplicial category which is DK-equivalent to $\mathbb{M}^{cf}$. Then the 
model categorical externalization functor
\[\mathbb{R}\mathbf{Ext}\colon\mathbf{Cat}_{\infty}(\mathbb{M}_0)\rightarrow\mathbf{Fun}(\mathbf{M}^{op},\mathbf{QCat})_{\mathrm{proj}}\]
from Corollary~\ref{corcosmemblin} induces a fully faithful functor of $(\infty,2)$-categories. That is to say, for every pair of Reedy 
cofibrant complete Segal objects $X,Y\in\mathbf{Cat}_{\infty}(\mathbb{M})$, and any simplicially enriched projective cofibrant replacement 
functor $\mathbb{L}$, the induced morphism
\begin{align}\label{equcorextlocequiv}
\mathbb{R}\mathbf{Ext}\colon\mathbf{Cat}_{\infty}(\mathbb{M}_0)(X,Y)\rightarrow\mathbf{Fun}(\mathbf{M}^{op},\mathbf{QCat})_{\mathrm{proj}}(\mathbb{L}\mathbb{R}\mathbf{Ext}(X),\mathbb{L}\mathbb{R}\mathbf{Ext}(Y))
\end{align}
between hom-objects is an equivalence of quasi-categories.
\end{corollary}
\begin{proof}
Via Corollary~\ref{corcosmemblin}, without loss of generality we can assume that $\mathbf{M}=\mathbb{M}^{cf}$.
Given Reedy cofibrant complete Segal objects $X,Y$ in $\mathbb{M}$, to show that the functor (\ref{equcorextlocequiv}) is an equivalence of
quasi-categories it suffices to show that the corresponding functor
\[((\mathbb{R}\mathbf{Ext}(X,Y))^{\Delta^{\bullet}})^{\simeq}\colon(\mathbf{Cat}_{\infty}(\mathbb{M})(X,Y)^{\Delta^{\bullet}})^{\simeq}\rightarrow(\mathbf{Fun}(\mathbf{M}^{op},\mathbf{QCat})_{\mathrm{proj}}(\mathbb{L}\mathbb{R}\mathbf{Ext}(X),\mathbb{R}\mathbf{Ext}(Y))^{\Delta^{\bullet}})^{\simeq}\]
is a (pointwise) equivalence of complete Segal spaces. Indeed, $\mathbb{R}\mathbf{Ext}(Y)$ is projectively fibrant, and so the weak 
equivalence $\mathbb{L}\mathbb{R}\mathbf{Ext}(Y)\rightarrow\mathbb{R}\mathbf{Ext}(Y)$ induces a homotopy equivalence of respective
hom-spaces. By virtue of the fact that both sides are cotensored over finite simplicial sets 
and that $\mathbb{R}\mathbf{Ext}$ preserves them, this is to show that for any $n\geq 0$ the map
\[(\mathbb{R}\mathbf{Ext}(X,Y^{\Delta^n}))^{\simeq}\colon\mathbf{Cat}_{\infty}(\mathbb{M})(X,Y^{\Delta^{n}})^{\simeq}\rightarrow\mathbf{Fun}(\mathbf{M}^{op},\mathbf{QCat})_{\mathrm{proj}}(\mathbb{L}\mathbb{R}\mathbf{Ext}(X),\mathbb{R}\mathbf{Ext}(Y^{\Delta^{n}}))^{\simeq}\]
is an equivalence of hom-spaces. By Theorem~\ref{propglobrepext1} this functor is equivalent to the action
\[\mathrm{Ext}(X,Y^{\Delta^n})\colon\mathrm{Cat}_{\infty}(\mathcal{C})(X,Y^{\Delta^{n}})\rightarrow\mathrm{Fun}(\mathcal{C}^{op},\mathrm{Cat}_{\infty})(\mathrm{Ext}(X),\mathrm{Ext}(Y^{\Delta^{n}}))\]
of the $\infty$-categorical externalization functor on corresponding hom-spaces. The latter is an equivalence as externalization is fully 
faithful (Lemma~\ref{corsegalyonedaff}).
\end{proof}

For the following concluding corollary we recall the convention to write $\mathbf{Fun}(\mathcal{C}^{op},\mathbf{QCat})$ for 
the $(\infty,2)$-category $\mathbf{Fun}(\mathfrak{C}(\mathcal{C})^{op},\mathbf{QCat})$ whenever $\mathcal{C}$ is an $\infty$-category as 
fixed in Section~\ref{secsubformalmodelindep}. We apply the same convention to the projective case.

\begin{corollary}\label{corcompareinftycosmoses}
Suppose $\mathbb{M}$ is a left proper, combinatorial and $(\mathbf{S},\mathrm{Kan})$-enriched model category. Then the
$(\infty,2)$-categories $\mathbf{Cat}_{\infty}(\mathbb{M}_0)^c$ and $\mathbf{Cat}_{\infty}(\mathcal{C})$ are equivalent. This is to say 
they are connected by a zig-zag of equivalences in the $(\mathbf{S},\mathrm{Cat})$-induced model structure on the category of simplicial 
categories.
\end{corollary}
\begin{proof}
We have the following diagram of simplicial categories, where $\mathbb{L}$ denotes any simplicially enriched projectively cofibrant 
replacement functor.
\begin{align}\label{diagcorcompareinftycosmoses}
\begin{gathered}
\xymatrix{
 & \mathbf{Cat}_{\infty}(\mathcal{C})\ar@{^(->}[d] & \\
& \mathbf{Fun}(\mathcal{C}^{op},\mathbf{QCat})\ar@{^(->}[d] & \\
\mathbf{Cat}_{\infty}(\mathbb{M}_0)^c \ar[r]^(.4){\mathbb{R}\mathbf{Ext}} & \mathbf{Fun}(\mathcal{C}^{op},\mathbf{QCat})_{\mathrm{proj}}
\ar[r]^(.5){\mathbb{L}}_(.5){\simeq} & \mathbf{Fun}(\mathcal{C}^{op},\mathbf{QCat})_{\mathrm{proj}}^c
}
\end{gathered}
\end{align}
The functor $\mathbb{R}\mathbf{Ext}$ on the bottom left is defined by way of Corollary~\ref{corcosmemblin} for
$\mathbf{M}=\mathfrak{C}(\mathcal{C})$. The two essential images of $\mathbf{Cat}_{\infty}(\mathbb{M}_0)$ and
$\mathbf{Cat}_{\infty}(\mathcal{C})$ in $\mathbf{Fun}(\mathcal{C}^{op},\mathbf{QCat})_{\mathrm{proj}}^c$ coincide by Theorem~\ref{propglobrepext1}. 
The horizontal composition $\mathbb{L}\mathbb{R}\mathbf{Ext}$ in (\ref{diagcorcompareinftycosmoses}) induces 
equivalences on hom-quasi-categories by Corollary~\ref{cormcextff} applied to $\mathbf{M}=\mathfrak{C}(\mathcal{C})$.
The inclusion of injectively fibrant to projectively fibrant presheaves followed by projective cofibrant replacement is easily seen to 
induce equivalences on hom-quasi-categories as well.
It follows that both 
$\mathbf{Cat}_{\infty}(\mathbb{M}_0)$ and $\mathbf{Cat}_{\infty}(\mathcal{C})$ are equivalent to the same full (and replete)
sub-$(\infty,2)$-category of $\mathbf{Fun}(\mathcal{C}^{op},\mathbf{QCat})_{\mathrm{proj}}^c$ \cite[Definition A.3.2.1]{luriehtt}.
\end{proof}

\begin{remark}\label{remcisinskimc}
We recall that a model category $\mathbb{M}$ is commonly referred to as a \emph{Cisinski model category} if it is cofibrantly 
generated, its cofibrations are exactly the monomorphisms, and the underlying category of $\mathbb{M}$ is a Grothendieck topos. In 
particular, if $\mathbb{M}$ is a Cisinski model category then all objects in $\mathbb{M}$ are cofibrant, and so the model categorical 
externalization functor
$\mathbf{Ext}\colon\mathbf{Cat}_{\infty}(\mathbb{M})\rightarrow\mathbf{Fun}(\mathbb{M}^{op},\mathbf{QCat})_{\mathrm{proj}}$ is 
automatically pointwise right derived. Furthermore, in this case every object in $s\mathbb{M}$ is Reedy cofibrant as well, and so
$\mathbf{Cat}_{\infty}(\mathbb{M})=\mathbf{Cat}_{\infty}(\mathbb{M})^c$ is an $\infty$-cosmos. This makes the constructions above more 
straight-forward. In fact, every presentable $\infty$-category $\mathcal{C}$ has a model categorical presentation $\mathbb{M}$ that is an
$(\mathbf{S},\mathrm{Kan})$-enriched Cisinski model category by \cite[Proposition A.3.7.6]{luriehtt} (or by its proof rather). 
Thus, in this case the $\infty$-cosmos $\mathbf{Cat}_{\infty}(\mathcal{C})$ defined via the $\infty$-categorical 
externalization functor in Section~\ref{secformal} can always be presented internally in such an $\mathbb{M}$.
\end{remark}

We end this section with a proof of Theorem~\ref{thm2}.4 as stated in Section~\ref{secintro} when formulated in the given context of
strict $(\infty,2)$-categories. Therefore, we once again recall Rezk's model structure $(s\mathbf{S},\mathrm{Cat}_{\infty})$ for (Reedy 
fibrant) complete Segal spaces \cite{rezk} as well as its Quillen equivalence $U$ to $(\mathbf{S},\mathrm{QCat})$ as an
$(\mathbf{S},\mathrm{QCat})$-enriched model category \cite{jtqcatvsss, riehlverityelements}.

\begin{definition}\label{def1loc2topos}
Let $\mathbb{M}$ be an $(\mathbf{S},\mathrm{QCat})$-enriched model category. We refer to the $\mathbf{QCat}$-enriched category
$\mathbb{M}^{cf}$ of bifibrant objects as the \emph{underlying $(\infty,2)$-category} of $\mathbb{M}$.
Say that an $(\infty,2)$-category $\mathbf{C}$ is an \emph{$(\infty,1)$-localic $(\infty,2)$-topos} if 
there is a small $\mathbf{Kan}$-enriched category $\mathbf{B}$ together with a set $T\subset\mathbf{Fun}(\mathbf{B}^{op},s\mathbf{S})$ of 
morphisms such that the following two conditions hold.
\begin{enumerate}
\item The left Bousfield localization
$\mathbf{Fun}(\mathbf{B}^{op},(s\mathbf{S},\mathrm{Cat}_{\infty}))_{\mathrm{inj}}\rightarrow\mathcal{L}_T\mathbf{Fun}(\mathbf{B}^{op},(s\mathbf{S},\mathrm{Cat}_{\infty}))_{\mathrm{inj}}$ at $T$ preserves homotopy-cartesian squares and finite simplicial cotensors of fibrant 
objects. The latter is to say that $T$-local fibrant replacement of injectively fibrant objects preserves finite simplicial cotensors up to 
$T$-local equivalence.
\item The underlying $(\infty,2)$-category of $\mathcal{L}_T\mathbf{Fun}(\mathbf{B}^{op},(s\mathbf{S},\mathrm{Cat}_{\infty}))_{\mathrm{inj}}$
is equivalent to $\mathbf{C}$ (in the $(\mathbf{S},\mathrm{Cat})$-induced model structure on the category of simplicial categories).
\end{enumerate}
\end{definition}

\begin{remark}
In \cite[Definition 9.2.1]{am_2topos}, an $(\infty,2)$-category $\mathbf{C}$ is defined to be an $(\infty,1)$-localic $(\infty,2)$-topos if 
there is an $\infty$-topos $\mathcal{C}$ such that $\mathbf{C}$ is equivalent to the full sub-$(\infty,2)$-category of
$\mathbf{Fun}(\mathcal{C}^{op},\mathbf{Cat}_{\infty})$ spanned by the small limit preserving functors. As $\infty$-toposes are 
presentable, the latter is precisely the $(\infty,2)$-category $\mathbf{Cat}_{\infty}(\mathcal{C})$ (Remark~\ref{rem_pressheaves}). Against 
this background, the following proposition is an $(\infty,1)$-localic version of \cite[Theorem 7.4.1]{am_2topos}, and shows that 
Definition~\ref{def1loc2topos} is equivalent to \cite[Definition 9.2.1]{am_2topos}.
\end{remark}

\begin{proposition}\label{prop1oc2topos}
Suppose $\mathcal{C}$ is an $\infty$-topos. Then $\mathbf{Cat}_{\infty}(\mathcal{C})$ is an $(\infty,1)$-localic $(\infty,2)$-topos. 
Vice versa, every $(\infty,1)$-localic $(\infty,2)$-topos is of this form.
\end{proposition}
\begin{proof}
Let $\mathcal{C}$ is an $\infty$-topos. Via \cite[Section 6]{rezkhtytps} and \cite[Section A3.7]{luriehtt}, there is a small
$\mathbf{Kan}$-enriched category $\mathbf{B}$ together with a set $T\subset\mathbf{Fun}(\mathbf{B}^{op},\mathbf{S})$ of morphisms such that
\begin{enumerate}
\item the left Bousfield localization $\mathbf{Fun}(\mathbf{B}^{op},(\mathbf{S},\mathrm{Kan}))_{\mathrm{inj}}\rightarrow\mathcal{L}_T\mathbf{Fun}(\mathbf{B}^{op},(\mathbf{S},\mathrm{Kan}))_{\mathrm{inj}}$ at $T$ preserves homotopy-cartesian squares, and
\item the $\infty$-topos $\mathbf{C}$ is equivalent to the underlying $(\infty,1)$-category of
$\mathcal{L}_T\mathbf{Fun}(\mathbf{B}^{op},(\mathbf{S},\mathrm{Kan}))_{\mathrm{inj}}$.
\end{enumerate}
We use this localization to construct a left Bousfield localization as required in Definition~\ref{def1loc2topos} basically by specifying it 
pointwise. Therefore, in the following we denote the $(\mathbf{S},\mathrm{Kan})$-enriched model category
$\mathbf{Fun}(\mathbf{B}^{op},(\mathbf{S},\mathrm{Kan}))_{\mathrm{inj}}$ by $\hat{\mathbb{B}}$. 
The simplicially enriched left exact left Bousfield localization $\hat{\mathbb{B}}\rightarrow\mathcal{L}_T\hat{\mathbb{B}}$ at the set $T$ 
induces a simplicially enriched left exact left Bousfield localization
$(\hat{\mathbb{B}}^{\Delta^{op}})_{\mathrm{inj}}\rightarrow(\mathcal{L}_T\hat{\mathbb{B}})^{\Delta^{op}}_{\mathrm{inj}}$.
Furthermore, let $\mathrm{CS}\in\mathbf{S}^{\Delta^1}$ be the set of spine inclusions together with the 
endpoint-inclusion of the walking isomorphism (as recalled in Section~\ref{subsecinf1}). Let $B\in\mathbf{S}^{\Delta^1}$ be the class of 
standard boundary inclusions. Then, for any set $I$ of generating cofibrations of $\hat{\mathbb{B}}$, left Bousfield localization at the set
\[\mathrm{CS(\hat{\mathbb{B}})}:=\{(f\hat{\times}b)\hat{\otimes}i\mid f\in\mathrm{CS},b\in B, i\in I\}\]
of corresponding Leibniz tensor products yields the simplicially enriched model categories
$\mathcal{L}_{\mathrm{CS}(\hat{\mathbb{B}})}(\hat{\mathbb{B}}^{\Delta^{op}})_{\mathrm{inj}}$ of complete Segal objects in
$\hat{\mathbb{B}}$, and $\mathcal{L}_{\mathrm{CS}(\hat{\mathbb{B}})}(\mathcal{L}_{T}\hat{\mathbb{B}})^{\Delta^{op}}_{\mathrm{inj}}$ of 
complete Segal objects in $\mathcal{L}_T\hat{\mathbb{B}}$, respectively, by \cite[Proposition E.3.7]{riehlverityelements}.
We obtain a square of simplicially enriched left Bousfield localizations as follows.
\begin{align}\label{diagprop1oc2topos}
\begin{gathered}
\xymatrix{
(\hat{\mathbb{B}}^{\Delta^{op}})_{\mathrm{inj}}\ar[d]\ar[r] & \mathcal{L}_{\mathrm{CS}(\hat{\mathbb{B}})}(\hat{\mathbb{B}}^{\Delta^{op}})_{\mathrm{inj}}\ar@{-->}[d] \\
(\mathcal{L}_{T}\hat{\mathbb{B}})^{\Delta^{op}}_{\mathrm{inj}}\ar[r] & \mathcal{L}_{\mathrm{CS}(\hat{\mathbb{B}})}(\mathcal{L}_{T}\hat{\mathbb{B}})^{\Delta^{op}}_{\mathrm{inj}}
}
\end{gathered}
\end{align}
Indeed, the left Bousfield localization
$(\hat{\mathbb{B}}^{\Delta^{op}})_{\mathrm{inj}}\rightarrow(\mathcal{L}_T\hat{\mathbb{B}})^{\Delta^{op}}_{\mathrm{inj}}$
is the left Bousfield localization of $(\hat{\mathbb{B}}^{\Delta^{op}})_{\mathrm{inj}}$ at a set $sT$ of natural transformations between 
simplicial objects.
By construction, a simplicial object $X$ is $sT$-local if and only if each $X_n$ is $T$-local.
Thus, the dashed vertical arrow on the right hand side of (\ref{diagprop1oc2topos}) exists and is again a left Bousfield 
localization by the fact that left Bousfield localizations at different sets of morphisms commute with one another:
\begin{align}\label{equprop1oc2topos0}
\mathcal{L}_{\mathrm{CS}(\hat{\mathbb{B}})}(\mathcal{L}_{T}\hat{\mathbb{B}})^{\Delta^{op}}_{\mathrm{inj}}=\mathcal{L}_{\mathrm{CS}(\hat{\mathbb{B}})}(\mathcal{L}_{sT}(\hat{\mathbb{B}}^{\Delta^{op}})_{\mathrm{inj}})=\mathcal{L}_{sT}(\mathcal{L}_{\mathrm{CS}(\hat{\mathbb{B}})}(\hat{\mathbb{B}}^{\Delta^{op}})_{\mathrm{inj}}).
\end{align}
The left Bousfield localization 
$(\hat{\mathbb{B}}^{\Delta^{op}})_{\mathrm{inj}}\rightarrow(\mathcal{L}_{T}\hat{\mathbb{B}})^{\Delta^{op}}_{\mathrm{inj}}$ furthermore 
preserves $\mathrm{CS}(\hat{\mathbb{B}})$-locality of fibrant objects. That is, because fibrant replacement of a Reedy fibrant object
$X\in\hat{\mathbb{B}}^{\Delta^{op}}$ in $(\mathcal{L}_{T}\hat{\mathbb{B}})^{\Delta^{op}}_{\mathrm{inj}}$ can be computed by pointwise 
fibrant replacement in $\mathcal{L}_{T}\hat{\mathbb{B}}$, the latter of which preserves homotopy-pullbacks by assumption. In particular, the Bousfield localization $(\hat{\mathbb{B}}^{\Delta^{op}})_{\mathrm{inj}}\rightarrow(\mathcal{L}_{T}\hat{\mathbb{B}})^{\Delta^{op}}_{\mathrm{inj}}$ preserves homotopy-pullbacks as well, and hence so does the left Bousfield localization
\begin{align}\label{equprop1oc2topos}
\mathcal{L}_{\mathrm{CS}(\hat{\mathbb{B}})}(\hat{\mathbb{B}}^{\Delta^{op}})_{\mathrm{inj}}\rightarrow\mathcal{L}_{\mathrm{CS}(\hat{\mathbb{B}})}((\mathcal{L}_{T}\hat{\mathbb{B}})^{\Delta^{op}}_{\mathrm{inj}}).
\end{align}
All of this can be verified easily on underlying $\infty$-categories.
The left Bousfield localization (\ref{equprop1oc2topos}) is $(\mathbf{S},\mathrm{QCat})$-enriched if both model categories are equipped with 
their respective canonical $(\mathbf{S},\mathrm{QCat})$-enrichment from \cite[Proposition E.3.7]{riehlverityelements}. 
By Remark~\ref{remmscofgencase}, the underlying $(\infty,2)$-category of
$\mathcal{L}_{\mathrm{CS}(\hat{\mathbb{B}})}((\mathcal{L}_{T}\hat{\mathbb{B}})^{\Delta^{op}}_{\mathrm{inj}})$ is precisely the
$\infty$-cosmos $\mathbf{Cat}_{\infty}(\mathcal{L}_T\hat{\mathbb{B}})$, which is equivalent to the $(\infty,2)$-category
$\mathbf{Cat}_{\infty}(\mathcal{C})$ by Corollary~\ref{corcompareinftycosmoses}.
The $(\mathbf{S},\mathrm{QCat})$-enriched model category
$\mathcal{L}_{\mathrm{CS}(\hat{\mathbb{B}})}(\hat{\mathbb{B}}^{\Delta^{op}})_{\mathrm{inj}}$ is isomorphic to the (accordingly)
$(\mathbf{S},\mathrm{QCat})$-enriched model category $\mathbf{Fun}(\mathbf{B}^{op},(s\mathbf{S},\mathrm{Cat}_{\infty}))_{\mathrm{inj}}$.
We are hence only left 
to show that the localization (\ref{equprop1oc2topos}) preserves finite simplicial cotensors of fibrant objects up to $sT$-local equivalence. 

Thus, let $X$ be a complete Segal object in $\hat{\mathbb{B}}$ and let $K$ be a finite simplicial set. As the $sT$-local 
fibrant replacement of a complete Segal object can be computed by its pointwise $T$-local fibrant replacement, we are to 
show that $\rho_{\ast}(X^S)\simeq\rho_{\ast}(X)^S$ in $\mathbf{Cat}_{\infty}(\mathcal{L}_T\hat{\mathbb{B}})$ for
$\rho\colon\mathcal{L}_T\hat{\mathbb{B}}\rightarrow\mathcal{L}_T\hat{\mathbb{B}}$ any simplicially enriched fibrant replacement functor.
%
Therefore, consider the following (generally non-commutative) diagram.
\[\xymatrix{
\Delta\ar[d]_y\ar[r]^{X^{op}} & \hat{\mathbb{B}}^{op}\ar[r]^{\rho^{op}} & \mathcal{L}_T\hat{\mathbb{B}}^{op} \\
\mathbf{S}\ar@/_/[ur]|{y_!(X^{op})}\ar@/_1pc/[urr]_{y_!(\rho X^{op})} & & 
}\]
By definition of the simplicial cotensor $X^S$ in (\ref{defmccotensor}), we are to exhibit a natural homotopy equivalence between the 
functors $\rho^{op} y_!(X^{op})\colon\mathbf{S}\rightarrow\mathcal{L}_T\hat{\mathbb{B}}^{op}$ and
$y_!(\rho X^{op})\colon\mathbf{S}\rightarrow\mathcal{L}_T\hat{\mathbb{B}}^{op}$ when restricted to the subcategory
$\mathbf{S}^{\mathrm{fin}}\subset\mathbf{S}$ of (homotopically) finite simplicial sets. Existence of such an equivalence now follows from the 
fact that both left Quillen functors agree when restricted to $\Delta$, that $y\colon\Delta\rightarrow\mathbf{S}^{\mathrm{fin}}$ generates
$\mathbf{S}^{\mathrm{fin}}$ under finite homotopy colimits, and that $\rho^{op}$ preserves finite homotopy colimits by assumption.

Vice versa, suppose $\mathbf{C}$ is an $(\infty,1)$-localic $(\infty,2)$-topos. Let $\mathbf{B}$ be a small $\mathbf{Kan}$-enriched 
category together with a set $T\subset\mathbf{Fun}(\mathbf{B}^{op},s\mathbf{S})$ of morphisms as in Definition~\ref{def1loc2topos}.
We obtain a commutative diagram of $(\mathbf{S},\mathrm{QCat})$-enriched model categories as follows.
\begin{align}\label{diagprop1oc2topos1}
\begin{gathered}
\xymatrix{
\mathbf{Fun}(\mathbf{B}^{op},(s\mathbf{S},\mathrm{Cat}_{\infty}))_{\mathrm{inj}}\ar[r]^{U_{\ast}} & \mathbf{Fun}(\mathbf{B}^{op},(\mathbf{S},\mathrm{QCat}))_{\mathrm{inj}}\ar[r]^{(k^!)_{\ast}} & \mathbf{Fun}(\mathbf{B}^{op},(\mathbf{S},\mathrm{Kan}))_{\mathrm{inj}} \\
\mathcal{L}_T\mathbf{Fun}(\mathbf{B}^{op},(s\mathbf{S},\mathrm{Cat}_{\infty}))_{\mathrm{inj}}\ar[r]_{U_{\ast}}\ar[u] & \mathcal{L}_{\bar{T}}\mathbf{Fun}(\mathbf{B}^{op},(\mathbf{S},\mathrm{QCat}))_{\mathrm{inj}}\ar[r]_{(k^!)_{\ast}}\ar[u] & \mathcal{L}_{\bar{T}}\mathbf{Fun}(\mathbf{B}^{op},(\mathbf{S},\mathrm{Kan}))_{\mathrm{inj}}\ar[u] 
}
\end{gathered}
\end{align}
Here, the horizontal arrow $U_{\ast}$ is a right Quillen equivalence, induced by the horizontal projection functor
$U\colon s\mathbf{S}\rightarrow\mathbf{S}$ recalled above. The horizontal arrow $k^!$ is the right homotopy 
localization of \cite[Section 1]{jtqcatvsss} that models the core functor
$(\cdot)^{\simeq}\colon\mathrm{Cat}_{\infty}\rightarrow\mathcal{S}$ (see the beginning of Section~\ref{secsubformalunderqcat}). We define
$\bar{T}$ to be (a small set generating) the preimage of $T$ under the left adjoint of $U_{\ast}$ \cite[Theorem 3.3.20]{hirschhorn03}. The 
left adjoint of the leftmost vertical 
arrow in the diagram preserves homotopy-cartesian squares by assumption. It follows that the left adjoint of the rightmost vertical arrow 
then does as well. In other words, the Bousfield localization $\hat{\mathbb{B}}\rightarrow\mathcal{L}_{\bar{T}}{\hat{\mathbb{B}}}$ is left 
exact. Thus, as noted in the beginning of the proof, the simplicial model category $\mathcal{L}_T{\hat{\mathbb{B}}}$ presents an
$\infty$-topos $\mathcal{C}$. Furthermore, we have equivalences as follows.
\begin{align}
\label{equprop1oc2topos11} \mathbf{Cat}_{\infty}(\mathcal{C}) & \simeq\mathbf{Cat}_{\infty}(\mathcal{L}_{\bar{T}}{\hat{\mathbb{B}}}) \\
\label{equprop1oc2topos12} &\simeq\left(\mathcal{L}_{\mathrm{CS}(\hat{\mathbb{B}})}(\mathcal{L}_{\bar{T}}\hat{\mathbb{B}})^{\Delta^{op}}_{\mathrm{inj}}\right)^f \\
\label{equprop1oc2topos13} &\simeq \left(\mathcal{L}_{s\bar{T}}(\mathcal{L}_{\mathrm{CS}(\hat{\mathbb{B}})}(\hat{\mathbb{B}}^{\Delta^{op}})_{\mathrm{inj}})\right)^f \\
\label{equprop1oc2topos14} &\simeq \left(\mathcal{L}_T(\mathcal{L}_{\mathrm{CS}(\hat{\mathbb{B}})}(\hat{\mathbb{B}}^{\Delta^{op}})_{\mathrm{inj}}))\right)^f \\
\notag &\simeq\left(\mathcal{L}_T\mathbf{Fun}(\mathbf{B}^{op},(s\mathbf{S},\mathrm{Cat}_{\infty}))_{\mathrm{inj}}\right)^f \\
\label{equprop1oc2topos15} & \simeq \mathbf{C}
\end{align}
Here, Line~(\ref{equprop1oc2topos11}) follows again by Corollary~\ref{corcompareinftycosmoses}, and Line~(\ref{equprop1oc2topos12}) by 
Remark~\ref{remmscofgencase}. Line~(\ref{equprop1oc2topos13}) is an instance of Equation~(\ref{equprop1oc2topos0}). For the equivalence 
(\ref{equprop1oc2topos14}) we are to show that the two sets $T$ and $s\bar{T}$ generate the same left Bousfield localization. Therefore it 
suffices to show that a Reedy fibrant complete Segal object $X\in\hat{\mathbb{B}}^{\Delta^{op}}$ is $T$-local if and only if it is 
pointwise $\bar{T}$-local. And indeed, the horizontal arrows in Diagram~(\ref{diagprop1oc2topos1}) are right Quillen functors, and
$(k^!\circ U)_{\ast}(X)\simeq(U_{\ast}(X))^{\simeq}\simeq X_0$ whenever $X$ is a $\mathbf{B}$-indexed complete Segal space.
In particular, if $X$ is $T$-local, then $X_0$ is $\bar{T}$-local. Furthermore, for any $[n]\in\Delta^{op}$, we have
$(X^{\Delta^n})_0\simeq X_n$ in $\hat{\mathbb{B}}^{\Delta^{op}}$, and the localization at $T$ preserves finite simplicial cotensors by 
assumption. It follows that $X_n\in\hat{\mathbb{B}}$ is $\bar{T}$-local for all $[n]\in\Delta^{op}$. The other direction follows similarly.
Lastly, the equivalence (\ref{equprop1oc2topos15}) is given by assumption.
\end{proof}

\section*{Acknowledgments}
This paper was written while being a guest at the Max Planck Institute for Mathematics in Bonn, Germany, whose hospitality is greatly 
appreciated. The author thanks Mathieu Anel, Steve Awodey, Jonas Frey, Nicola Gambino, Calum Hughes, Felix Loubaton, Adrian Miranda, 
Viktoriya Ozornova, and Emily Riehl for their time and accommodation to discuss various parts of this paper during the time 
it was being written.

\bibliographystyle{plain}
\bibliography{BSBib}
\end{document}